\documentclass[12pt]{amsart}
\usepackage[margin=1in]{geometry}
\usepackage{bm}
\usepackage[dvipsnames,svgnames,x11names]{xcolor}
\usepackage[hyphens]{url}
\usepackage[colorlinks=true,
            linkcolor=Maroon,
            citecolor=blue,
            urlcolor=blue,
            hypertexnames=false,
            linktocpage]{hyperref}
\usepackage{bookmark}
\usepackage{amsmath,thmtools,mathtools,amssymb}
\mathtoolsset{showonlyrefs=true}
\usepackage{graphicx}
\usepackage{tikz}
\usepackage{fancyhdr}
\usepackage{esint}
\usepackage{enumerate}
\usepackage{caption}
\allowdisplaybreaks

\newcommand{\R}{\mathbb{R}}
\newcommand{\Sn}{S^{n-1}}
\newcommand{\one}{\mathbf{1}}

\newcommand{\ip}[2]{\langle #1,#2\rangle}
\newcommand{\abs}[1]{\left|#1\right|}
\newcommand{\norm}[1]{\left\|#1\right\|}
\newcommand{\Vol}{\operatorname{Vol}}
\newcommand{\tr}{\operatorname{tr}}

\newcommand{\relint}{\operatorname{relint}}

\theoremstyle{plain}
\newtheorem{theorem}{Theorem}[section]
\newtheorem*{theorem*}{Theorem}

\newtheorem{proposition}[theorem]{Proposition}
\newtheorem{lemma}[theorem]{Lemma}
\newtheorem{corollary}[theorem]{Corollary}

\theoremstyle{definition}
\newtheorem{definition}[theorem]{Definition}

\theoremstyle{remark}
\newtheorem{remark}[theorem]{Remark}

\numberwithin{equation}{section}

\newcommand{\eq}[1]{\begin{equation}\begin{alignedat}{2}#1\end{alignedat}\end{equation}}

\makeatletter
\renewcommand{\subsection}{\@startsection{subsection}{2}%
  \z@{.5\linespacing\@plus.7\linespacing}{.25\linespacing}%
  {\normalfont\bfseries}}
\makeatother

\title[Centro-affine spectral geometry of polytopes]
  {Centro-affine spectral geometry of polytopes}
\author[Y. Hu, M. N. Ivaki]{Yingxiang Hu, Mohammad N. Ivaki}
\begin{document}

\begin{abstract}
In this paper, we develop a polyhedral version of Milman's centro-affine spectral framework for full-dimensional origin-symmetric polytopes. A weighted graph Laplacian on the facets plays the role of the smooth centro-affine Laplacian. For $n\geq 3$, we prove that its first nonconstant even eigenvalue satisfies
\eq{
  \lambda_{1,e}(P)\geq n-1+\frac{1}{n-1}
}
whenever $P$ is a local minimizer of $\lambda_{1,e}$ among origin-symmetric polytopes with the same facet normals, but not necessarily the same normal fan.
\end{abstract}

\maketitle

\tableofcontents

\section{Introduction}
\label{sec:introduction}

Throughout the paper, we assume that $n\geq 2$. The Euclidean unit sphere in $\R^n$ is denoted by $\Sn$, and $\mathcal{H}^k$ denotes $k$-dimensional Hausdorff measure. For $\varphi\in C^2(\Sn)$, write
\eq{
  r_{\varphi}=\bar{\nabla}^2\varphi+\varphi\bar{g}.
}
For $k\geq 2$, a convex body $K\subseteq\R^n$ is called $C^k_+$ if its support function $h_K\in C^k(\Sn)$ and the radius-of-curvature tensor $r_K:=r_{h_K}$ is positive definite. For a $C^2_+$ body $K$ with $h_K>0$, the centro-affine metric and volume measure in the Gauss map parametrization are
\eq{
  g_K&=\frac{r_K}{h_K},\quad
  dV_K&=h_K\det_{\bar{g}}(r_K)\,d\omega.
}
Here $\bar{g}$ denotes the round metric on $\Sn$, $\bar{\nabla}$ its Levi--Civita connection, and $d\omega$ its Riemannian measure. For a function $\psi$, the symbol $\nabla\psi$ denotes its gradient with respect to $g_K$. When $K$ is origin symmetric, the first nonconstant even centro-affine eigenvalue admits the Rayleigh characterization
\eq{\label{eq:introduction-smooth-Rayleigh-characterization}
  \lambda_{1,e}(K)
  =\inf_{\substack{\psi\in C^2(\Sn),\ \psi\not\equiv 0\\
      \psi(-u)=\psi(u),\ \int_{\Sn}\psi\,dV_K=0}}
  \frac{\int_{\Sn}\abs{\nabla\psi}_{g_K}^2\,dV_K}
       {\int_{\Sn}\psi^2\,dV_K}.
}
See Milman's centro-affine spectral framework \cite{Mil25}. In $\mathbb{R}^2$, the sharp estimate $\lambda_{1,e}\geq 2$ follows from \cite{BLYZ12}. Kolesnikov and Milman \cite{KM22} proved that every origin-symmetric $C^2_+$ body $K$ satisfies
\eq{
\lambda_{1,e}(K)\geq n-1+\frac{c_0}{n^{3/2}}.
}
Here, $c_0>0$ is a universal constant. Chen, Huang, Li, and Liu \cite{CHLL20} and Putterman \cite{Put21} establish local-to-global passages for the corresponding $L^p$-Brunn--Minkowski inequality. Klartag \cite{Kla23} bounded the isotropic KLS constant by $O(\sqrt{\log n})$, and Letwin's recent preprint \cite{Let26} lowers the KLS bound to $O((\log n)^{1/4})$. Applying \cite[Thm. 6.4, Cor. 6.8]{KM22} to Letwin's estimate, one obtains $\lambda_{1,e}\geq n-1+c_1/(n\sqrt{\log n})$ for a universal constant $c_1>0$.

By \cite[Thm. 3.1]{IM24} and approximation, an origin-symmetric $C^2_+$ body $K$ satisfies $\lambda_{1,e}(K)\geq n$ under the pinching condition
\eq{
\alpha\mathrm{Id}\leq D^2\left(\frac{1}{2}\norm{x}_{\ell K}^{2}\right)
\leq(n+1)\alpha\mathrm{Id}\quad(x\in\R^n\setminus\{0\}),
}
for some $\ell\in\mathrm{GL}(n,\R)$ and $\alpha>0$. Here, $\norm{\cdot}_{\ell K}$ denotes the Minkowski functional of $\ell K$, and $D^2$ is the Euclidean Hessian.

A convex body in $\R^n$ or a function on $\Sn$ is called unconditional if it is invariant under reflection in every coordinate hyperplane. For an unconditional $C^2_+$ body $K$, the infimum in \eqref{eq:introduction-smooth-Rayleigh-characterization}, restricted to unconditional functions, is at least $n$. This follows from Saroglou's work \cite[Thm. 1.2]{Sar15}; see also Kolesnikov and Milman \cite[Sec. 8.1]{KM22} and \cite[Thm. 8.12]{HIWeighted26}. Approximation and second variation of the log-Minkowski inequality in \cite[Cor. 1.2]{HIUnc26} establish the stronger result $\lambda_{1,e}\geq n$. The same lower bound $n$ holds for every $C^2_+$ origin-symmetric zonoid \cite[Thm. 1.4]{vH23b}.

B\"or\"oczky \cite{Bor23} and Huang, Yang, and Zhang \cite{HYZ25} survey the log-Brunn--Minkowski conjecture, its historical development, and its connections with Minkowski problems.

We also mention that He and Liu \cite[Thm. 5]{HL25} proved that $K\mapsto\lambda_{1,e}(K)$ has no local minimizer among origin-symmetric $C^4_+$ convex bodies. The upper semicontinuity under polyhedral approximation established in \autoref{thm:sec-8-upper-semicontinuity} provides further motivation for investigating the centro-affine spectral gap problem through the lens of polytopes.

For an origin-symmetric polytope, the smooth Rayleigh quotient has a finite-dimensional counterpart on the facets. Let $P\subseteq\R^n$ be full dimensional and origin symmetric. Its facets come in antipodal pairs indexed by a finite signed set $I=I^{+}\sqcup(-I^{+})$ with $0\notin I$. The irredundant representation, containing one inequality for each facet, has the form
\eq{\label{eq:introduction-facet-representation}
  P=P(h)&=\bigcap_{i\in I}\{x\in\R^n:\ip{x}{u_i}\leq h_i\},
 \\
  \abs{u_i}&=1,
  \quad
  h_i>0,
  \quad
  u_{-i}=-u_i,
  \quad
  h_{-i}=h_i.
}
For $i\in I$, denote the corresponding facet and its mass by
\eq{
  F_i=P\cap\{x\in\R^n:\ip{x}{u_i}=h_i\},\quad
  m_i=h_i\mathcal{H}^{n-1}(F_i).
}
Adjacent facets $F_i$ and $F_j$ determine the angle and conductance
\eq{
  \theta_{ij}=\arccos\ip{u_i}{u_j},
  \quad
  c_{ij}=\frac{h_ih_j\mathcal{H}^{n-2}(F_i\cap F_j)}{\sin\theta_{ij}},
}
while $c_{ij}=0$ for nonadjacent facets. For facet functions $f=(f_i)_{i\in I}$ and $g=(g_i)_{i\in I}$, we define the weighted graph Dirichlet form by
\eq{
  \mathsf{L}_P(f,g) = \sum_{\{i,j\}\in\binom{I}{2}}c_{ij}(f_i-f_j)(g_i-g_j),
}
where each unordered pair occurs once. In comparison with the smooth setting, the masses $m_i$ replace the smooth measure $dV_K$, while $\mathsf{L}_P(f,f)$ replaces the smooth Dirichlet energy in the Rayleigh quotient.

We say a facet function $f=(f_i)_{i\in I}$ is even if $f_{-i}=f_i$ for every $i\in I$. For such $f$, we define the weighted mean
\eq{
\bar{f}=\frac{\sum_{i\in I}m_if_i}{\sum_{i\in I}m_i}.
}
The first nonconstant even eigenvalue is
\eq{\label{eq:introduction-Rayleigh-characterization}
\lambda_{1,e}(P)
=\min_{\substack{f_{-i}=f_i\\ f \text{ not constant}}}
\frac{\mathsf{L}_P(f,f)}
{\sum_{i\in I}m_i(f_i-\bar{f})^2}.
}

For a support vector $q$ near $h$, let
\eq{
P(q)=\bigcap_{i\in I}\{x\in\R^n:\ip{x}{u_i}\leq q_i\}.
}

\begin{theorem}\label{thm:main}
Let $n\geq 3$ and $P(h)$ be a full-dimensional origin-symmetric polytope with the complete irredundant facet representation \eqref{eq:introduction-facet-representation}. Suppose that $\lambda_{1,e}(P(q))\geq \lambda_{1,e}(P(h))$ for every even support vector $q$ sufficiently close to $h$ for which every labeled inequality still defines a facet. Then
\eq{
  \lambda_{1,e}(P(h)) \geq n-1+\frac{1}{n-1}.
}
\end{theorem}

For $n\geq 3$, we now briefly outline the proof of \autoref{thm:main}. We write $\odot$ for coordinatewise multiplication, use $\one$ for the constant-one facet function, and interpret powers marked by $\odot$ coordinatewise. Let $f$ be a first nonconstant even eigenfunction normalized by
\eq{
  \sum_{i\in I}m_if_i=0, \quad
  \sum_{i\in I}m_if_i^2=1.
}
For $\lambda=\lambda_{1,e}(P(h))$ and $\beta=n-1-\lambda$, consider the support and trial paths
\eq{
  h_t=h\odot(\one+tf), \quad
  u_t=f-(1+\beta)t f^{\odot2}.
}
The logarithmic velocity of $h_t$ at zero is $f$.

For $P_t=P(h_t)$, write $m_i(t)$ and $c_{ij}(t)$ for its facet masses and ridge conductances, and let
\eq{
  \bar{u}_t
  =\frac{\sum_{i\in I}m_i(t)(u_t)_i}
       {\sum_{i\in I}m_i(t)}.
}
The corresponding centered Rayleigh quotient is
\eq{
  \mathcal{R}_{\mathrm{tr}}(t)
  =\frac{
    \sum_{\{i,j\}\in\binom{I}{2}}
    c_{ij}(t)\bigl((u_t)_i-(u_t)_j\bigr)^2
  }{
    \sum_{i\in I}
    m_i(t)\bigl((u_t)_i-\bar{u}_t\bigr)^2
  }.
}

We first consider the case when a local minimizer $h$ lies in a simple type chamber, where volume is polynomial in the support numbers. Twice differentiating the quotient introduces the fourth directional derivative of the volume polynomial along $h\odot f$. This derivative vanishes when $n=3$. For $n\geq 4$, the degree-two Hodge--Riemann--Minkowski inequality of McMullen \cite[Thm. 8.2]{McM93} and Timorin \cite[Thm. 5.1.1]{Tim99} bounds the fourth directional derivative from below and yields
\eq{
  \mathcal{R}_{\mathrm{tr}}''(0)\leq
  \frac{2\beta\lambda}{(n-1)\sum_{i\in I}m_i}
  \bigl\{1+(n-1)\beta\bigr\}.
}
On the other hand, the Rayleigh principle and local minimality imply
\eq{
  \mathcal{R}_{\mathrm{tr}}(t)
\geq \lambda_{1,e}(P(h_t))
\geq \lambda
  =\mathcal{R}_{\mathrm{tr}}(0)
}
for all sufficiently small $\abs{t}$. Therefore, the second derivative test requires $\mathcal{R}_{\mathrm{tr}}''(0)\geq 0$. By \autoref{cor:sec-5-strict-even-eigenvalue-bound}, $\beta<0$. Hence $\beta\leq-\frac{1}{n-1}$.

When a minimizing support vector $h$ is nonsimple, it may lie in the closures of several simple type chambers. The corresponding volume polynomials have the same value, gradient, and Hessian at $h$, although their third and fourth derivatives may depend on the chamber. To handle this case, we prove a midpoint inequality for the support Hessian by comparing mixed volumes. Together with local minimality, this inequality forces the one-sided Rayleigh derivatives along opposite support directions to vanish. The resulting stationarity equation allows us to apply the second-variation argument in an \emph{incident} simple type chamber and rule out the possibility that
\eq{
  n-1<\lambda_{1,e}(P)<n-1+\frac{1}{n-1}.
}

The Hessian of the volume polynomial connects the spectral problem \eqref{eq:introduction-Rayleigh-characterization} to earlier work on polyhedral geometry. Izmestiev \cite{Izm10} studied the support-parameter space of polytopes with prescribed facet normals and constructed a Colin de Verdi\`ere matrix from the negative volume Hessian. Fillastre and Izmestiev \cite{FI17} described irredundancy domains and the decomposition of their interiors into type cones. Together, the two works provide a blueprint for extending smooth centro-affine geometry to the polytopal setting.

The degree-one Hodge--Riemann relation is precisely the second Minkowski inequality. In a simple type chamber, it is equivalent to concavity of $h\mapsto\Vol(P(h))^{1/n}$ under linear support interpolation. The proof of \autoref{thm:main} uses the degree-two Hodge--Riemann relation. Van Handel \cite{vH23a} formulated the higher-degree Hodge--Riemann relations as mixed-volume inequalities and used the degree-two relation to disprove Fedotov's conjectured extension of Shephard's inequalities. We also mention that Shenfeld and van Handel \cite{SvH19} proved the Alexandrov--Fenchel inequality by the Bochner method, including a finite-dimensional spectral argument for polytopes. See also their later work, where they treat the equality cases of Minkowski's quadratic inequality and the Alexandrov--Fenchel inequality \cite{SvH22,SvH23}.

\begin{remark}\label{rem:introduction-multiplicity-bound}
Under the assumptions of \autoref{thm:main}, if $\lambda_{1,e}(P(h))$ has multiplicity $m$ in the space of even facet functions, then
\eq{
  \lambda_{1,e}(P(h))\geq n-1+\frac{m}{n-1}.
}
Moreover, for any $N\geq n$, every global minimizer $P$ of $\lambda_{1,e}$ in the class of full-dimensional origin-symmetric polytopes in $\mathbb{R}^n$ with at most $N$ pairs of facets, and $m=n-1$, is a parallelotope.

Further details may appear elsewhere.
\end{remark}

\section{Polytopes and the facet network}
\label{sec:polytopes-facet-network}

In this section, we collect the basic terminology for polytopes and complete irredundant facet representations. We then attach to each such representation a labeled facet network whose vertex and edge weights are the facet masses and ridge conductances. For origin-symmetric polytopes, this network defines $\lambda_{1,e}(P)$, the first nonconstant even eigenvalue.

 Coordinate subscripts such as $h_i$ are labels and never denote differentiation. The symbol $\one$ denotes the constant-one vector on any finite label set.

A set $K\subseteq\R^n$ is \emph{convex} if $(1-t)x+ty\in K$ for every $x,y\in K$ and $t\in[0,1]$. A \emph{convex body} is a compact convex set with nonempty interior.

For a nonempty compact set $K\subseteq\R^n$, its support function is
\eq{
  h_K(u)=\max_{x\in K}\ip{x}{u}
  \quad(u\in\R^n).
}

For a nonempty set $A\subseteq\R^n$, its convex hull and affine hull are
\eq{
  \operatorname{conv}A
  &=\left\{\sum_{k=1}^m\lambda_ka_k:
  m\in\{1,2,\ldots\}, \quad a_k\in A, \quad \lambda_k\geq 0, \quad
  \sum_{k=1}^m\lambda_k=1\right\},\\
  \operatorname{aff}A
  &=\left\{\sum_{k=1}^m\lambda_ka_k:
  m\in\{1,2,\ldots\}, \quad a_k\in A, \quad \lambda_k\in\R, \quad
  \sum_{k=1}^m\lambda_k=1\right\}.
}
The dimension of $A$ is the dimension of its affine hull. The relative interior $\relint A$ is the interior of $A$ in $\operatorname{aff}A$. A subset of $\R^n$ is \emph{full dimensional} if its affine hull is $\R^n$.

A polytope is the convex hull of a finite nonempty set. A nonempty subset $F$ of a polytope $P$ is a \emph{face of $P$} if $F=P$ or, for some nonzero linear functional $\ell:\R^n\to\R$,
\eq{
  F=\left\{x\in P:\ell(x)=\max_{y\in P}\ell(y)\right\}.
}
Following Schneider's convention \cite[p. 74]{Sch14}, we regard $\varnothing$ as a face of $P$ and we set $\dim\varnothing=-1$. A face is \emph{proper} if it is neither $\varnothing$ nor $P$. A face of dimension $0$ or $1$ is called a vertex or edge, respectively. If $d=\dim P$, a face of dimension $d-2$ is a ridge and one of dimension $d-1$ is a facet. The \emph{face lattice} of a polytope is its set of faces ordered by inclusion. A $d$-dimensional polytope is \emph{simple} if every vertex belongs to exactly $d$ facets.

The set of vertices of $P$ is denoted by
\eq{
  \operatorname{vert}P=\{v\in P:\{v\}\text{ is a face of }P\}.
}
Every polytope is the convex hull of its vertices \cite[p. 52, Prop. 2.2]{Zie95}.

Points $p_1,\ldots,p_{d+1}$ are \emph{affinely independent} if $p_2-p_1,\ldots,p_{d+1}-p_1$ are linearly independent, and their convex hull is then a \emph{$d$-simplex}. The \emph{standard $d$-simplex} is 
\eq{
\operatorname{conv}\{\mathbf{e}_1,\ldots,\mathbf{e}_{d+1}\}\subseteq\R^{d+1},
}
where $\mathbf{e}_1,\ldots,\mathbf{e}_{d+1}$ are the standard coordinate vectors.

Suppose that $P$ contains the origin in its interior. A finite half-space representation of $P$ indexed by a set $J$ has the form
\eq{
   P=\bigcap_{i\in J}\{x\in\R^n:\ip{x}{u_i}\leq h_i\}, \quad \abs{u_i}=1, \quad h_i>0.
}
The inequality indexed by $i\in J$ is a \emph{facet inequality} if $P\cap\{x\in\R^n:\ip{x}{u_i}=h_i\}$ is a facet of $P$. The representation is \emph{complete and irredundant} if these inequalities are in one-to-one correspondence with the facets of $P$; equivalently, removing any one of them enlarges $P$.

\begin{definition}[Labeled facet network]\label{def:sec-2-labeled-facet-network}
For a complete irredundant representation, the facet $F_i$ and its mass $m_i$ are
\eq{
  F_i=\{x\in P:\ip{x}{u_i}=h_i\},\quad
  m_i=h_i\mathcal{H}^{n-1}(F_i).
}
The vector $h=(h_i)_{i\in J}$ is the \emph{support vector} of this labeled half-space representation.

Two facets are \emph{adjacent} when they meet in a ridge. Adjacent facets $F_i$ and $F_j$ determine
\eq{\label{eq:sec-2-polytope-conductances}
  R_{ij}=F_i\cap F_j, \quad \theta_{ij}=\arccos\ip{u_i}{u_j}, \quad c_{ij}=\frac{h_ih_j\mathcal{H}^{n-2}(R_{ij})}{\sin\theta_{ij}}.
}
The number $c_{ij}$ is the \emph{ridge conductance}, with $c_{ij}=0$ for nonadjacent facets. The \emph{labeled facet network} is the graph whose vertices are the facet labels and whose edges join adjacent facets. Its vertex and edge weights are $m_i$ and $c_{ij}$, respectively.
\end{definition}

\begin{lemma}\label{lem:sec-2-facet-graph-connected}
The facet-adjacency graph of a full-dimensional polytope is connected.
\end{lemma}

\begin{proof}
A translation preserves the facet-adjacency graph. Hence we may assume the origin lies in the interior of $P$. The polar polytope of $P$ is
\eq{
  P^{\ast}=\{y\in\R^n:\ip{x}{y}\leq 1\text{ for every }x\in P\}.
}
Face-lattice duality sends the facets and ridges of $P$ to the vertices and edges of $P^{\ast}$, respectively; see \cite[p. 109, Eq. (2.28)]{Sch14}. Thus the facet-adjacency graph of $P$ is isomorphic to the vertex-edge graph of $P^{\ast}$, which is connected by Balinski's theorem \cite[p. 95, Thm. 3.14]{Zie95}.
\end{proof}

A \emph{facet function} on a labeled facet network with vertex set $J$ is a vector $f=(f_i)_{i\in J}$. For facet functions $f$ and $g$, the weighted mean of $f$ and the weighted graph Dirichlet form are
\eq{
  \bar{f}
  &=\frac{\sum_{i\in J}m_if_i}{\sum_{i\in J}m_i},\\
  \mathsf{L}_P(f,g)&=\sum_{\{i,j\}\in\binom{J}{2}}c_{ij}(f_i-f_j)(g_i-g_j).
}
In the formula for $\mathsf{L}_P(f,g)$, each unordered pair appears once.

Whenever the outer unit facet normals occur in antipodal pairs, we use a signed label set
\eq{
  I=I^{+}\sqcup(-I^{+}),\quad 0\notin I.
}
The set $I^{+}$ contains one label from each pair. For $i\in I^{+}$, the opposite facet has label $-i$ and outer unit normal $u_{-i}=-u_i$.
No relation between $h_i$ and $h_{-i}$ is imposed. Throughout the paper, $I$ \emph{is reserved for a signed label set of this form}. The support vector is \emph{even} when $h_{-i}=h_i$ for every $i\in I$. A facet function is \emph{even} when $f_{-i}=f_i$ for every $i\in I$. For a complete irredundant representation, evenness of $h$ is equivalent to $P=-P$.

The linear involution $\iota:\R^I\to\R^I$ is defined by
\eq{
  (\iota z)_i=z_{-i}.
}
Thus a vector $z\in\R^I$ is even precisely when $\iota z=z$; it is \emph{odd} when $\iota z=-z$.

\begin{definition}[First nonconstant even eigenvalue]
Suppose that $h$ is even, equivalently $P=-P$. The first nonconstant even eigenvalue of $P$ is
\eq{
  \lambda_{1,e}(P)
  =\min_{\substack{f\ \mathrm{even}\\f \text{ not constant}}}
  \frac{\mathsf{L}_P(f,f)}
       {\sum_{i\in I}m_i(f_i-\bar{f})^2}.
}
\end{definition}
Neither replacing $f$ by $f-\bar{f}\one$ nor multiplying it by a nonzero scalar alters the quotient. Therefore, the minimum may be taken over the normalized set
\eq{
  \left\{f\in\R^I:
  f_{-i}=f_i,\quad
  \sum_{i\in I}m_if_i=0,\quad
  \sum_{i\in I}m_if_i^2=1\right\}.
}
The boundedness of $P$ forces the vectors $(u_i)_{i\in I^{+}}$ to span $\R^n$; otherwise a nonzero vector perpendicular to their span would generate a line in $P$. Hence the even facet space has dimension $\abs{I}/2=\abs{I^{+}}\geq n$. The normalized set is nonempty and compact. Moreover, the continuous function $f\mapsto \mathsf{L}_P(f,f)$ attains a minimum on this set. By \autoref{lem:sec-2-facet-graph-connected}, this minimum is positive: $\mathsf{L}_P(f,f)$ vanishes only on constants, which the normalization excludes.

\section{Polyhedral cones, normal fans, and support coordinates}
\label{sec:support-coordinates}

In this section, we recall the basics of polyhedral cones and normal fans and introduce the support coordinates used throughout the paper. Moreover, we describe the simple type chambers on which the normal fan remains fixed and volume is represented by a homogeneous polynomial.

\subsection{Polyhedral cones and fans}

For $z_1,\ldots,z_m\in\R^n$, we define their positive hull by
\eq{
\operatorname{pos}\{z_1,\ldots,z_m\}=\left\{\sum_{k=1}^mt_kz_k:t_k\geq 0\right\},
}
with $\operatorname{pos}(\varnothing)=\{0\}$. A \emph{polyhedral cone} $\mathfrak{C}$ is the positive hull of finitely many vectors. A representation $\mathfrak{C}=\operatorname{pos}\{v_1,\ldots,v_m\}$ need not be minimal: some $v_i$ may be removed without changing $\mathfrak{C}$. The cone $\mathfrak{C}$ is \emph{simplicial} if it has such a representation with $v_1,\ldots,v_m$ linearly independent.

A \emph{face} of a polyhedral cone $\mathfrak{C}$ is either $\mathfrak{C}$ itself or a set
\eq{
  \{z\in\mathfrak{C}:\ell(z)=0\}, \quad \ell\leq 0\text{ on }\mathfrak{C},
}
where $\ell$ is a nonzero linear functional. The cone is \emph{pointed} when $\mathfrak{C}\cap(-\mathfrak{C})=\{0\}$.

For $0\neq v\in\mathfrak{C}$, the ray spanned by $v$ is $\R_{\geq 0}v$, and its nonzero generators are precisely the positive multiples of $v$. An \emph{extreme ray} of $\mathfrak{C}$ is a ray that is a face of $\mathfrak{C}$. Every pointed polyhedral cone is the positive hull of one nonzero generator from each of its extreme rays \cite[p. 17, Thm. 1.4.3]{Sch14}.

\begin{definition}[Polyhedral fan]
A \emph{polyhedral fan} in $\R^n$ is a finite collection $\mathcal{F}$ of polyhedral cones such that every face of a cone in $\mathcal{F}$ belongs to $\mathcal{F}$ and the intersection of two cones in $\mathcal{F}$ is a face of each. The fan is \emph{complete} if
\eq{
  \bigcup_{\zeta\in\mathcal{F}}\zeta=\R^n.
}
The fan is \emph{simplicial} if every cone in $\mathcal{F}$ is simplicial. A cone in $\mathcal{F}$ is \emph{maximal} if it is maximal under inclusion.

A fan $\mathcal{F}'$ \emph{refines} a fan $\mathcal{F}$ if every cone of $\mathcal{F}'$ is contained in a cone of $\mathcal{F}$. Equivalently, $\mathcal{F}$ \emph{coarsens} $\mathcal{F}'$.
\end{definition}

We write $\mathcal{F}$ for an arbitrary polyhedral fan and reserve $\Sigma$ for a complete simplicial fan.

\begin{remark}
Note that every maximal cone $\zeta$ of a complete fan in $\R^n$ is $n$-dimensional. Indeed, suppose that a maximal cone $\zeta$ had smaller dimension. Let $x\in\relint \zeta$ and $z\notin\operatorname{span}(\zeta)$. For each $k\geq 1$, completeness provides a cone $\mathfrak{C}_k$ of the fan such that $x+k^{-1}z\in\mathfrak{C}_k$. Since the fan is finite, there are a cone $\mathfrak{C}$ and integers $k_r\to\infty$ such that
\eq{
x+k_r^{-1}z\in\mathfrak{C}
\quad (r\geq 1).
}
The cone $\mathfrak{C}$ is closed. Hence $x\in\mathfrak{C}$. Since the face $\zeta\cap\mathfrak{C}$ of $\zeta$ contains $x\in\relint \zeta$, it equals $\zeta$; in particular, $\zeta\subseteq\mathfrak{C}$. On the other hand,
\eq{
  x+k_r^{-1}z\in\mathfrak{C}\setminus\operatorname{span}(\zeta)
  \quad(r\geq 1).
}
Therefore, $\zeta\subsetneq\mathfrak{C}$, contrary to the maximality of $\zeta$.
\end{remark}

\subsection{Normal fans and face duality}

\begin{definition}[Normal and polytopal fans]\label{def:sec-3-normal-polytopal-fan}
Let $P\subseteq\R^n$ be a nonempty polytope. For $v\in\R^n$, the exposed face of $P$ in direction $v$ is
\eq{
  P^v=\{x\in P:\ip{x}{v}=h_P(v)\}.
}
For a nonempty face $F$ of $P$, including the improper face $F=P$, its \emph{normal cone} is
\eq{
  N_P(F)
  =\left\{v\in\R^n:\ip{x}{v}=h_P(v)
  \text{ for every }x\in F\right\}.
}
Equivalently,
\eq{
  N_P(F)=\{v\in\R^n:F\subseteq P^v\}
  =\{v\in\R^n:P^v\cap F=F\}.
}
For a vertex $x$ of $P$, we write $N_P(x)=N_P(\{x\})$. The \emph{normal fan} of $P$, $\mathcal{F}_P$, is the collection of its normal cones \cite[p. 193, Ex. 7.3]{Zie95}. A complete fan is \emph{polytopal} if it is the normal fan of a polytope.
\end{definition}

\begin{lemma}\label{lem:sec-3-face-normal-cone-duality}
Let $P\subseteq\R^n$ be a full-dimensional polytope. The assignment $F\mapsto N_P(F)$ is a bijection from the nonempty faces of $P$ onto the cones of $\mathcal{F}_P$. For any nonempty faces $F$ and $G$:
\eq{
  F\subseteq G
  \quad\iff\quad
  N_P(F)\supseteq N_P(G).
}
Moreover, $\dim F+\dim N_P(F)=n$ for every nonempty face $F$. In addition, for every nonempty face $G$ of $P$ and every $v\in\R^n$,
\eq{\label{eq:sec-3-relative-normal-exposed-face}
  v\in\relint N_P(G)
  \quad\iff\quad P^v=G.
}
\end{lemma}

For proper faces, the lemma and \eqref{eq:sec-3-relative-normal-exposed-face} appear in \cite[p. 108, Thm. 2.4.9, Eqs. (2.24)--(2.26)]{Sch14}. The case $F=P$ follows from $N_P(P)=\{0\}$.

\begin{corollary}\label{cor:sec-3-normal-cones-form-fan}
The normal cones of a full-dimensional polytope form a complete polyhedral fan.
\end{corollary}
\begin{proof}
By \cite[p. 108, Thm. 2.4.9]{Sch14}, every normal cone is polyhedral. We first show that every face of a normal cone is again a normal cone. Let $F\subseteq G$ be nonempty faces of $P$, and choose $x\in\relint F$ and $y\in\relint G$. For $v\in N_P(F)$,
\eq{
  \ip{y-x}{v}\leq 0,
}
with equality exactly when $G\subseteq P^v$. Thus
\eq{\label{eq:sec-3-nested-normal-cone-face}
  N_P(G)=N_P(F)\cap\{v\in\R^n:\ip{y-x}{v}=0\},
}
and $N_P(G)$ is a face of $N_P(F)$.

Conversely, let $\mathfrak{C}$ be a face of $N_P(F)$, choose $v\in\relint\mathfrak{C}$, and put $G=P^v$. Since $v\in N_P(F)$, we have $F\subseteq G$. Moreover, \eqref{eq:sec-3-relative-normal-exposed-face} implies
\eq{
  v\in\relint N_P(G).
}
Both $\mathfrak{C}$ and $N_P(G)$ are faces of $N_P(F)$ whose relative interiors contain $v$. Thus $\mathfrak{C}=N_P(G)$.

If $\tilde{F}$ is the smallest face of $P$ containing two nonempty faces $F$ and $G$, then
\eq{
  v\in N_P(F)\cap N_P(G)
  \quad\iff\quad F\cup G\subseteq P^v
  \quad\iff\quad \tilde{F}\subseteq P^v
  \quad\iff\quad v\in N_P(\tilde{F}).
}
Thus $N_P(F)\cap N_P(G)=N_P(\tilde{F})$ is a face of both cones. Finally,
\eq{
  v\in N_P(P^v)\quad(v\in\R^n),
}
and the normal cones cover $\R^n$.
\end{proof}

\begin{remark}
\label{rem:sec-3-simple-polytopes-simplicial-polytopal-fans}
For a fan $\mathcal{F}$ in $\R^n$,
\eq{
  \substack{\mathcal{F}\text{ is complete, simplicial,}\\\text{and polytopal}}
  \quad\iff\quad
  \substack{\mathcal{F}=\mathcal{F}_Q\text{ for some full-dimensional}\\
    \text{simple polytope }Q\subseteq\R^n.}
}
\end{remark}

\subsection{Support coordinates and type chambers}

A finite family of vectors \emph{positively spans} $\R^n$ if its positive hull is $\R^n$. Equivalently, the origin lies in the interior of its convex hull.

\begin{definition}[Prescribed facet normals and support parameters]\label{def:sec-3-prescribed-facet-normals}
Consider a finite label set $J$ and pairwise distinct unit vectors $(u_i)_{i\in J}$ that positively span $\R^n$. For $h=(h_i)_{i\in J}\in\R^J$, the associated half-space intersection is defined by
\eq{\label{eq:sec-3-polytope-support-parameters}
  P(h)=\bigcap_{i\in J}\{x\in\R^n:\ip{x}{u_i}\leq h_i\}.
}
The vector $h$ varies the positions of the supporting hyperplanes while keeping their normals fixed. The half-space intersection $P(h)$ may be empty, and some prescribed inequalities may be redundant.

For $x\in\R^n$, the \emph{slack} of the inequality with label $i$ at $x$ is defined as $h_i-\ip{x}{u_i}$. If $x\in P(h)$, the inequality is \emph{active at $x$} when its slack vanishes and \emph{strict at $x$} when its slack is positive.
\end{definition}

Whenever antipodal indexing is used, the prescribed normals carry the signed label set $I$ from \autoref{sec:polytopes-facet-network}. In this case, the support space is $\R^I$, and for $h\in\R^I$ we define
\eq{\label{eq:sec-3-antipodal-support-parameters}
  P(h)=\bigcap_{i\in I}\{x\in\R^n:\ip{x}{u_i}\leq h_i\}.
}
Antipodal indexing concerns only the normal data. 

If $h\in\R^I$ is even and $P(h)\neq\varnothing$, then $P(h)=-P(h)$. If the representation is complete and irredundant, the converse also holds.

For two vectors $a$ and $b$ indexed by the same finite set, write
\eq{
  (a\odot b)_i=a_ib_i,
  \quad
  a^{\odot k}=\underbrace{a\odot\cdots\odot a}_{k\text{ factors}}
  \quad(k\geq 1).
}
Thus $\odot$ denotes coordinatewise multiplication.

The entries of $h$ are called the support parameters for the fixed normals $u_i$. If $P(h)$ is nonempty and every inequality is irredundant, then $h_i=h_{P(h)}(u_i)$. Hence the entries $h_i$ are precisely the support numbers of $P(h)$ in the prescribed directions. A \emph{support derivative} or the \emph{support Hessian} refers to differentiation of volume with respect to $h$. A logarithmic support direction $a$ corresponds to the ordinary support direction $h\odot a$.

\begin{definition}[Simple type chamber]
For a complete simplicial polytopal fan $\Sigma$ with the prescribed rays, its \emph{simple type chamber} is
\eq{
  \mathcal{T}_{\Sigma}
  =\{h\in\R^J:\text{every inequality defining $P(h)$ is irredundant and }
    \mathcal{F}_{P(h)}=\Sigma\}.
}
It is also called the \emph{type cone} of $\Sigma$. The polytopes represented by $\mathcal{T}_{\Sigma}$ form what Schneider calls an \emph{$a$-type} \cite[p. 109]{Sch14}.
\end{definition}

For each maximal cone $\zeta$ of $\Sigma$, let
\eq{
  J_{\zeta}=\{i_1,\ldots,i_n\}
}
label its extreme rays. The vectors $u_{i_1},\ldots,u_{i_n}$ form a basis of $\R^n$. The linear map
\eq{
  U_{\zeta}x=\bigl(\ip{x}{u_{i_1}},\ldots,\ip{x}{u_{i_n}}\bigr)^{\mathsf{T}}
}
has the standard-basis matrix
\eq{
  U_{\zeta}
  =\begin{pmatrix}
    \ip{\mathbf{e}_1}{u_{i_1}}&\cdots&\ip{\mathbf{e}_n}{u_{i_1}}\\
    \vdots&&\vdots\\
    \ip{\mathbf{e}_1}{u_{i_n}}&\cdots&\ip{\mathbf{e}_n}{u_{i_n}}
  \end{pmatrix}
  \in\R^{n\times n}
}
and is invertible. The hyperplanes $\{x\in\R^n:\ip{x}{u_i}=h_i\}$, $i\in J_{\zeta}$, meet at
\eq{\label{eq:sec-3-maximal-cone-vertex}
  x_{\zeta}(h)=U_{\zeta}^{-1}(h_{i_1},\ldots,h_{i_n})^{\mathsf{T}},
}
which depends linearly on $h$. For each maximal cone $\zeta$, the point $x_{\zeta}(h)$ lies strictly inside every prescribed half-space whose label is not in $J_{\zeta}$ precisely when
\eq{\label{eq:sec-3-type-chamber-inequalities}
  \ip{x_{\zeta}(h)}{u_j}<h_j
  \quad(\zeta\text{ maximal},\ j\notin J_{\zeta}).
}

\begin{lemma}\label{lem:sec-3-type-chamber-inequalities}
A vector $h\in\R^J$ belongs to $\mathcal{T}_{\Sigma}$ if and only if \eqref{eq:sec-3-type-chamber-inequalities} holds. Consequently, $\mathcal{T}_{\Sigma}$ is a nonempty open convex cone whose closure is a polyhedral cone.
\end{lemma}

\begin{proof}
Assume first that $h\in\mathcal{T}_{\Sigma}$. By definition, $\mathcal{F}_{P(h)}=\Sigma$, and every prescribed inequality defines a facet of $P(h)$. By \autoref{lem:sec-3-face-normal-cone-duality}, each maximal cone $\zeta$ of $\Sigma$ is the normal cone of a unique vertex of $P(h)$. The facets containing this vertex are precisely those whose outer unit normals have labels in $J_{\zeta}$. Their supporting hyperplanes meet at $x_{\zeta}(h)$. Hence $x_{\zeta}(h)$ is the vertex whose normal cone is $\zeta$. Moreover, for every $j\notin J_{\zeta}$, the inequality $\ip{x}{u_j}\leq h_j$ in the defining representation of $P(h)$ is strict at $x_{\zeta}(h)$, namely $\ip{x_{\zeta}(h)}{u_j}<h_j$.

Conversely, assume that \eqref{eq:sec-3-type-chamber-inequalities} holds. We first show that $\mathcal{F}_{P(h)}$ and $\Sigma$ have the same maximal cones.  For every maximal cone $\zeta$ of $\Sigma$, \eqref{eq:sec-3-maximal-cone-vertex} implies $\ip{x_{\zeta}(h)}{u_i}=h_i$ for $i\in J_{\zeta}$, whereas \eqref{eq:sec-3-type-chamber-inequalities} states $\ip{x_{\zeta}(h)}{u_j}<h_j$ for $j\notin J_{\zeta}$. It follows that $x_{\zeta}(h)$ belongs to the half-space intersection $P(h)$ defined in \eqref{eq:sec-3-polytope-support-parameters}. Precisely the prescribed inequalities with labels in $J_{\zeta}$ are active at this point. Let $\mathfrak{C}_{\zeta}$ consist of the directions $z$ such that $x_{\zeta}(h)+\varepsilon z\in P(h)$ for all sufficiently small $\varepsilon>0$. Since every prescribed inequality with a label outside $J_{\zeta}$ is strict at $x_{\zeta}(h)$,
\eq{
  \mathfrak{C}_{\zeta}
  =\left\{z\in\R^n:\ip{z}{u_i}\leq 0
  \quad(i\in J_{\zeta})\right\}.
}
The image of $\mathfrak{C}_{\zeta}$ under the invertible linear map $U_{\zeta}$ is the negative orthant. Hence $\mathfrak{C}_{\zeta}$ is full dimensional and pointed. The full dimensionality of $\mathfrak{C}_{\zeta}$ implies that $P(h)$ is full dimensional, while its pointedness implies that $x_{\zeta}(h)$ is a vertex.

We next determine the normal cone at this vertex. If $v\in N_{P(h)}(x_{\zeta}(h))$ and $z\in\mathfrak{C}_{\zeta}$, then $x_{\zeta}(h)+\varepsilon z\in P(h)$ for every sufficiently small $\varepsilon>0$. By the definition of the normal cone,
\eq{
  0
  \geq\ip{v}{x_{\zeta}(h)+\varepsilon z}
    -\ip{v}{x_{\zeta}(h)}
  =\varepsilon\ip{v}{z}.
}
Hence $\ip{v}{z}\leq 0$. Conversely, suppose that $\ip{v}{z}\leq 0$ for every $z\in\mathfrak{C}_{\zeta}$. For any $y\in P(h)$, convexity implies
\eq{
  x_{\zeta}(h)+\varepsilon(y-x_{\zeta}(h))\in P(h)
  \quad(0\leq\varepsilon\leq 1).
}
Hence $y-x_{\zeta}(h)\in\mathfrak{C}_{\zeta}$ and $\ip{v}{y-x_{\zeta}(h)}\leq 0$. Since $y$ was arbitrary, $v\in N_{P(h)}(x_{\zeta}(h))$. Therefore,
\eq{
  N_{P(h)}(x_{\zeta}(h))
  =\{v\in\R^n:\ip{v}{z}\leq 0
  \quad\text{for every }z\in\mathfrak{C}_{\zeta}\}.
}
Since the vectors $(u_i)_{i\in J_{\zeta}}$ form a basis of $\R^n$, every $v\in\R^n$ has a unique representation $v=\sum_{i\in J_{\zeta}}a_iu_i$. For $z\in\mathfrak{C}_{\zeta}$,
\eq{
  \ip{v}{z}
  =\sum_{i\in J_{\zeta}}a_i\ip{u_i}{z},
  \quad
  \ip{u_i}{z}\leq 0
  \quad(i\in J_{\zeta}).
}
If every $a_i$ is nonnegative, then $\ip{v}{z}\leq 0$ for every $z\in\mathfrak{C}_{\zeta}$. Conversely, suppose that $a_k<0$ for some $k\in J_{\zeta}$. The equality $U_{\zeta}(\mathfrak{C}_{\zeta})=(-\infty,0]^n$ provides a vector $z\in\mathfrak{C}_{\zeta}$ such that
\eq{
  \ip{u_k}{z}=-1,
  \quad
  \ip{u_i}{z}=0
  \quad(i\in J_{\zeta}\setminus\{k\}).
}
Then $\ip{v}{z}=-a_k>0$. Therefore, $\ip{v}{z}\leq 0$ for every $z\in\mathfrak{C}_{\zeta}$ if and only if all the coefficients $a_i$ are nonnegative. That is,
\eq{
  N_{P(h)}(x_{\zeta}(h))
  =\operatorname{pos}\{u_i:i\in J_{\zeta}\}
  =\zeta.
}
Therefore, every maximal cone $\zeta$ of $\Sigma$ is a maximal cone of $\mathcal{F}_{P(h)}$.

For the reverse inclusion, we prove that every vertex of $P(h)$ equals $x_{\zeta}(h)$ for some maximal cone $\zeta$ of $\Sigma$. Let $y$ be a vertex of $P(h)$, and take $v\in\relint N_{P(h)}(y)$. By completeness of $\Sigma$, $v$ belongs to a maximal cone $\zeta$. Since $\zeta=N_{P(h)}(x_{\zeta}(h))$, the point $x_{\zeta}(h)$ belongs to $P(h)^v$. In view of \eqref{eq:sec-3-relative-normal-exposed-face}, $P(h)^v=\{y\}$. Hence $y=x_{\zeta}(h)$.

By \autoref{cor:sec-3-normal-cones-form-fan}, $\mathcal{F}_{P(h)}$ is a polyhedral fan, as is $\Sigma$ by assumption. Every cone in either fan is a face of one of its maximal cones. It follows that the two fans contain the same cones:
\eq{
  \mathcal{F}_{P(h)}=\Sigma.
}

Next, we verify that every prescribed inequality in \eqref{eq:sec-3-polytope-support-parameters} is irredundant; equivalently, each must determine a facet of $P(h)$. Given $i\in J$, completeness of $\Sigma$ provides a maximal cone $\zeta$ containing the ray $\R_{\geq 0}u_i$. Since $i\in J_{\zeta}$,
\eq{
  \ip{x_{\zeta}(h)}{u_i}=h_i.
}
Every $x\in P(h)$ satisfies $\ip{x}{u_i}\leq h_i$, while $x_{\zeta}(h)\in P(h)$ satisfies equality. Hence
\eq{
  h_{P(h)}(u_i)=h_i.
}
Since $\mathcal{F}_{P(h)}=\Sigma$ and $u_i\in\relint(\R_{\geq 0}u_i)$, \eqref{eq:sec-3-relative-normal-exposed-face} shows that $P(h)^{u_i}$ is the facet corresponding to this ray. Accordingly, every prescribed inequality determines a facet of $P(h)$, and the equality $\mathcal{F}_{P(h)}=\Sigma$ then implies $h\in\mathcal{T}_{\Sigma}$.

We finally describe the chamber $\mathcal{T}_{\Sigma}$. For every maximal cone $\zeta$ and every $j\notin J_{\zeta}$, consider the linear functional
\eq{
  \ell_{\zeta,j}(q)
  =q_j-\ip{x_{\zeta}(q)}{u_j}.
}
By \eqref{eq:sec-3-maximal-cone-vertex}, $x_{\zeta}(q)$ is independent of $q_j$. Hence $\partial_{q_j}\ell_{\zeta,j}=1$, and $\ell_{\zeta,j}$ is nonzero. The two implications above show that
\eq{
  \mathcal{T}_{\Sigma}
  =\left\{q\in\R^J:\ell_{\zeta,j}(q)>0
  \quad(\zeta\text{ a maximal cone of }\Sigma,\ j\notin J_{\zeta})\right\}.
}
These linear functionals form a finite family. Hence $\mathcal{T}_{\Sigma}$ is an open convex cone. It is nonempty: since $\Sigma$ is polytopal, there is a polytope $Q$ with $\mathcal{F}_Q=\Sigma$. The cone $N_Q(Q)$ is a linear subspace and belongs to the simplicial fan $\Sigma$. Thus $N_Q(Q)=\{0\}$, and $Q$ is full dimensional. The maximal normal cones are simplicial, and face--normal-cone duality implies that $Q$ is simple. Since the rays of $\Sigma$ are generated by the prescribed unit normals,
\eq{
  Q=P\bigl((h_Q(u_i))_{i\in J}\bigr),
  \quad \bigl(h_Q(u_i)\bigr)_{i\in J}\in\mathcal{T}_{\Sigma}.
}

Replacing the strict inequalities by non-strict ones, we obtain
\eq{
  \overline{\mathcal{T}_{\Sigma}}
  =\left\{q\in\R^J:\ell_{\zeta,j}(q)\geq 0
  \quad(\zeta\text{ a maximal cone of }\Sigma,\ j\notin J_{\zeta})\right\},
}
which is a polyhedral cone.
\end{proof}

For $v\in\R^n$, define the translation direction $\mathfrak{t}(v)\in\R^J$ by
\eq{\label{eq:sec-3-translation-direction}
  \mathfrak{t}(v)_i=\ip{v}{u_i}.
}
From the half-space representation \eqref{eq:sec-3-polytope-support-parameters}, it follows that
\eq{
  P(h+\mathfrak{t}(v))=P(h)+v.
}
Hence translation by $v$ preserves volume, irredundancy, and the normal fan.

\begin{lemma}\label{lem:sec-3-type-chamber-Minkowski-addition}
For $h,q\in\overline{\mathcal{T}_{\Sigma}}$ and $s,t\geq 0$ with $s+t>0$,
\eq{
  P(sh+tq)=sP(h)+tP(q).
}
\end{lemma}

\begin{proof}
By \autoref{lem:sec-3-type-chamber-inequalities}, $sh+tq\in\overline{\mathcal{T}_{\Sigma}}$. Let $r\in\overline{\mathcal{T}_{\Sigma}}$, and let $\zeta$ be a maximal cone of $\Sigma$. Equation \eqref{eq:sec-3-maximal-cone-vertex} and \autoref{lem:sec-3-type-chamber-inequalities} imply
\eq{
  \ip{x_{\zeta}(r)}{u_i}&=r_i
  \quad(i\in J_{\zeta}),\\
  \ip{x_{\zeta}(r)}{u_j}&\leq r_j
  \quad(j\notin J_{\zeta}).
}
Hence $x_{\zeta}(r)\in P(r)$. Now note that every $u\in\zeta$ can be written as
\eq{
  u=\sum_{i\in J_{\zeta}}\alpha_i u_i,
  \quad
  \alpha_i\geq 0,
}
and every $x\in P(r)$ satisfies
\eq{
  \ip{x}{u}
  \leq\sum_{i\in J_{\zeta}}\alpha_i r_i
  =\ip{x_{\zeta}(r)}{u}.
}
Therefore, $ h_{P(r)}(u)=\ip{x_{\zeta}(r)}{u}$. Moreover, the map $r\mapsto x_{\zeta}(r)$ is linear. Applying the preceding formula to $h$, $q$, and $sh+tq$, we obtain, for every $u\in\zeta$,
\eq{
  h_{P(sh+tq)}(u)
  &=\ip{x_{\zeta}(sh+tq)}{u}\\
  &=s h_{P(h)}(u)+t h_{P(q)}(u)\\
  &=h_{sP(h)+tP(q)}(u).
}
Since the maximal cones of $\Sigma$ cover $\R^n$, the two support functions agree everywhere, and the polytopes are equal.
\end{proof}

\subsection{The chamber volume polynomial}

\begin{proposition}[The chamber volume polynomial]\label{prop:sec-3-volume-polynomial}
There is a unique homogeneous polynomial $V_{\Sigma}$ of degree $n$ on $\R^J$ satisfying
\eq{
  V_{\Sigma}(h)=\Vol(P(h))
  \quad(h\in\mathcal{T}_{\Sigma}).
}
\end{proposition}

\begin{proof}
Existence follows from the inductive construction in \cite[Thm. 2.1.1]{Tim99}; see also \cite[Sec. 14]{McM93}. Since $\mathcal{T}_{\Sigma}$ is nonempty and open, two polynomials that agree with volume on $\mathcal{T}_{\Sigma}$ agree on all of $\R^J$.
\end{proof}

For a polynomial $\varphi$ on $\R^J$, an integer $k\geq 1$, and $h,z_1,\ldots,z_k\in\R^J$, write
\eq{
  D^k\varphi(h)[z_1,\ldots,z_k]
  =\left.
    \frac{\partial^k}{\partial t_1\cdots\partial t_k}
   \right|_{t_1=\cdots=t_k=0}
  \varphi(h+t_1z_1+\cdots+t_kz_k).
}
\begin{lemma}\label{lem:sec-3-antipodal-volume-parity}
For the antipodally indexed normals, suppose that $\Sigma$ is invariant under the antipodal map. Then
\eq{\label{eq:sec-3-antipodal-volume-polynomial}
  V_{\Sigma}(\iota z)=V_{\Sigma}(z) \quad \text{for every }z\in\R^I.
}
At every even support vector, a mixed derivative in even and odd directions vanishes whenever an odd number of its directions are odd.
\end{lemma}

\begin{proof}
The antipodal invariance of $\Sigma$ implies that $\iota\mathcal{T}_{\Sigma}=\mathcal{T}_{\Sigma}$. For $z\in\mathcal{T}_{\Sigma}$ and $x\in\R^n$,
\eq{
  x\in P(\iota z)
  &\iff
  \ip{x}{u_i}\leq z_{-i}\quad(i\in I)\\
  &\iff
  \ip{-x}{u_{-i}}\leq z_{-i}\quad(i\in I)\\
  &\iff -x\in P(z).
}
Thus the defining inequalities imply
\eq{
  P(\iota z)=-P(z).
}
Taking volumes on both sides, we find
\eq{
  V_{\Sigma}(\iota z)=\Vol(P(\iota z))=\Vol(-P(z))=V_{\Sigma}(z).
}
The two polynomials agree on the open set $\mathcal{T}_{\Sigma}$. Therefore, they agree on $\R^I$.

Let $h\in\R^I$ be even, and suppose that each of $z_1,\ldots,z_k\in\R^I$ is either even or odd. Differentiating \eqref{eq:sec-3-antipodal-volume-polynomial} at $h$ in the directions $z_1,\ldots,z_k$, we obtain
\eq{
  D^kV_{\Sigma}(h)[z_1,\ldots,z_k]
  =D^kV_{\Sigma}(h)[\iota z_1,\ldots,\iota z_k].
}
If exactly $r$ directions are odd, multilinearity implies
\eq{
  D^kV_{\Sigma}(h)[z_1,\ldots,z_k]
  =(-1)^rD^kV_{\Sigma}(h)[z_1,\ldots,z_k].
}
Therefore, when $r$ is odd, the derivative vanishes.
\end{proof}

\section{Logarithmic support calculus and the facet network}
\label{sec:facet-network}

Throughout this section, $\Sigma$ denotes a complete simplicial polytopal fan whose rays are $\R_{\geq 0}u_i$, $i\in J$, for the prescribed normals from \autoref{def:sec-3-prescribed-facet-normals}. Its simple type chamber and chamber volume polynomial are denoted by $\mathcal{T}_{\Sigma}\subseteq\R^J$ and $V_{\Sigma}$. We express the first two logarithmic support derivatives of $V_{\Sigma}$ in terms of the facet masses and ridge conductances and derive the weighted graph Dirichlet form. When the normals are antipodally indexed, $\Sigma$ is antipodally invariant, and $h\in\mathcal{T}_{\Sigma}$ is even, we prove that $\lambda_{1,e}(P(h))\geq n$ if and only if the logarithmic volume Hessian is nonpositive on even facet functions.

\subsection{Logarithmic support forms}
\label{sec:logarithmic-support-forms}

For $h\in\mathcal{T}_{\Sigma}$ with $0\in\operatorname{int}P(h)$, write $P=P(h)$ and $V=V_{\Sigma}$. Then $h_i>0$ for every $i\in J$, and a facet function $a$ induces the logarithmic support direction $h\odot a$. All derivatives of $V$ below are evaluated at $h$ and regarded as symmetric multilinear forms. The four logarithmic support forms are
\eq{\label{eq:sec-4-variation-forms}
  A(a)&=DV(h)[h\odot a], \\ B(a,b)&=D^2V(h)[h\odot a,h\odot b], \\ C(a,b,c)&=D^3V(h)[h\odot a,h\odot b,h\odot c], \\ D(a,b,c,d)&=D^4V(h)[h\odot a,h\odot b,h\odot c,h\odot d].
}
The last form vanishes when $n<4$.

We also write
\eq{
   S=A(\one).
}
From the homogeneity of $V$, we obtain the Euler identities
\eq{\label{eq:sec-4-euler-relations}
  S&=nV(h), \\ B(\one,a)&=(n-1)A(a), \\ C(\one,a,b)&=(n-2)B(a,b), \\ D(\one,a,b,c)&=(n-3)C(a,b,c).
}

The mass form, its centered version, and the energy form are defined by
\eq{\label{eq:sec-4-mass-centered-energy-forms}
  \mathsf{M}(a,b)&=A(a\odot b), \\ \widetilde{\mathsf{M}}(a,b)&=\mathsf{M}(a,b)-\frac{A(a)A(b)}{S}, \\ \mathsf{L}(a,b)&=(n-1)A(a\odot b)-B(a,b).
}
The centered mass form is unchanged if a constant is added to either argument:
\eq{\label{eq:sec-4-centered-mass-constant-invariance}
  \widetilde{\mathsf{M}}(a+c\one,b+d\one)
  =\widetilde{\mathsf{M}}(a,b)
  \quad(a,b\in\R^J,\ c,d\in\R).
}

A facet function $a$ is called \emph{centered} if
\eq{
  A(a)=0.
}
For an arbitrary facet function $a$, its centered version is defined by
\eq{\label{eq:sec-4-centered-representative}
  a^{\circ}=a-\frac{A(a)}{S}\one.
}

\begin{proposition}[The logarithmic volume Hessian]\label{prop:sec-4-logarithmic-volume-Hessian}
For $z\in\R^J$ sufficiently close to the origin, let
\eq{
  \Xi_h(z)=\log V\bigl(h\odot e^z\bigr), \quad (e^z)_i=e^{z_i}.
}
Then for facet functions $a$ and $b$,
\eq{
   D^2\Xi_h(0)[a,b]=\frac{n\widetilde{\mathsf{M}}(a,b)-\mathsf{L}(a,b)}{V(h)}.
}
\end{proposition}

\begin{proof}
Let $Q(s,t)=h\odot e^{sa+tb}$. At $(0,0)$,
\eq{
  \partial_sQ=h\odot a, \quad \partial_tQ=h\odot b, \quad \partial_s\partial_tQ=h\odot(a\odot b).
}
Moreover, by the chain rule,
\eq{
  \left.\partial_sV(Q(s,t))\right|_{(0,0)}&=DV(h)[h\odot a]=A(a), \\
  \left.\partial_tV(Q(s,t))\right|_{(0,0)}&=DV(h)[h\odot b]=A(b), \\
  \left.\partial_s\partial_tV(Q(s,t))\right|_{(0,0)}&=D^2V(h)[h\odot a,h\odot b]+DV(h)[h\odot(a\odot b)] \\
  &=B(a,b)+A(a\odot b).
}
Therefore, the Hessian of $\Xi_h$ at the origin is
\eq{
  D^2\Xi_h(0)[a,b]=\frac{B(a,b)+A(a\odot b)}{V(h)} - \frac{A(a)A(b)}{V(h)^2}.
}
Now, using $S=nV(h)$ and $\mathsf{L}=(n-1)\mathsf{M}-B$,
\eq{
  B(a,b)+A(a\odot b)&=nA(a\odot b)-\mathsf{L}(a,b), \\
  \frac{B(a,b)+A(a\odot b)}{V(h)}-\frac{A(a)A(b)}{V(h)^2}&=\frac{n}{V(h)}\left(A(a\odot b)-\frac{A(a)A(b)}{S}\right)-\frac{\mathsf{L}(a,b)}{V(h)} \\
  &=\frac{n\widetilde{\mathsf{M}}(a,b)-\mathsf{L}(a,b)}{V(h)}.
}
\end{proof}

\subsection{Support derivatives and the weighted graph}

The next lemma relates the support derivatives of $V_{\Sigma}$ to the volumes of facets and ridges. Multiplication by the corresponding logarithmic factors $h_i$ and $h_ih_j$ then produces the masses and conductances of the labeled facet network from \autoref{sec:polytopes-facet-network}.

In the formulas below, $q=(q_i)_{i\in J}$ denotes the variable in the support space $\R^J$, and every derivative is evaluated at $q=h$. For the proof, see, for example, \cite[Eqs. (3.22)--(3.25)]{Fil92}.

\begin{lemma}\label{lem:sec-4-volume-support-derivatives}
Let $h\in\mathcal{T}_{\Sigma}$. Then
\eq{
  \frac{\partial V}{\partial q_i}(h)=\mathcal{H}^{n-1}(F_i).
}
For $i\neq j$,
\eq{
  \frac{\partial^2V}{\partial q_i\partial q_j}(h)
  =\begin{cases}
 \frac{\mathcal{H}^{n-2}(F_i\cap F_j)}{\sin\theta_{ij}}, & \dim(F_i\cap F_j)=n-2, \\[1.1em]
  0, & \text{otherwise}
  \end{cases}.
}
Here $\theta_{ij}$ is the angle between the outward unit normals $u_i$ and $u_j$.
\end{lemma}

In view of the definition of $A$ and the first identity in \autoref{lem:sec-4-volume-support-derivatives},
\eq{
  A(a)=\sum_{i\in J}m_ia_i,
  \quad
  m_i=h_i\mathcal{H}^{n-1}(F_i).
}
Consequently, for all facet functions $a$ and $b$,
\eq{
  \mathsf{M}(a,b)=\sum_{i\in J}m_ia_ib_i.
}
Recall that $\mathbf{e}_i\in\R^J$ denotes the $i$th coordinate vector. For $B_{ij}=B(\mathbf{e}_i,\mathbf{e}_j)$, the mixed derivative formula yields $B_{ij}=c_{ij}$ whenever $i\neq j$. Since $\one=\sum_{j\in J}\mathbf{e}_j$, evaluating the Euler identity $B(\one,a)=(n-1)A(a)$ at $a=\mathbf{e}_i$, we obtain
\eq{
  B_{ii}+\sum_{j\neq i}c_{ij}=(n-1)m_i\implies B_{ii}=(n-1)m_i-\sum_{j\neq i}c_{ij}.
}

\begin{proposition}\label{prop:sec-4-graph-laplacian}
For all facet functions $a$ and $b$,
\eq{
   \mathsf{L}(a,b)=\sum_{\{i,j\}\in\binom{J}{2}} c_{ij}(a_i-a_j)(b_i-b_j).
}
In particular, $\mathsf{L}$ is positive semidefinite and its kernel consists of the constant facet functions.
\end{proposition}

\begin{proof}
For $i\neq j$, the $(i,j)$ entry of the mass form $\mathsf{M}$ vanishes because $\mathbf{e}_i\odot\mathbf{e}_j=0$, whereas the corresponding entry of $B$ is $c_{ij}$. Thus the relation $\mathsf{L}=(n-1)\mathsf{M}-B$ yields $\mathsf{L}_{ij}=-c_{ij}$. By Euler identities, $\mathsf{L}(\one,b)=0$. Hence $\mathsf{L}_{ii}=\sum_{j\neq i}c_{ij}$. Using $c_{ij}=c_{ji}$, the associated bilinear form becomes
\eq{
  a^{\mathsf{T}}\mathsf{L}b&=\sum_{i\in J}a_ib_i\sum_{j\neq i}c_{ij}-\sum_{i\neq j}c_{ij}a_ib_j \\
  &=\frac{1}{2}\sum_{i\neq j}c_{ij}\bigl(a_ib_i+a_jb_j-a_ib_j-a_jb_i\bigr) \\
  &=\sum_{\{i,j\}\in\binom{J}{2}}c_{ij}(a_i-a_j)(b_i-b_j).
}
In particular,
\eq{
  \mathsf{L}(a,a)&=\sum_{\{i,j\}\in\binom{J}{2}}c_{ij}(a_i-a_j)^2\geq 0,\\
  \mathsf{L}(a,a)&=0
  \quad\iff\quad
  a_i=a_j\quad\text{for every edge }\{i,j\}.
}
Applying \autoref{lem:sec-2-facet-graph-connected}, we conclude that $\ker\mathsf{L}=\R\one$.
\end{proof}

Note that $\mathsf{L}$ is precisely the weighted graph Dirichlet form $\mathsf{L}_P$ defined in \autoref{sec:polytopes-facet-network}.

\begin{corollary}\label{cor:sec-4-logarithmic-volume-Hessian-spectral-gap}
For the antipodally indexed normals, suppose that $\Sigma$ is antipodally invariant and $h\in\mathcal{T}_{\Sigma}$ is even. Then
\eq{
  \left.\frac{d^2}{dt^2}\right|_{t=0}\log V_{\Sigma}(h\odot e^{ta})\leq 0 \quad\text{for every even facet function }a
}
if and only if $\lambda_{1,e}(P(h))\geq n$.
\end{corollary}

\begin{proof}
By \autoref{prop:sec-4-logarithmic-volume-Hessian}, the first condition is equivalent to
\eq{\label{eq:sec-4-even-Hessian-form-bound}
  \mathsf{L}(a,a)\geq n\widetilde{\mathsf{M}}(a,a) \quad\text{for every even }a.
}
Let $a^{\circ}$ be the centered version of $a$ from \eqref{eq:sec-4-centered-representative}. By $\mathsf{L}(\one,b)=0$ and \eqref{eq:sec-4-centered-mass-constant-invariance},
\eq{
  \mathsf{L}(a,a)=\mathsf{L}(a^{\circ},a^{\circ}),
  \quad
  \widetilde{\mathsf{M}}(a,a)=\widetilde{\mathsf{M}}(a^{\circ},a^{\circ})=A((a^{\circ})^{\odot2}).
}
Thus \eqref{eq:sec-4-even-Hessian-form-bound} is equivalent to
\eq{
  \mathsf{L}(b,b)\geq nA(b^{\odot2})
  \quad\text{for every centered even facet function }b.
}

Now \eqref{eq:introduction-Rayleigh-characterization} becomes
\eq{
  \lambda_{1,e}(P(h))
  =\min_{\substack{0\neq b\ \text{even}\\A(b)=0}}
  \frac{\mathsf{L}(b,b)}{A(b^{\odot2})}.
}
Therefore, \eqref{eq:sec-4-even-Hessian-form-bound} holds if and only if this minimum is at least $n$.
\end{proof}

\section{The support Hessian, its kernel, and the spectral baseline}
\label{sec:support-hessian}

In this section, we determine the kernel of the support Hessian and prove that it has exactly one positive eigenvalue. We begin by characterizing equality in the second Minkowski inequality for two full-dimensional polytopes with the same prescribed facet normals; see \autoref{def:sec-3-prescribed-facet-normals}. Every prescribed inequality defines a facet of both polytopes, but their normal fans need not agree. We then use the equality case of the second Minkowski inequality to determine the kernel and inertia of the support Hessian on a simple type chamber. When the normals are antipodally indexed, $\Sigma$ is antipodally invariant, and $h\in\mathcal{T}_{\Sigma}$ is even, the support Hessian has no nonzero even vector in its kernel. Consequently, $\lambda_{1,e}(P(h))>n-1$. We end the section by applying the spectral theorem to obtain an $\mathsf{M}$-orthonormal eigenbasis of the centered even facet space that diagonalizes $\mathsf{L}$ and $B$.

For convex bodies $P$ and $Q$, define the mixed volumes $\mathcal{V}_k(P,Q)$, $0\leq k\leq n$, by
\eq{\label{eq:sec-5-mixed-volume-expansion}
  \Vol(sP+tQ)=\sum_{k=0}^n\binom{n}{k}s^{n-k}t^k\mathcal{V}_k(P,Q)
  \quad(s,t\geq 0).
}
We abbreviate $\mathcal{V}_k(P,Q)$ by $\mathcal{V}_k$ and also write it as $\mathcal{V}(P[n-k],Q[k])$, where $P[r]$ denotes $r$ copies of $P$.

The second Minkowski inequality states that
\eq{\label{eq:sec-5-second-Minkowski}
  \mathcal{V}(P[n-1],Q)^2\geq \Vol(P)\mathcal{V}(P[n-2],Q[2]).
}
See, for example, \cite[p. 382, Eq. (7.19)]{Sch14}.

\begin{definition}\label{def:sec-5-polytopal-tangential-body}
Let $K\subseteq L$ be full-dimensional polytopes in $\R^n$. We say that $L$ is an \emph{$(n-2)$-tangential body of $K$} if
\eq{
  K\cap F\neq\varnothing
  \quad\text{for every face $F$ of $L$ satisfying}
  \quad n-2\leq \dim F\leq n-1.
}
That is, $K$ meets every facet and every ridge of $L$.
\end{definition}

Let us compare this with the terminology used in \cite[pp. 85--86]{Sch14}. For $0\neq u\in\R^n$, let
\eq{
  H(L,u)&=\{x\in\R^n:\ip{x}{u}=h_L(u)\},\\
  F(L,u)&=L\cap H(L,u)
  =L^u.
}
Let $T(L,u)$ denote the unique face of $N_L(F(L,u))$ whose relative interior contains $u$. Since $L$ is a polytope, \eqref{eq:sec-3-relative-normal-exposed-face} implies
\eq{
  u\in\relint N_L(F(L,u)).
}
That is, $T(L,u)=N_L(F(L,u))$. For polytopes $K\subseteq L$, Schneider's definition is equivalent to the following implication for every $u\in\R^n\setminus\{0\}$:
\eq{
  \dim N_L(F(L,u))\leq 2
  \quad\Rightarrow\quad
  H(L,u)\text{ supports }K.
}
By \autoref{lem:sec-3-face-normal-cone-duality},
\eq{
  \dim N_L(F(L,u))\leq 2
  \quad\iff\quad
  \dim F(L,u)\geq n-2.
}
Since $K\subseteq L$ and $F(L,u)=L\cap H(L,u)$,
\eq{
  H(L,u)\text{ supports }K
  \quad\iff\quad
  K\cap F(L,u)\neq\varnothing.
}
Thus, for polytopes, \autoref{def:sec-5-polytopal-tangential-body} agrees with his definition \cite[p. 86]{Sch14}.

For full-dimensional polytopes $P$ and $Q$, equality in the second Minkowski inequality \eqref{eq:sec-5-second-Minkowski} holds if and only if $P$ is homothetic to an $(n-2)$-tangential body of $Q$; see \cite[p. 432, Thm. 7.6.19]{Sch14}. By \autoref{def:sec-5-polytopal-tangential-body}, this is equivalent to the existence of $\gamma>0$ and $v\in\R^n$ such that, with $\widetilde{P}=\gamma P+v$,
\eq{\label{eq:sec-5-polytopal-equality-condition}
  Q\subseteq\widetilde{P}, \quad
  Q\cap F\neq\varnothing
  \quad\text{for every face $F$ of $\widetilde{P}$ satisfying}
  \quad n-2\leq \dim F\leq n-1.
}

The following lemma does not require the two normal fans to coincide.

\begin{lemma}[Equality in the second Minkowski inequality]\label{lem:sec-5-common-normal-Minkowski-equality}
Suppose that $h,q\in\R^J$ define full-dimensional polytopes $P=P(h)$ and $Q=P(q)$ and that every prescribed inequality defines a facet of the corresponding polytope. Equality holds in \eqref{eq:sec-5-second-Minkowski} if and only if
\eq{
  Q=\gamma P+v,
}
for some $\gamma>0$ and $v\in\R^n$.
\end{lemma}

\begin{proof}
Suppose that equality holds. Then \eqref{eq:sec-5-polytopal-equality-condition} holds for some $\gamma>0$ and $v\in\R^n$.

For $i\in J$, the exposed face $\widetilde{F}_i=\widetilde{P}^{u_i}$ is a facet of $\widetilde{P}$ because the prescribed inequality with normal $u_i$ defines a facet of $P$. Applying \eqref{eq:sec-5-polytopal-equality-condition} to $\widetilde{F}_i$, we obtain
\eq{
  Q\cap\widetilde{F}_i\neq\varnothing.
}
For $x_i\in Q\cap\widetilde{F}_i$, the inclusion $Q\subseteq\widetilde{P}$ implies $h_Q(u_i)\leq h_{\widetilde{P}}(u_i)$, whereas $x_i\in Q\cap\widetilde{F}_i$ implies
\eq{
  h_Q(u_i)\geq \ip{x_i}{u_i}=h_{\widetilde{P}}(u_i).
}
Combining the two inequalities, we obtain
\eq{
  h_Q(u_i)=h_{\widetilde{P}}(u_i) \quad (i\in J).
}
Since both polytopes are intersections of the half-spaces with the prescribed normals $u_i$, these equalities imply $Q=\widetilde{P}=\gamma P+v$.

Conversely, suppose that $Q=\gamma P+v$. Translation invariance and homogeneity of mixed volumes imply
\eq{
  \mathcal{V}(P[n-1],Q)=\gamma\Vol(P), \quad
  \mathcal{V}(P[n-2],Q[2])=\gamma^2\Vol(P).
}
Thus both sides of \eqref{eq:sec-5-second-Minkowski} are equal to $\gamma^2\Vol(P)^2$.
\end{proof}

\begin{theorem}\label{thm:sec-5-support-hessian-kernel}
For $h\in\mathcal{T}_{\Sigma}$ with $0\in\operatorname{int}P(h)$, denote the first two logarithmic support forms at $h$ by $A$ and $B$, as in \eqref{eq:sec-4-variation-forms}. Then $D^2V_{\Sigma}(h)$ has exactly one positive eigenvalue, and
\eq{
  \ker D^2V_{\Sigma}(h)=\mathfrak{t}(\R^n).
}
Equivalently, every $a\in\R^J$ satisfies
\eq{\label{eq:sec-5-AF-support-form}
  B(a,a)\leq (n-1)\frac{A(a)^2}{nV_{\Sigma}(h)}.
}
Moreover, equality holds if and only if
\eq{\label{eq:sec-5-AF-equality}
  a_j=c_0+\frac{\ip{v}{u_j}}{h_j} \quad \text{for every }j\in J,
}
for some $c_0\in\R$ and $v\in\R^n$.
\end{theorem}

\begin{proof}
See the proof of \autoref{thm:sec-8-irredundant-support-Hessian-inequality}: under the invertible diagonal change of variables $a\mapsto h\odot a$, the matrix of $B$ is congruent to $D^2V_{\Sigma}(h)$. Sylvester's law of inertia and the relation $B(a,b)=D^2V_{\Sigma}(h)[h\odot a,h\odot b]$ transfer the positive-eigenvalue count and the kernel description.
\end{proof}

The remainder of the section concerns the antipodal setting. Assume that the normals are indexed by $I$, $\Sigma$ is antipodally invariant, and $h\in\mathcal{T}_{\Sigma}$ is even.

\begin{corollary}\label{cor:sec-5-strict-even-eigenvalue-bound}
The polytope $P(h)$ satisfies $\lambda_{1,e}(P(h))>n-1$.
\end{corollary}

\begin{proof}
See the proof of \autoref{cor:sec-8-irredundant-strict-even-eigenvalue-bound}.
\end{proof}

\begin{definition}\label{def:sec-5-even-facet-functions}
The space of even facet functions and its centered subspace are
\eq{
  \mathcal{E}^+
  &=\{a\in\R^I:\iota a=a\},\quad
  \mathcal{E}_0^{+}
  &=\{a\in\mathcal{E}^+:A(a)=0\}.
}
\end{definition}

By \autoref{prop:sec-4-graph-laplacian}, $\mathsf{L}$ is positive definite on $\mathcal{E}_0^{+}$. The mass form $\mathsf{M}(a,b)=\sum_{i\in I}m_i a_ib_i$ is an inner product on $\R^I$, and in particular on $\mathcal{E}_0^{+}$. Therefore, it determines a unique linear operator $\mathsf{T}:\mathcal{E}_0^{+}\to\mathcal{E}_0^{+}$ such that
\eq{
  \mathsf{M}(\mathsf{T}a,b)=\mathsf{L}(a,b)
  \quad\text{for every }a,b\in\mathcal{E}_0^{+}.
}
Since $\mathsf{L}$ is symmetric,
\eq{
  \mathsf{M}(\mathsf{T}a,b)
  =\mathsf{L}(a,b)
  =\mathsf{L}(b,a)
  =\mathsf{M}(a,\mathsf{T}b).
}
Thus $\mathsf{T}$ is positive definite and self-adjoint with respect to $\mathsf{M}$. The spectral theorem provides an $\mathsf{M}$-orthonormal basis $(e_k)$ of $\mathcal{E}_0^{+}$ consisting of eigenvectors of $\mathsf{T}$. Index the corresponding eigenvalues in nondecreasing order:
\eq{
  0<\lambda_1=\lambda_{1,e}(P)\leq \lambda_2\leq \cdots.
}
The eigenvectors then satisfy
\eq{
  \mathsf{M}(e_j,e_k)&=\delta_{jk},\\
  \mathsf{L}(e_k,b)&=\lambda_k\mathsf{M}(e_k,b)
  \quad\text{for every }b\in\mathcal{E}_0^+.
}
In view of the identity $B(a,b)=(n-1)\mathsf{M}(a,b)-\mathsf{L}(a,b)$, one has
\eq{
  B(e_j,e_k)=\rho_k\delta_{jk},
  \quad
  \rho_k=n-1-\lambda_k.
}
In view of \autoref{cor:sec-5-strict-even-eigenvalue-bound}, every nonzero $a\in\mathcal{E}_0^+$ satisfies
\eq{\label{eq:sec-5-B-spectral-order}
  B(a,a)
 \leq
  (n-1-\lambda_1)\mathsf{M}(a,a)
<0.
}

\section{Rayleigh variation in a simple type chamber}
\label{sec:simple-chamber-Rayleigh-variation}

In this section, we study logarithmic support variations at a local minimizer in a simple type chamber. We continue with the antipodally invariant fan $\Sigma$ and the even support vector $h\in\mathcal{T}_{\Sigma}$ from the preceding section. We assume that $h$ locally minimizes $q\mapsto\lambda_{1,e}(P(q))$ among the even vectors $q\in\mathcal{T}_{\Sigma}$.

\subsection{First variations of the support forms}

For an even facet function $g$, consider $h_t=h\odot(\one+tg)$. The openness of $\mathcal{T}_{\Sigma}$ from \autoref{lem:sec-3-type-chamber-inequalities} implies that $h_t\in\mathcal{T}_{\Sigma}$ for all sufficiently small $\abs{t}$. A subscript $t$ indicates that a quantity from \eqref{eq:sec-4-variation-forms}--\eqref{eq:sec-4-mass-centered-energy-forms} is evaluated at $h_t$; for example,
\eq{
   A_t(a)=DV(h_t)[h_t\odot a].
}

\begin{lemma}\label{lem:sec-6-derivatives-of-forms}
For facet functions $a$ and $b$ independent of $t$,
\eq{
   \left.\frac{d}{dt}\right|_{t=0}A_t(a)&=A(a\odot g)+B(a,g), \\ \left.\frac{d}{dt}\right|_{t=0}B_t(a,b)&=B(a\odot g,b)+B(a,b\odot g)+C(a,b,g), \\ \left.\frac{d}{dt}\right|_{t=0}S_t&=nA(g).
}
\end{lemma}

\begin{proof}
Since $h_t=h+t(h\odot g)$,
\eq{
  A_t(a)=DV\bigl(h+t(h\odot g)\bigr)
  [h\odot a+t(h\odot a\odot g)].
}
The derivative at zero is
\eq{
  \left.\frac{d}{dt}\right|_{t=0}A_t(a)
  &=D^2V(h)[h\odot g,h\odot a]+DV(h)[h\odot a\odot g] \\
  &=B(a,g)+A(a\odot g).
}

Likewise,
\eq{
  B_t(a,b)=D^2V\bigl(h+t(h\odot g)\bigr)
  [h\odot a+t(h\odot a\odot g),h\odot b+t(h\odot b\odot g)],
}
and its derivative at zero is
\eq{
  \left.\frac{d}{dt}\right|_{t=0}B_t(a,b)
  &=D^3V(h)[h\odot g,h\odot a,h\odot b] \\
  &\quad+D^2V(h)[h\odot a\odot g,h\odot b] \\
  &\quad+D^2V(h)[h\odot a,h\odot b\odot g] \\
  &=C(a,b,g)+B(a\odot g,b)+B(a,b\odot g).
}
Finally, $S_t=nV(h_t)$ and
\eq{
  \left.\frac{d}{dt}\right|_{t=0}S_t=nDV(h)[h\odot g]=nA(g).
}
\end{proof}

\subsection{Rayleigh stationarity}

Write
\eq{
  \lambda&=\lambda_{1,e}(P),\quad
  \beta&=n-1-\lambda.
}
By \autoref{cor:sec-5-strict-even-eigenvalue-bound}, $\beta<0$. Define the first nonconstant even eigenspace by
\eq{\label{eq:sec-6-first-even-eigenspace}
  \mathcal{E}_{\lambda}^{+}=\left\{u\in\mathcal{E}_0^{+}:\mathsf{L}(u,k)=\lambda\mathsf{M}(u,k)
  \text{ for every }k\in\mathcal{E}_0^{+}\right\}.
}
Note that for $u\in\mathcal{E}_{\lambda}^{+}$, the eigenvalue equation extends to every even facet function $k$. Indeed, $k$ and its centered version differ by a constant, while $\mathsf{L}(u,\one)=0=\mathsf{M}(u,\one)$. Therefore, every $u\in\mathcal{E}_{\lambda}^{+}$ and every even $k$ satisfy
\eq{\label{eq:sec-6-eigen-equations}
  A(u)&=0,\\
  \mathsf{L}(u,k)&=\lambda\mathsf{M}(u,k),\\
  B(u,k)&=\beta\mathsf{M}(u,k).
}
For the remainder of this section, we consider $f\in\mathcal{E}_{\lambda}^{+}$ normalized by
\eq{\label{eq:sec-6-eigenfunction-normalization}
  A(f)&=0,\quad
  \mathsf{M}(f,f)&=1.
}

Recall that $h$ locally minimizes $q\mapsto\lambda_{1,e}(P(q))$ among the even vectors $q\in\mathcal{T}_{\Sigma}$.

\begin{lemma}[Rayleigh stationarity]\label{lem:sec-6-rayleigh-stationarity}
For every even facet function $g$,
\eq{\label{eq:sec-6-scalar-stationarity}
  C(f,f,g)
  =\beta\bigl\{B(f^{\odot2},g)-A(f^{\odot2}\odot g)\bigr\}.
}
\end{lemma}

\begin{proof}
Along $h_t=h\odot(\one+tg)$, the centered Rayleigh quotient of $f$ at $h_t$ is
\eq{
  \mathcal{R}(t)
  =\frac{\mathsf{L}_t(f,f)}{\widetilde{\mathsf{M}}_t(f,f)}.
}
The denominator equals $1$ at $t=0$ and remains positive for sufficiently small $\abs{t}$. The Rayleigh principle and local minimality imply
\eq{
  \mathcal{R}(t)
  &\geq\lambda_{1,e}(P(h_t))
  \geq\lambda
  =\mathcal{R}(0)
}
for all sufficiently small $\abs{t}$. Hence $\mathcal{R}'(0)=0$.

Note that
\eq{
  \widetilde{\mathsf{M}}_t(f,f)
  &=A_t(f^{\odot2})-\frac{A_t(f)^2}{S_t},\\
  \mathsf{L}_t(f,f)
  &=(n-1)A_t(f^{\odot2})-B_t(f,f).
}
Since $A_0(f)=A(f)=0$, the derivative of $A_t(f)^2/S_t$  at zero vanishes. By \autoref{lem:sec-6-derivatives-of-forms}, the derivatives of $\widetilde{\mathsf{M}}_t(f,f)$ and $\mathsf{L}_t(f,f)$  at zero are
\eq{
  \left.\frac{d}{dt}\right|_{t=0}
  \widetilde{\mathsf{M}}_t(f,f)
  &=A(f^{\odot2}\odot g)+B(f^{\odot2},g),\\
  \left.\frac{d}{dt}\right|_{t=0}\mathsf{L}_t(f,f)
  &=(n-1)\bigl\{A(f^{\odot2}\odot g)+B(f^{\odot2},g)\bigr\}\\
  &\quad-B(f\odot g,f)-B(f,f\odot g)-C(f,f,g).
}
Since $f\odot g$ is even, the symmetry of $B$ and \eqref{eq:sec-6-eigen-equations} show that
\eq{
  B(f\odot g,f)=B(f,f\odot g)
  =\beta\mathsf{M}(f,f\odot g)
  =\beta A(f^{\odot2}\odot g).
}
Substitution yields
\eq{
  \left.\frac{d}{dt}\right|_{t=0}\mathsf{L}_t(f,f)
  &=(n-1)\bigl\{A(f^{\odot2}\odot g)+B(f^{\odot2},g)\bigr\}\\
  &\quad-2\beta A(f^{\odot2}\odot g)-C(f,f,g).
}
At $t=0$, $\widetilde{\mathsf{M}}(f,f)=1$ and $\mathsf{L}(f,f)=\lambda$. Differentiating the Rayleigh quotient, we obtain
\eq{
  0=\mathcal{R}'(0)
  &=\left.\frac{d}{dt}\right|_{t=0}\mathsf{L}_t(f,f)
  -\lambda\left.\frac{d}{dt}\right|_{t=0}\widetilde{\mathsf{M}}_t(f,f)\\
  &=(n-1-\lambda)\bigl\{A(f^{\odot2}\odot g)+B(f^{\odot2},g)\bigr\}\\
  &\quad-2\beta A(f^{\odot2}\odot g)-C(f,f,g)\\
  &=\beta\bigl\{B(f^{\odot2},g)-A(f^{\odot2}\odot g)\bigr\}-C(f,f,g).
}
\end{proof}

Define the shifted centered form by
\eq{
  \mathcal{D}_{\lambda}(a,b)=\mathsf{L}(a,b)-\lambda\widetilde{\mathsf{M}}(a,b).
}
The form $\mathcal{D}_{\lambda}$ is unchanged when a constant is added to either argument, and \eqref{eq:sec-6-eigen-equations} implies
\eq{\label{eq:sec-6-D-eigen-equations}
  \mathcal{D}_{\lambda}(f,k)=0 \quad \text{for every even }k.
}
Let $\mathcal{D}_{\lambda,t}$ denote the corresponding form at $h_t$.

\begin{lemma}\label{lem:sec-6-shifted-form-derivative}
Along the support path $h_t=h\odot(\one+tf)$,
\eq{\label{eq:sec-6-shifted-form-derivative}
  \left.\frac{d}{dt}\right|_{t=0} \mathcal{D}_{\lambda,t}(f,k)=(1+\beta)\mathcal{D}_{\lambda}(f^{\odot2},k)
}
for every even facet function $k$.
\end{lemma}

\begin{proof}
Since $f\odot k$ is even, symmetry of $B$, \eqref{eq:sec-6-eigen-equations}, and \eqref{eq:sec-6-eigenfunction-normalization} imply
\eq{
  B(f,f)&=\beta,\\
  B(f\odot k,f)&=B(f,f\odot k)
  =\beta\mathsf{M}(f,f\odot k)
  =\beta A(f^{\odot2}\odot k).
}
Moreover, by \autoref{lem:sec-6-derivatives-of-forms}, \eqref{eq:sec-6-eigen-equations}, and \eqref{eq:sec-6-eigenfunction-normalization},
\eq{
  \left.\frac{d}{dt}\right|_{t=0}A_t(f)
  &=A(f^{\odot2})+B(f,f)=1+\beta,\\
  \left.\frac{d}{dt}\right|_{t=0}A_t(f\odot k)
  &=A(f^{\odot2}\odot k)+B(f\odot k,f)\\
  &=(1+\beta)A(f^{\odot2}\odot k).
}
Since $A(f)=0$, these formulas determine the derivative of the centered mass form:
\eq{\label{eq:sec-6-centered-mass-derivative}
  \left.\frac{d}{dt}\right|_{t=0}\widetilde{\mathsf{M}}_t(f,k)
  &=(1+\beta)\left(A(f^{\odot2}\odot k)-\frac{A(k)}{S}\right)\\
  &=(1+\beta)\widetilde{\mathsf{M}}(f^{\odot2},k).
}

Next, symmetry of $C$ and the Rayleigh stationarity equation imply
\eq{
  C(f,k,f)=C(f,f,k)
  =\beta\bigl\{B(f^{\odot2},k)-A(f^{\odot2}\odot k)\bigr\}.
}
Substituting this identity and $B(f,f\odot k)=\beta A(f^{\odot2}\odot k)$ into the variation formula for $B_t$,
\eq{
  \left.\frac{d}{dt}\right|_{t=0}B_t(f,k)
  &=B(f^{\odot2},k)+B(f,f\odot k)+C(f,k,f)\\
  &=(1+\beta)B(f^{\odot2},k).
}

The energy form of $h_t$ at $(f,k)$ is
\eq{
  \mathsf{L}_t(f,k)=(n-1)A_t(f\odot k)-B_t(f,k).
}
Combining the preceding derivatives yields
\eq{
  \left.\frac{d}{dt}\right|_{t=0}\mathsf{L}_t(f,k)
  =(1+\beta)\mathsf{L}(f^{\odot2},k).
}
Finally, subtracting $\lambda$ times \eqref{eq:sec-6-centered-mass-derivative} from this identity, we obtain
\eq{
  \left.\frac{d}{dt}\right|_{t=0}\mathcal{D}_{\lambda,t}(f,k)
  =(1+\beta)\mathcal{D}_{\lambda}(f^{\odot2},k).
}
\end{proof}

\subsection{The trial Rayleigh quotient and its second variation}

We define a path of trial vectors by
\eq{\label{eq:sec-6-trial-vector}
  u_t=f-(1+\beta)t f^{\odot2}.
}
The derivative of $u_t$ at zero is $u_0'=-(1+\beta)f^{\odot2}$. Therefore, \eqref{eq:sec-6-shifted-form-derivative} implies, for every even facet function $k$,
\eq{\label{eq:sec-6-trial-vector-cancellation}
  \left.\frac{d}{dt}\right|_{t=0}\mathcal{D}_{\lambda,t}(u_t,k)
  =(1+\beta)\mathcal{D}_{\lambda}(f^{\odot2},k)
  +\mathcal{D}_{\lambda}(u_0',k)=0.
}
Along $h_t=h\odot(\one+tf)$, the trial Rayleigh quotient associated with $u_t$ is defined by
\eq{
  \mathcal{R}_{\mathrm{tr}}(t)
  =\frac{\mathsf{L}_t(u_t,u_t)}{\widetilde{\mathsf{M}}_t(u_t,u_t)}.
}

We also introduce the scalar quantities
\eq{\label{eq:sec-6-scalar-data}
  \mu_4=A(f^{\odot4}),\quad
  b_2=B(f^{\odot2},f^{\odot2}),\quad
  d_4=D(f,f,f,f).
}

\begin{lemma}\label{lem:sec-6-exact-second-derivative}
The trial Rayleigh quotient satisfies
\eq{\label{eq:sec-6-exact-second-derivative}
  \mathcal{R}_{\mathrm{tr}}'(0)=0, \quad \mathcal{R}_{\mathrm{tr}}''(0)=-\beta(2\beta-1)(\beta+2)\mu_4 +3\beta^2b_2-d_4.
}
\end{lemma}

\begin{proof}
Note that
\eq{
  \mathcal{R}_{\mathrm{tr}}(t)-\lambda
  =\frac{\mathcal{D}_{\lambda,t}(u_t,u_t)}
  {\widetilde{\mathsf{M}}_t(u_t,u_t)}.
}
By \eqref{eq:sec-6-D-eigen-equations}, $\mathcal{D}_{\lambda}(f,f)=\mathcal{D}_{\lambda}(f^{\odot2},f)=\mathcal{D}_{\lambda}(f,u_0')=0$. Hence, using \eqref{eq:sec-6-shifted-form-derivative} and symmetry,
\eq{
  \left.\frac{d}{dt}\right|_{t=0}\mathcal{D}_{\lambda,t}(u_t,u_t)=0.
}
Consequently, the first two derivatives at zero are
\eq{
  \mathcal{R}_{\mathrm{tr}}'(0)=0,\quad
  \mathcal{R}_{\mathrm{tr}}''(0)
  =\left.\frac{d^2}{dt^2}\right|_{t=0}
  \mathcal{D}_{\lambda,t}(u_t,u_t).
}

In the remainder of the proof, dots indicate differentiation of the form only, not of its arguments. Thus, for facet functions $a$ and $b$ independent of $t$,
\eq{
  \dot{\mathcal{D}}_{\lambda,0}(a,b)
  =\left.\frac{d}{dt}\right|_{t=0}\mathcal{D}_{\lambda,t}(a,b),\quad
  \ddot{\mathcal{D}}_{\lambda,0}(a,b)
  =\left.\frac{d^2}{dt^2}\right|_{t=0}\mathcal{D}_{\lambda,t}(a,b).
}

Applying \eqref{eq:sec-6-shifted-form-derivative} with $k=u_0'$ and using $u_0'=-(1+\beta)f^{\odot2}$, we obtain
\eq{
  \dot{\mathcal{D}}_{\lambda,0}(f,u_0')
  &=-(1+\beta)^2\mathcal{D}_{\lambda}(f^{\odot2},f^{\odot2}),\\
  \mathcal{D}_{\lambda}(u_0',u_0')
  &=(1+\beta)^2\mathcal{D}_{\lambda}(f^{\odot2},f^{\odot2}).
}
We calculate
\eq{\label{eq:sec-6-trial-second-reduction}
  \mathcal{R}_{\mathrm{tr}}''(0)
  &=\ddot{\mathcal{D}}_{\lambda,0}(f,f)
  +4\dot{\mathcal{D}}_{\lambda,0}(f,u_0')
  +2\mathcal{D}_{\lambda}(u_0',u_0')\\
  &=\ddot{\mathcal{D}}_{\lambda,0}(f,f)
  -2(1+\beta)^2\mathcal{D}_{\lambda}(f^{\odot2},f^{\odot2}).
}
Next we evaluate $\mathcal{D}_{\lambda}(f^{\odot2},f^{\odot2})$ and $\ddot{\mathcal{D}}_{\lambda,0}(f,f)$.

First, using the definitions of $\mathcal{D}_{\lambda}$, $\mathsf{L}$, and $\widetilde{\mathsf{M}}$, together with $A(f^{\odot2})=1$, we obtain
\eq{
  \mathcal{D}_{\lambda}(f^{\odot2},f^{\odot2})
  &=\mathsf{L}(f^{\odot2},f^{\odot2})
  -\lambda\widetilde{\mathsf{M}}(f^{\odot2},f^{\odot2})\\
  &=(n-1)\mathsf{M}(f^{\odot2},f^{\odot2})-B(f^{\odot2},f^{\odot2})\\
  &\quad-\lambda\left(\mathsf{M}(f^{\odot2},f^{\odot2})
  -\frac{A(f^{\odot2})^2}{S}\right)\\
  &=\beta\mu_4-b_2+\frac{\lambda}{S}.
}

Next, to calculate $\ddot{\mathcal{D}}_{\lambda,0}(f,f)$, we rewrite $\mathcal{D}_{\lambda,t}(f,f)$ as
\eq{
  \mathcal{D}_{\lambda,t}(f,f)
  &=\mathsf{L}_t(f,f)-\lambda\widetilde{\mathsf{M}}_t(f,f)\\
  &=(n-1)A_t(f^{\odot2})-B_t(f,f)
  -\lambda\left(A_t(f^{\odot2})-\frac{A_t(f)^2}{S_t}\right)\\
  &=\beta A_t(f^{\odot2})-B_t(f,f)
  +\lambda\frac{A_t(f)^2}{S_t}.
}
Recall from the proof of \autoref{lem:sec-6-shifted-form-derivative} that
\eq{
  \dot{A}_0(f)=A(f^{\odot2})+B(f,f)=1+\beta.
}
Since $A_0(f)=0$, we have
\eq{
  \left.\frac{d^2}{dt^2}\right|_{t=0}
  \frac{A_t(f)^2}{S_t}
  =\frac{2\dot{A}_0(f)^2}{S}
  =\frac{2(1+\beta)^2}{S}.
}

The formula for $\ddot{\mathcal{D}}_{\lambda,0}(f,f)$ now requires the second derivatives of $A_t(f^{\odot2})$ and $B_t(f,f)$. The eigenvalue equation and the Rayleigh stationarity equation with $g=f^{\odot2}$ imply
\eq{
  B(f^{\odot3},f)=\beta\mu_4,\quad
  C(f^{\odot2},f,f)=\beta(b_2-\mu_4).
}

Along $h_t=h+t(h\odot f)$,
\eq{
  A_t(f^{\odot2})
  &=DV\bigl(h+t(h\odot f)\bigr)
  [h\odot f^{\odot2}+t(h\odot f^{\odot3})],\\
  B_t(f,f)
  &=D^2V\bigl(h+t(h\odot f)\bigr)
  [h\odot f+t(h\odot f^{\odot2}),h\odot f+t(h\odot f^{\odot2})].
}
Differentiating twice, we obtain
\eq{
  \ddot{A}_0(f^{\odot2})
  &=2B(f^{\odot3},f)+C(f^{\odot2},f,f)\\
  &=\beta(b_2+\mu_4),\\
  \ddot{B}_0(f,f)
  &=2B(f^{\odot2},f^{\odot2})+4C(f^{\odot2},f,f)+D(f,f,f,f)\\
  &=(2+4\beta)b_2-4\beta\mu_4+d_4.
}

Substituting into the formula for $\mathcal{D}_{\lambda,t}(f,f)$, we find
\eq{
  \ddot{\mathcal{D}}_{\lambda,0}(f,f)
  &=\beta\ddot{A}_0(f^{\odot2})-\ddot{B}_0(f,f)
  +\frac{2\lambda(1+\beta)^2}{S}\\
  &=\beta^2(b_2+\mu_4)
  -\bigl\{(2+4\beta)b_2-4\beta\mu_4+d_4\bigr\}
  +\frac{2\lambda(1+\beta)^2}{S}\\
  &=(\beta^2-4\beta-2)b_2
  +\beta(\beta+4)\mu_4-d_4
  +\frac{2\lambda(1+\beta)^2}{S}.
}
Finally, substituting the formulas for $\ddot{\mathcal{D}}_{\lambda,0}(f,f)$ and $\mathcal{D}_{\lambda}(f^{\odot2},f^{\odot2})$ into \eqref{eq:sec-6-trial-second-reduction}, we arrive at
\eq{
  \mathcal{R}_{\mathrm{tr}}''(0)
  &=\bigl\{\beta^2-4\beta-2+2(1+\beta)^2\bigr\}b_2\\
  &\quad+\bigl\{\beta(\beta+4)-2\beta(1+\beta)^2\bigr\}\mu_4-d_4\\
  &=3\beta^2b_2-\beta(2\beta-1)(\beta+2)\mu_4-d_4.
}
\end{proof}

\section{The degree-two Hodge--Riemann--Minkowski inequality}
\label{sec:degree-two-HRM-estimate}

In this section, we estimate the fourth-order term $d_4$ in \eqref{eq:sec-6-exact-second-derivative} and complete the fixed-normal-fan second-variation argument. Throughout this section, $n\geq 3$, and the assumptions and notation of \autoref{sec:simple-chamber-Rayleigh-variation} remain in force. In particular, $f$ satisfies the normalization \eqref{eq:sec-6-eigenfunction-normalization}, the eigenvalue equations \eqref{eq:sec-6-eigen-equations} with $u=f$, and the Rayleigh stationarity equation \eqref{eq:sec-6-scalar-stationarity}. When $n=3$, the term $d_4$ vanishes because the volume polynomial has degree three. For $n\geq 4$, we construct a \emph{primitive correction} of $\mathfrak{d}(f)^2$ and apply the degree-two Hodge--Riemann-Minkowski inequality to obtain a lower bound on $d_4$. We recall:
\eq{
  \mu_4=A(f^{\odot4}),\quad
  b_2=B(f^{\odot2},f^{\odot2}),\quad
  d_4=D(f,f,f,f).
}

\subsection{The algebra of volume derivatives}
\label{subsec:sec-7-polytope-algebra}

The support vector $h\in \mathcal{T}_{\Sigma}$ and the polytope $P=P(h)$ with $0\in\operatorname{int}P(h)$ are fixed in the constructions below, whereas $q=(q_i)_{i\in I}$ denotes the variable in the full support space $\R^I$. For $i\in I$, differentiation in the coordinate direction $\mathbf{e}_i$ is denoted by
\eq{
  \partial_{q_i}=\frac{\partial}{\partial q_i}.
}
The algebra of constant-coefficient differential operators is
\eq{
  \mathfrak{D}_q=\R[\partial_{q_i}:i\in I].
}
For $\Theta,\Psi\in\mathfrak{D}_q$, their product is defined by $ (\Theta\Psi)\varphi=\Theta(\Psi\varphi)$ for every polynomial $\varphi$.
Since the coordinate derivatives commute, $\mathfrak{D}_q$ is a graded commutative algebra, with degree equal to the order of differentiation. $\partial_{q_i}$ has degree one, while $\partial_{q_{i_1}}\cdots\partial_{q_{i_k}}$ has degree $k$.

An \emph{ideal} of $\mathfrak{D}_q$ is a linear subspace $\mathcal{I}\subseteq\mathfrak{D}_q$ such that $\Psi\Theta\in\mathcal{I}$ whenever $\Theta\in\mathcal{I}$ and $\Psi\in\mathfrak{D}_q$. The \emph{annihilator ideal} of $V$ is
\eq{
  \operatorname{Ann}(V)
  =\{\Theta\in\mathfrak{D}_q:\Theta(V)=0\}.
}

An ideal $\mathcal{I}\subseteq\mathfrak{D}_q$ is \emph{homogeneous} if every $\Theta=\sum_{k=0}^r\Theta_k\in\mathcal{I}$, with $\Theta_k$ of degree $k$, satisfies $\Theta_k\in\mathcal{I}$ for every $k$. Since $V$ is homogeneous of degree $n$, the ideal $\operatorname{Ann}(V)$ is homogeneous: the nonzero polynomials $\Theta_k(V)$ have distinct degrees $n-k$ and cannot cancel. Hence the quotient is a finite-dimensional graded commutative algebra,
\eq{
  \mathcal{A}(V)
  =\frac{\mathfrak{D}_q}{\operatorname{Ann}(V)}
  =\bigoplus_{k=0}^n\mathcal{A}^k(V).
}
We denote the class of $\Theta$ by $[\Theta]$; then $[\Theta]=[\Psi]$ exactly when $(\Theta-\Psi)(V)=0$.

The next lemma concerns the product of complementary degrees. If $\alpha=[\Theta]\in\mathcal{A}^k(V)$ and $\gamma=[\Psi]\in\mathcal{A}^{n-k}(V)$ are represented by homogeneous operators, then $\Theta\Psi$ has degree $n$. Since $V$ has degree $n$, the polynomial $(\Theta\Psi)(V)$ is constant. The lemma proves that this constant is independent of the chosen representatives and that the resulting pairing is nondegenerate.

\begin{lemma}[{\cite[Prop. 2.5.1]{Tim99}}]\label{lem:sec-7-polytope-algebra-duality}
For every $0\leq k\leq n$, multiplication followed by evaluation on $V$ defines a nondegenerate pairing
\eq{\label{eq:sec-7-polytope-algebra-duality}
  \mathcal{A}^k(V)\times\mathcal{A}^{n-k}(V)\to\R, \quad
  (\alpha,\gamma)\mapsto(\alpha\gamma)V.
}
Here nondegenerate means that every nonzero $\alpha\in\mathcal{A}^k(V)$ pairs nontrivially with some $\gamma\in\mathcal{A}^{n-k}(V)$, and conversely with the two factors interchanged.
\end{lemma}

\begin{proof}
Since $\operatorname{Ann}(V)$ is an ideal, replacing $\Theta$ or $\Psi$ by another representative changes $\Theta\Psi$ by an operator in $\operatorname{Ann}(V)$, which vanishes on $V$.

For $0\neq\alpha=[\Theta]\in\mathcal{A}^k(V)$ represented by a homogeneous operator $\Theta$ of degree $k$, the polynomial $\Theta(V)$ is nonzero and homogeneous of degree $n-k$. For a multi-index $\nu=(\nu_i)_{i\in I}$ of nonnegative integers, put
\eq{
  \abs{\nu}=\sum_{i\in I}\nu_i,\quad
  \nu!=\prod_{i\in I}\nu_i!,\quad
  q^{\nu}=\prod_{i\in I}q_i^{\nu_i},\quad
  \partial_q^{\nu}=\prod_{i\in I}\partial_{q_i}^{\nu_i}.
}
The homogeneous expansion of $\Theta(V)$ is
\eq{
  \Theta(V)(q)=\sum_{\abs{\nu}=n-k}c_{\nu} q^{\nu}.
}
Since $\Theta(V)$ is nonzero, some multi-index $\nu^{*}$ satisfies $\abs{\nu^{*}}=n-k$ and $c_{\nu^{*}}\neq 0$. Whenever $\partial_q^{\nu^{*}}q^{\nu}\neq 0$, we have $\nu_i\geq \nu_i^{*}$ for every $i\in I$. Both multi-indices have total degree $n-k$. Therefore, the coordinatewise inequalities force $\nu=\nu^{*}$. Consequently,
\eq{
  \partial_q^{\nu^{*}}\Theta(V)
  =(\nu^{*})!c_{\nu^{*}}\neq 0.
}
Therefore, the class $[\partial_q^{\nu^{*}}]\in\mathcal{A}^{n-k}(V)$ pairs nontrivially with $\alpha$. Applying the same argument to a nonzero class in $\mathcal{A}^{n-k}(V)$ proves nondegeneracy in the other factor.
\end{proof}

\subsection{The radial class and primitive degree-two elements}

For $z\in\R^I$, denote differentiation in the direction $z$ by
\eq{
  \partial_z=\sum_{i\in I}z_i\partial_{q_i}.
}
Every homogeneous first-order operator in $\mathfrak{D}_q$ has this form. Since $h_i>0$ for every $i\in I$, the relation $z=h\odot a$ determines a unique facet function $a$, namely $a_i=z_i/h_i$.
For a facet function $a$, define
\eq{
  \mathfrak{d}(a)=[\partial_{h\odot a}]
  \in\mathcal{A}^1(V).
}
Every class in $\mathcal{A}^1(V)$ has this form. The constant facet function $\one$ corresponds to the radial support direction $h$, whose class
\eq{
  \omega=\mathfrak{d}(\one)=[\partial_h]
  \in\mathcal{A}^1(V)
}
is called the \emph{radial Lefschetz class}.

When $n\geq 4$, a class $\Gamma\in\mathcal{A}^2(V)$ is said to be \emph{primitive with respect to $\omega$} if
\eq{\label{eq:sec-7-primitive-degree-two}
  \Gamma\omega^{n-3}=0
  \quad\text{in }\mathcal{A}^{n-1}(V).
}
Since every class in $\mathcal{A}^1(V)$ has the form $\mathfrak{d}(a)$, nondegeneracy of the pairing in \autoref{lem:sec-7-polytope-algebra-duality} implies that \eqref{eq:sec-7-primitive-degree-two} is equivalent to
\eq{
  \bigl(\Gamma\omega^{n-3}\mathfrak{d}(a)\bigr)V=0
  \quad\text{for every facet function }a.
}

\subsection{The degree-two Hodge--Riemann--Minkowski inequality}

We now state the degree-two Hodge--Riemann--Minkowski inequality for $V$ in $\mathcal{A}(V)$.

\begin{theorem}\label{thm:sec-7-degree-two-HRM}
Suppose that $n\geq 4$. If $\Gamma\in\mathcal{A}^2(V)$ satisfies $\Gamma\omega^{n-3}=0$, then
\eq{
  \bigl(\Gamma^2\omega^{n-4}\bigr)V\geq 0.
}
Equality holds if and only if $\Gamma=0$.
\end{theorem}

\autoref{thm:sec-7-degree-two-HRM} is the degree-two case of McMullen's Hodge--Riemann--Minkowski inequalities \cite[Thm. 8.2]{McM93}, expressed in Timorin's differential-operator formulation \cite[Thm. 5.1.1]{Tim99}.

\begin{remark}\label{rem:sec-7-mixed-volume-HR}
Assume that $n\geq 4$. The degree-two Hodge--Riemann inequality of McMullen and Timorin admits the following mixed-volume formulation; see van Handel \cite[Thm. 3.1]{vH23a}. In the specialization relevant here, let $K_1,\ldots,K_m$ have the same normal fan $\Sigma$ as the fixed simple polytope $P$, and let $x_1,\ldots,x_m\in\R$. If
\eq{
  \sum_{i=1}^m x_i\mathcal{V}(K_i[2],M,P[n-3])=0
}
for every polytope $M$ with normal fan $\Sigma$, then
\eq{
  \sum_{i,j=1}^m x_ix_j\mathcal{V}(K_i[2],K_j[2],P[n-4])\geq 0.
}
Write $q^{(i)}=(h_{K_i}(u_j))_{j\in I}$ for the support vector of $K_i$, and set
\eq{
  \Gamma=\sum_{i=1}^m x_i[\partial_{q^{(i)}}]^2.
}
The mixed-volume derivative identity and \autoref{lem:sec-7-polytope-algebra-duality} imply that the hypothesis is equivalent to $\Gamma\omega^{n-3}=0$, while the sum in the conclusion equals $(\Gamma^2\omega^{n-4})V/n!$. Every degree-two class admits such a representation by polarization \cite[Lem. 3.3]{vH23a}. Equality holds if and only if $\Gamma=0$; see \cite[Thm. 3.2]{vH23a}.

This formulation assumes a common simple normal fan. At nonsimple support vectors, we apply the inequality inside an incident simple type chamber and pass the fourth-derivative bound to the boundary; see \autoref{lem:sec-10-incident-fourth-derivative-lower-bound}.
\end{remark}

\begin{lemma}\label{lem:sec-7-repeated-radial-derivatives}
Let $0\leq k\leq n$ and $z_1,\ldots,z_k\in\R^I$. Then
\eq{\label{eq:sec-7-repeated-radial-derivatives}
  \bigl(
    \partial_{z_1}\cdots\partial_{z_k}
    (\partial_h)^{n-k}
  \bigr)V
  =(n-k)!D^kV(h)[z_1,\ldots,z_k].
}
\end{lemma}

\begin{proof}
Set $ p=\partial_{z_1}\cdots\partial_{z_k}V$. This is a homogeneous polynomial of degree $n-k$. Since the differential operators have constant coefficients and commute,
\eq{
  \bigl(
    \partial_{z_1}\cdots\partial_{z_k}
    (\partial_h)^{n-k}
  \bigr)V
  =(\partial_h)^{n-k}p.
}
The right-hand side is constant and may be evaluated at the origin:
\eq{
  (\partial_h)^{n-k}p
  =\left.\frac{d^{n-k}}{dt^{n-k}}\right|_{t=0}p(th)
  =(n-k)!p(h)
  =(n-k)!D^kV(h)[z_1,\ldots,z_k].
}
\end{proof}

For facet functions $a,b,c,d$, it follows from \autoref{lem:sec-7-repeated-radial-derivatives} and \eqref{eq:sec-4-variation-forms} that
\eq{
  \bigl(\mathfrak{d}(a)\mathfrak{d}(b)\omega^{n-2}\bigr)V
  &=(n-2)!B(a,b),\quad n\geq 2,\\
  \bigl(\mathfrak{d}(a)\mathfrak{d}(b)\mathfrak{d}(c)
    \omega^{n-3}\bigr)V
  &=(n-3)!C(a,b,c),\quad n\geq 3,\\
  \bigl(\mathfrak{d}(a)\mathfrak{d}(b)\mathfrak{d}(c)
    \mathfrak{d}(d)\omega^{n-4}\bigr)V
  &=(n-4)!D(a,b,c,d),\quad n\geq 4.
}
\subsection{A primitive correction}

We first find the unique even facet function $\eta$ that satisfies
\eq{\label{eq:sec-7-eta-definition}
  C(f,f,g)=(n-2)B(\eta,g)
  \quad\text{for every even facet function }g.
}
To see that this equation has a unique solution, consider the linear map
\eq{
  \mathcal{E}^{+}\to(\mathcal{E}^{+})^{*},\quad
  a\mapsto B(a,\mathord{\cdot}).
}
Let $a$ belong to its kernel. Write $a=c_0\one+a^{\circ}$, where $c_0=\frac{A(a)}{S}$ and $a^{\circ}\in\mathcal{E}_0^{+}$. Then
\eq{
0&=B(a,\one)=(n-1)A(a) \implies c_0=0.
}
Moreover, $0=B(a,a)=B(a^{\circ},a^{\circ})$. By \eqref{eq:sec-5-B-spectral-order}, $a^{\circ}=0$. Thus, the map is bijective, and \eqref{eq:sec-7-eta-definition} has a unique even solution $\eta$ when $n\geq 3$.

Since $f$ and $\eta$ are even, \autoref{lem:sec-3-antipodal-volume-parity} implies that
\eq{
C(f,f,g)=0,
\quad
B(\eta,g)=0
\quad\text{for every odd facet function }g.
}
Therefore, 
\eq{\label{eq:sec-7-eta-definition-all}
  C(f,f,g)=(n-2)B(\eta,g)
  \quad\text{for every facet function }g.
}

Assume now that $n\geq 4$. The class $\mathfrak{d}(f)^2$ is not necessarily primitive, and we introduce a correction by setting
\eq{\label{eq:sec-7-Gamma-definition}
  \Gamma=\mathfrak{d}(f)^2-\omega\mathfrak{d}(\eta).
}
For every facet function $g$, \autoref{lem:sec-7-repeated-radial-derivatives} and \eqref{eq:sec-7-eta-definition-all} yield
\eq{
  \bigl(\Gamma\omega^{n-3}\mathfrak{d}(g)\bigr)V
  &=\bigl(\mathfrak{d}(f)^2\mathfrak{d}(g)\omega^{n-3}\bigr)V
    -\bigl(\mathfrak{d}(\eta)\mathfrak{d}(g)\omega^{n-2}\bigr)V\\
  &=(n-3)!C(f,f,g)-(n-2)!B(\eta,g)\\
  &=(n-3)!\bigl(C(f,f,g)-(n-2)B(\eta,g)\bigr)\\
  &=0.
}
The classes $\mathfrak{d}(g)$ range over $\mathcal{A}^1(V)$. By the nondegeneracy of the pairing in \autoref{lem:sec-7-polytope-algebra-duality}, the preceding equality implies $\Gamma\omega^{n-3}=0$. Accordingly, $\Gamma$ is primitive, and the decomposition
\eq{
  \mathfrak{d}(f)^2
  =\Gamma+\omega\mathfrak{d}(\eta)
}
separates the primitive component from the component in $\omega\mathcal{A}^1(V)$.

\subsection{The fourth-derivative lower bound}

\begin{lemma}\label{lem:sec-7-fourth-derivative-lower-bound}
For the even facet function $\eta$ defined by \eqref{eq:sec-7-eta-definition}, the fourth derivative satisfies
\eq{\label{eq:sec-7-fourth-derivative-lower-bound}
  d_4\geq (n-2)(n-3)B(\eta,\eta).
}
For $n\geq 4$, equality holds if and only if $\Gamma=0$.
\end{lemma}

\begin{proof}
If $n=3$, the volume polynomial has degree three. Hence both sides of \eqref{eq:sec-7-fourth-derivative-lower-bound} vanish. 
Suppose that $n\geq 4$. By construction, the class $\Gamma$ in \eqref{eq:sec-7-Gamma-definition} is primitive. The first inequality below follows from \autoref{thm:sec-7-degree-two-HRM}. We then expand $\Gamma$, apply \autoref{lem:sec-7-repeated-radial-derivatives} term by term, and use \eqref{eq:sec-7-eta-definition} with $g=\eta$:
\eq{
  0
  &\leq \frac{1}{(n-4)!}
    \bigl(\Gamma^2\omega^{n-4}\bigr)V\\
  &=\frac{1}{(n-4)!}\bigl(
    \mathfrak{d}(f)^4\omega^{n-4}
    -2\mathfrak{d}(f)^2\mathfrak{d}(\eta)\omega^{n-3}
    +\mathfrak{d}(\eta)^2\omega^{n-2}\bigr)V\\
  &=d_4-2(n-3)C(f,f,\eta)
    +(n-3)(n-2)B(\eta,\eta)\\
  &=d_4-2(n-3)(n-2)B(\eta,\eta)
    +(n-3)(n-2)B(\eta,\eta)\\
  &=d_4-(n-3)(n-2)B(\eta,\eta).
}
\end{proof}

\subsection{The trial Rayleigh second-variation bound}

\begin{lemma}\label{lem:sec-7-trial-Rayleigh-second-variation-bound}
For the trial vector $u_t$ defined in \eqref{eq:sec-6-trial-vector}, the trial quotient $\mathcal{R}_{\mathrm{tr}}$ satisfies
\eq{\label{eq:sec-7-final-second-variation}
  \mathcal{R}_{\mathrm{tr}}'(0)=0,
  \quad
  \mathcal{R}_{\mathrm{tr}}''(0)
 \leq
  \frac{2\beta\lambda}{S(n-1)}
  \bigl(1+(n-1)\beta\bigr).
}
Equality in the estimate for $\mathcal{R}_{\mathrm{tr}}''(0)$ holds if and only if $f^{\odot2}=\frac{1}{S}\one$ and, when $n\geq 4$, the primitive class $\Gamma$ in \eqref{eq:sec-7-Gamma-definition} vanishes.
\end{lemma}

\begin{proof}
$\mathcal{R}_{\mathrm{tr}}'(0)=0$ follows from \eqref{eq:sec-6-exact-second-derivative}. Recall from \autoref{sec:support-hessian} that $(e_k)$ is an $\mathsf{M}$-orthonormal basis of $\mathcal{E}_0^{+}$ satisfying $B(e_j,e_k)=\rho_k\delta_{jk}$, where $\rho_k=n-1-\lambda_k$. The spectral order \eqref{eq:sec-5-B-spectral-order} implies $\rho_k\leq \beta<0$. 

The vector $w=f^{\odot2}$ is even and $A(w)=1$. Thus $w-S^{-1}\one\in\mathcal{E}_0^{+}$, and
\eq{
  w=\frac{1}{S}\one+\sum_k\gamma_ke_k,
}
for some constants $\gamma_k$. Also note that \eqref{eq:sec-6-scalar-stationarity} and \eqref{eq:sec-7-eta-definition} imply
\eq{\label{eq:sec-7-eta-spectral-identity}
  (n-2)B(\eta,g)
  =\beta\bigl(B(w,g)-\mathsf{M}(w,g)\bigr)
  \quad\text{for every even }g.
}

We next determine the coefficients in the expansion of $\eta$:
\eq{
  \eta=c_0\one+\sum_k\xi_ke_k.
}
Taking $g=\one$ in \eqref{eq:sec-7-eta-spectral-identity} and noting that $A(\one)=S$ and $A(e_k)=B(\one,e_k)=0$, we obtain
\eq{
  (n-2)(n-1)c_0S
  &=\beta\bigl(B(w,\one)-\mathsf{M}(w,\one)\bigr)\\
  &=\beta\bigl((n-1)A(w)-A(w)\bigr)
  =\beta(n-2).
}
It follows that $c_0=\frac{\beta}{(n-1)S}$.

Taking $g=e_j$ in \eqref{eq:sec-7-eta-spectral-identity} yields
\eq{
  (n-2)\rho_j\xi_j
  =(n-2)B(\eta,e_j)
  &=\beta\bigl(B(w,e_j)-\mathsf{M}(w,e_j)\bigr)\\
  &=\beta(\rho_j-1)\gamma_j.
}
Since $\rho_j\neq 0$, we have $\xi_j=\frac{\beta}{n-2}\frac{\rho_j-1}{\rho_j}\gamma_j$. Hence
\eq{
  \eta=\frac{\beta}{(n-1)S}\one+\frac{\beta}{n-2}\sum_k\frac{\rho_k-1}{\rho_k}\gamma_ke_k.
}

Now, the expansions of $w$ and $\eta$ yield
\eq{
  \mu_4=\mathsf{M}(w,w)
  &=\frac{1}{S}+\sum_k\gamma_k^2,\\
  b_2=B(w,w)
  &=\frac{n-1}{S}+\sum_k\rho_k\gamma_k^2,\\
  B(\eta,\eta)
  &=\frac{\beta^2}{(n-1)S}
    +\frac{\beta^2}{(n-2)^2}
    \sum_k\frac{(\rho_k-1)^2}{\rho_k}\gamma_k^2.
}

Recall from \autoref{lem:sec-6-exact-second-derivative} that
\eq{
  \mathcal{R}_{\mathrm{tr}}''(0)
  =-\beta(2\beta-1)(\beta+2)\mu_4
  +3\beta^2b_2-d_4.
}
Substituting the formulas for $\mu_4$, $b_2$, and $B(\eta,\eta)$, we obtain
\eq{\label{eq:sec-7-second-variation-remainder}
  \mathcal{R}_{\mathrm{tr}}''(0)
  &-\frac{2\beta\lambda}{(n-1)S}
  \bigl(1+(n-1)\beta\bigr)\\
  &=-\bigl(d_4-(n-2)(n-3)B(\eta,\eta)\bigr)
  -\frac{\beta^2(2n-3)}{n-2}
    \sum_k(\beta-\rho_k)\gamma_k^2\\
  &\quad-\frac{\beta(\beta-1)\lambda}{n-2}
    \sum_k\gamma_k^2-\frac{\beta^2(n-3)}{n-2}
    \sum_k(\rho_k^{-1}-\beta^{-1})\gamma_k^2.
}
By \autoref{lem:sec-7-fourth-derivative-lower-bound}, the first term on the right is nonpositive. Moreover, the remaining three terms are nonpositive, and the estimate follows.

It remains to determine when equality holds. Equality forces $\gamma_k=0$ for every $k$. This is equivalent to $f^{\odot2}=S^{-1}\one$. If $n=3$, the first term on the right-hand side vanishes because $d_4=0$. If $n\geq 4$, \autoref{lem:sec-7-fourth-derivative-lower-bound} shows that $d_4=(n-2)(n-3)B(\eta,\eta)$ if and only if $\Gamma=0$.
\end{proof}

\begin{proposition}\label{prop:sec-7-fixed-fan-lower-bound}
Suppose that $n\geq 3$ and that an even $h\in\mathcal{T}_{\Sigma}$ satisfies
\eq{
  \lambda_{1,e}(P(q))\geq \lambda_{1,e}(P(h))
}
for every even $q\in\mathcal{T}_{\Sigma}$ sufficiently close to $h$. Then
\eq{\label{eq:sec-7-fixed-fan-lower-bound}
  \lambda_{1,e}(P(h))\geq n-1+\frac{1}{n-1}.
}
\end{proposition}

\begin{proof}
Let $f$ be a first nonconstant even eigenfunction normalized by \eqref{eq:sec-6-eigenfunction-normalization}, and consider the support path $h_t=h\odot(\one+tf)$ with the trial quotient $\mathcal{R}_{\mathrm{tr}}$ from \autoref{lem:sec-7-trial-Rayleigh-second-variation-bound}. By \autoref{cor:sec-5-strict-even-eigenvalue-bound}, $\lambda=\lambda_{1,e}(P(h))>n-1$. If \eqref{eq:sec-7-fixed-fan-lower-bound} failed, then $-1/(n-1)<\beta<0$, and \eqref{eq:sec-7-final-second-variation} would imply $\mathcal{R}_{\mathrm{tr}}''(0)<0$. For all sufficiently small $\abs{t}$, the vector $h_t$ is even and belongs to $\mathcal{T}_{\Sigma}$. The Rayleigh principle and local minimality yield $\mathcal{R}_{\mathrm{tr}}(t)\geq \mathcal{R}_{\mathrm{tr}}(0)$. Thus $t=0$ is a local minimum of $\mathcal{R}_{\mathrm{tr}}$, contrary to $\mathcal{R}_{\mathrm{tr}}''(0)<0$.
\end{proof}

\section{The discrete centro-affine structure}
\label{sec:discrete-centro-affine-structure}

In this section, we first introduce the irredundant support domain and establish its basic properties. We then define the intrinsic support forms and the weighted graph Laplacian on this domain and prove the support-Hessian inequality and its equality case. Moreover, we establish upper semicontinuity of the first nonconstant even eigenvalue under polyhedral approximation. We conclude with mixed-volume proofs of $C^2$ regularity and the support-Hessian midpoint inequality.

\subsection{The irredundant domain and intrinsic support forms}
\label{sec:irredundant-support-domain}

For $q\in\R^J$, let $P(q)$ be the half-space intersection in \eqref{eq:sec-3-polytope-support-parameters}; the normals remain fixed while $q$ varies, without any symmetry assumption. For the $J$-indexed normal data, we omit the superscript $J$ and write $\mathcal{U}_{\mathrm{irr}}$. A superscript $I$ will distinguish the antipodal specialization.

\begin{definition}[Irredundant support domain]\label{def:sec-8-irredundant-support-domain}
The \emph{irredundant support domain} is
\eq{
  \mathcal{U}_{\mathrm{irr}}
  =\left\{h\in\R^J:
    \begin{array}{l}
    0\in\operatorname{int}P(h),\text{ and}\\
    \text{every labeled inequality in
      \eqref{eq:sec-3-polytope-support-parameters} defines a facet of }P(h)
    \end{array}
  \right\}.
}
\end{definition}

For unit facet normals, $\mathcal{U}_{\mathrm{irr}}$ is the part of Izmestiev's support-parameter space \cite[Sec. 2.2]{Izm10} for which $0\in\operatorname{int}P(h)$. The mixed support-derivative formula in \autoref{prop:sec-8-support-derivatives} appears in \cite[proof of Thm. 2.4]{Izm10}. The support-Hessian inequality and equality case in \autoref{thm:sec-8-irredundant-support-Hessian-inequality} follow from \cite[Lem. A.9, Thm. A.10]{Izm10}. He also records the $C^2$ regularity of \autoref{thm:sec-9-C2-regularity} in \cite[Sec. A.3]{Izm10}. The image of $\mathcal{U}_{\mathrm{irr}}$ under the quotient by translations is the interior of the irredundancy domain of Fillastre and Izmestiev \cite[Def. 1.22, Lem. 1.34]{FI17}.

On $\mathcal{U}_{\mathrm{irr}}$, we write
\eq{
  V(q)=\Vol(P(q))\quad(q\in\mathcal{U}_{\mathrm{irr}}).
}

\begin{lemma}\label{lem:sec-8-irredundant-open}
The domain $\mathcal{U}_{\mathrm{irr}}$ is an open convex cone.
\end{lemma}

\begin{proof}
Let $h\in\mathcal{U}_{\mathrm{irr}}$. For $i\in J$, write
\eq{
  H_i(q)&=\{x\in\R^n:\ip{x}{u_i}=q_i\},\quad
  F_i(q)&=P(q)\cap H_i(q).
}
Choose $y^{(i)}(h)\in\relint F_i(h)$. For $q$ near $h$, set $y^{(i)}(q)=y^{(i)}(h)+(q_i-h_i)u_i$. Since $u_i$ is a unit vector,
\eq{
  \ip{y^{(i)}(q)}{u_i}&=q_i,\\
  q_j-\ip{y^{(i)}(q)}{u_j}
  &=h_j-\ip{y^{(i)}(h)}{u_j}
    +(q_j-h_j)-(q_i-h_i)\ip{u_i}{u_j}
  \quad(j\neq i).
}
The first term on the right is positive because $y^{(i)}(h)\in\relint F_i(h)$. Finiteness of $J$ provides a neighborhood of $h$ in which all these slacks remain positive. Thus $F_i(q)$ contains a relatively open neighborhood of $y^{(i)}(q)$ in $H_i(q)$ and is a facet. After shrinking the neighborhood, every coordinate of $q$ is positive. Hence $0\in\operatorname{int}P(q)$. It follows that $\mathcal{U}_{\mathrm{irr}}$ is open.

Positive dilations preserve $\mathcal{U}_{\mathrm{irr}}$. To prove closure under addition, consider $h,h'\in\mathcal{U}_{\mathrm{irr}}$. For each $i\in J$, take
\eq{
  y\in\relint F_i(h),\quad
  y'\in\relint F_i(h').
}
At $y+y'$, the $i$th inequality defining $P(h+h')$ is an equality, while all the other defining inequalities are strict. Therefore, this inequality defines a facet of $P(h+h')$. Every coordinate of $h+h'$ is positive. Hence $0\in\operatorname{int}P(h+h')$. Thus $h+h'\in\mathcal{U}_{\mathrm{irr}}$, and closure under addition and positive dilation makes $\mathcal{U}_{\mathrm{irr}}$ a convex cone.
\end{proof}

Before introducing the intrinsic support forms on $\mathcal{U}_{\mathrm{irr}}$, we prove that every irredundant support vector lies in the closure of a simple type chamber. We then construct a neighborhood covered by incident chamber closures, with a symmetric covering for even supports.

Let $\mathcal{B}$ denote the collection of $n$-element subsets $\tilde{J}\subseteq J$ for which $(u_i)_{i\in\tilde{J}}$ is a basis of $\R^n$. As in \eqref{eq:sec-3-maximal-cone-vertex}, each $\tilde{J}\in\mathcal{B}$ determines a unique point $x_{\tilde{J}}(q)$ by
\eq{
  \ip{x_{\tilde{J}}(q)}{u_i}=q_i
  \quad(i\in\tilde{J}).
}
For $k\notin\tilde{J}$, define the corresponding slack by
\eq{\label{eq:sec-8-basis-slack}
  \delta_{\tilde{J},k}(q)=q_k-\ip{x_{\tilde{J}}(q)}{u_k}.
}
This is the same slack construction as $\ell_{\zeta,j}$ in the proof of \autoref{lem:sec-3-type-chamber-inequalities}, now for an arbitrary basis subset $\tilde{J}$. Both $x_{\tilde{J}}$ and $\delta_{\tilde{J},k}$ depend linearly on $q$. With $\mathbf{e}_k$ denoting the $k$th standard coordinate vector of $\R^J$,
\eq{\label{eq:sec-8-slack-coordinate-variation}
  \delta_{\tilde{J},k}(q+s\mathbf{e}_k)=\delta_{\tilde{J},k}(q)+s
  \quad(q\in\R^J,\ s\in\R).
}
In particular, every $\delta_{\tilde{J},k}$ is a nonzero linear functional.

\begin{lemma}\label{lem:sec-8-vertex-simplicity-criterion}
Let $q\in\mathcal{U}_{\mathrm{irr}}$. For every $\tilde{J}\in\mathcal{B}$,
\eq{\label{eq:sec-8-basis-intersection-in-polytope}
  x_{\tilde{J}}(q)\in P(q)
  \quad\iff\quad
  \delta_{\tilde{J},k}(q)\geq 0
  \quad(k\notin\tilde{J}).
}
Moreover, $P(q)$ is simple if and only if, for every $\tilde{J}\in\mathcal{B}$, either
\eq{\label{eq:sec-8-nonvertex-alternative}
  \delta_{\tilde{J},k}(q)<0
  \quad\text{for some }k\notin\tilde{J},
}
or
\eq{\label{eq:sec-8-simple-vertex-alternative}
  \delta_{\tilde{J},k}(q)>0
  \quad(k\notin\tilde{J}).
}
In the latter case, $x_{\tilde{J}}(q)$ is a vertex whose incident facets are precisely those indexed by $\tilde{J}$.
\end{lemma}

\begin{proof}
By definition, $\ip{x_{\tilde{J}}(q)}{u_i}=q_i$ for $i\in\tilde{J}$. For $k\notin\tilde{J}$, the inequality $\ip{x_{\tilde{J}}(q)}{u_k}\leq q_k$ is equivalent to $\delta_{\tilde{J},k}(q)\geq 0$. Hence $x_{\tilde{J}}(q)$ belongs to $P(q)$ exactly under the condition in \eqref{eq:sec-8-basis-intersection-in-polytope}.

When $x_{\tilde{J}}(q)\in P(q)$, $x_{\tilde{J}}(q)$ is a vertex: the $n$ supporting hyperplanes with labels in $\tilde{J}$ have independent normals and meet only at this point. Since $q\in\mathcal{U}_{\mathrm{irr}}$, the facets containing this vertex have label set
\eq{
  \tilde{J}\cup\{k\notin\tilde{J}:\delta_{\tilde{J},k}(q)=0\}.
}

Now we prove the second part of the lemma. Assume that $P(q)$ is simple. If $x_{\tilde{J}}(q)\notin P(q)$, then \eqref{eq:sec-8-basis-intersection-in-polytope} implies that $\delta_{\tilde{J},k}(q)<0$ for some $k\notin\tilde{J}$. If $x_{\tilde{J}}(q)\in P(q)$, all the remaining slacks are nonnegative. Simplicity and the preceding description of the incident facets imply $\delta_{\tilde{J},k}(q)>0$ for all $k\notin\tilde{J}$. Hence \eqref{eq:sec-8-simple-vertex-alternative} holds.

Conversely, suppose that the two alternatives hold for every $\tilde{J}\in\mathcal{B}$. Given a vertex $y$ of $P(q)$, the outer normals of the facets containing $y$ span $\R^n$. Select $n$ such facets whose outer normals form a basis, and let $\tilde{J}$ be their label set. Since $y$ belongs to these facets,
\eq{
  \ip{y}{u_i}=q_i
  \quad(i\in\tilde{J}).
}
These equations have the unique solution $x_{\tilde{J}}(q)$. Hence $y=x_{\tilde{J}}(q)$. Since $y\in P(q)$, \eqref{eq:sec-8-nonvertex-alternative} is impossible. Therefore, $\delta_{\tilde{J},k}(q)>0$ for every $k\notin\tilde{J}$, and no facet with a label outside $\tilde{J}$ contains $y$. Thus exactly the $n$ facets indexed by $\tilde{J}$ contain $y$, and $P(q)$ is simple.
\end{proof}

For $h\in\mathcal{U}_{\mathrm{irr}}$, set
\eq{\label{eq:sec-8-exceptional-support-directions}
  \mathcal{N}_h
  =\bigcup_{\substack{\tilde{J}\in\mathcal{B},\ k\notin\tilde{J}\\
      \delta_{\tilde{J},k}(h)=0}}
    \ker\delta_{\tilde{J},k}.
}
By \eqref{eq:sec-8-slack-coordinate-variation}, this is a finite union of proper hyperplanes.

For statements concerning even directions, denote the irredundant support domain of the antipodally indexed normals by $\mathcal{U}_{\mathrm{irr}}^I\subseteq\R^I$. By \autoref{lem:sec-8-irredundant-open}, this domain is an open convex cone. In passages concerning this domain, $\mathcal{B}$ denotes the collection of basis subsets $\tilde{J}\subseteq I$, while $x_{\tilde{J}}(q)$, $\delta_{\tilde{J},k}$, and $\mathcal{N}_h$ are defined by the preceding formulas with indices in $I$.

\begin{lemma}\label{lem:sec-8-simple-chamber-approximation}
Let $h\in\mathcal{U}_{\mathrm{irr}}$.
For every $z\in\R^J\setminus\mathcal{N}_h$, there are $t_0>0$ and a simple type chamber $\sigma=\mathcal{T}_{\Sigma}$ such that
\eq{
  h+tz\in\sigma\cap\mathcal{U}_{\mathrm{irr}}
  \quad(0<t<t_0).
}
Consequently, $P(h+tz)$ is simple for $0<t<t_0$ and $h\in\overline{\sigma}$.
The same conclusion holds for $h\in\mathcal{U}_{\mathrm{irr}}^I$, with the notation just introduced on $\R^I$.
\end{lemma}

\begin{proof}
Let $z\in\R^J\setminus\mathcal{N}_h$. For every $\tilde{J}\in\mathcal{B}$ and $k\notin\tilde{J}$, linearity of the slack functional implies
\eq{
  \delta_{\tilde{J},k}(h+tz)
  =\delta_{\tilde{J},k}(h)+t\delta_{\tilde{J},k}(z).
}
If $\delta_{\tilde{J},k}(h)=0$, the definition of $\mathcal{N}_h$ ensures that $\delta_{\tilde{J},k}(z)\neq 0$. Hence $\delta_{\tilde{J},k}(h+tz)$ has constant sign for every $t>0$. If $\delta_{\tilde{J},k}(h)\neq 0$, the sign of $\delta_{\tilde{J},k}(h+tz)$ agrees with that of $\delta_{\tilde{J},k}(h)$ for all sufficiently small $t>0$. The family of slack functionals is finite, and $\mathcal{U}_{\mathrm{irr}}$ is open. Therefore, there is $t_0>0$ such that $h+tz\in\mathcal{U}_{\mathrm{irr}}$ and every $\delta_{\tilde{J},k}(h+tz)$ is nonzero with a sign independent of $t$ whenever $0<t<t_0$. Now \autoref{lem:sec-8-vertex-simplicity-criterion} shows that $P(h+tz)$ is simple for $0<t<t_0$. The same lemma shows that $\tilde{J}$ labels the facets incident to a vertex precisely when
\eq{
  \delta_{\tilde{J},k}(h+tz)>0
  \quad(k\notin\tilde{J}).
}
In particular, the vertex--facet incidences are also independent of $t$. For every such $\tilde{J}$, the corresponding normal cone is $\operatorname{pos}\{u_i:i\in\tilde{J}\}$. Thus the maximal cones of the normal fan are independent of $t$. Since a complete fan is determined by its maximal cones, all these polytopes have the same normal fan, denoted by $\Sigma$. Then
\eq{\label{eq:sec-8-generic-ray-chamber}
  h+tz\in\mathcal{T}_{\Sigma}
  \quad(0<t<t_0)
}
and we have $h\in\overline{\mathcal{T}_{\Sigma}}$. The proof on $\mathcal{U}_{\mathrm{irr}}^I$ is identical.
\end{proof}

We say a simple type chamber $\sigma$ is \emph{incident to $h$} if $h\in\overline{\sigma}$. We now construct a neighborhood of $h$ covered by the closures of incident chambers. 

Fillastre and Izmestiev \cite[Cor. 1.42]{FI17} describe the type-cone decomposition modulo translations. For the local covering here, we use a refinement on $\mathcal{U}_{\mathrm{irr}}$ obtained by subdividing $\R^J$ along the hyperplanes
\eq{
  \{\ker\delta_{\tilde{J},k}:\tilde{J}\in\mathcal{B},\ k\notin\tilde{J}\},
}
where the slack functionals are defined in \eqref{eq:sec-8-basis-slack}. We call each connected component of the complement a \emph{support cell}, denoted by $\mathfrak{C}$. The support cells are open convex polyhedral cones. Their closures, together with their faces, form a finite polyhedral subdivision of $\R^J$.

On each $\mathfrak{C}$, every slack has a constant sign. Moreover, in view of \autoref{lem:sec-8-vertex-simplicity-criterion}, every $q\in\mathfrak{C}\cap\mathcal{U}_{\mathrm{irr}}$ represents a simple polytope, and the label sets of the facets incident to each vertex are independent of $q$. Since the prescribed normals are fixed, the normal fan is also independent of $q$. Hence every support cell meeting $\mathcal{U}_{\mathrm{irr}}$ corresponds to a unique simple type chamber $\sigma_{\mathfrak{C}}$, with
\eq{
  \mathfrak{C}\cap\mathcal{U}_{\mathrm{irr}}\subseteq\sigma_{\mathfrak{C}},\quad
  \overline{\mathfrak{C}}\cap\mathcal{U}_{\mathrm{irr}}\subseteq\overline{\sigma_{\mathfrak{C}}}.
}
Distinct support cells may correspond to the same type chamber.

For every $h\in\mathcal{U}_{\mathrm{irr}}$, we can choose an open ball $\mathcal{U}$ centered at $h$, with $\overline{\mathcal{U}}\subseteq\mathcal{U}_{\mathrm{irr}}$, that avoids every support-cell closure not containing $h$. Thus
\eq{\label{eq:sec-9-local-incident-chambers}
  \overline{\mathfrak{C}}\cap \mathcal{U}\neq\varnothing\ \Rightarrow\ h\in\overline{\mathfrak{C}},\quad
  \mathcal{U}\subseteq\bigcup_{\mathfrak{C}:\,h\in\overline{\mathfrak{C}}}\overline{\mathfrak{C}}.
}
In particular, for every $q\in \mathcal{U}$, $h$ and $q$ lie in the closure of a common simple type chamber.

For antipodally indexed normals, the same construction applies in $\R^I$. As in \autoref{def:sec-5-even-facet-functions}, let $\mathcal{E}^+\subseteq\R^I$ denote the space of even facet functions. We say a simple type chamber $\sigma$ is \emph{symmetric} if $\iota\sigma=\sigma$. 

Note that if $\sigma$ contains an even support vector, then $\iota\sigma=\sigma$: the normal fan $\Sigma$ of an even vector in $\sigma$ satisfies $-\Sigma=\Sigma$, and $\mathcal{F}_{P(\iota p)}=\mathcal{F}_{-P(p)}=-\Sigma=\Sigma$ for every $p\in\sigma$. Thus $\iota p\in\sigma$.

Suppose that $h$ is even and $\sigma$ is symmetric and incident to $h$. Choose $p\in\sigma\cap\mathcal{U}$. Since the ball $\mathcal{U}$ is invariant under $\iota$, $\iota p\in\sigma\cap\mathcal{U}$. Convexity then implies $(p+\iota p)/2\in\sigma\cap\mathcal{U}\cap\mathcal{E}^+$. Thus
\eq{\label{eq:sec-8-even-incident-neighborhood}
  \sigma\cap\mathcal{U}\cap\mathcal{E}^+
  \text{ is nonempty and relatively open in }\mathcal{E}^+.
}

For even $h$, the neighborhood $\mathcal{U}$ also satisfies
\eq{\label{eq:sec-8-symmetric-local-chamber-covering}
  \mathcal{U}\cap\mathcal{E}^+
  \subseteq\bigcup_{\substack{\sigma:\,h\in\overline{\sigma}\\
      \iota\sigma=\sigma}}\overline{\sigma}.
}
Indeed, for every $\tilde{J}\in\mathcal{B}$ and $k\notin\tilde{J}$,
\eq{
  \delta_{\tilde{J},k}(\mathbf{e}_k+\mathbf{e}_{-k})
  =\begin{cases}
    1,&-k\notin\tilde{J},\\
    2,&-k\in\tilde{J}.
  \end{cases}
}
Hence each slack restricts to a nonzero linear functional on $\mathcal{E}^+$.  Therefore, every $q\in\mathcal{U}\cap\mathcal{E}^+$ can be approximated by even vectors $q^{(\nu)}\in\mathcal{U}$ at which all slacks are nonzero. Finiteness of the support cells allows us to pass to a subsequence contained in one support cell $\mathfrak{C}$. By \eqref{eq:sec-9-local-incident-chambers}, both $h$ and $q$ lie in $\overline{\mathfrak{C}}$, and hence in $\overline{\sigma_{\mathfrak{C}}}$. Since the chamber $\sigma_{\mathfrak{C}}$ contains the even vectors $q^{(\nu)}$, it is symmetric.

We now return to the general $J$-indexed normal data. We also need continuity of facet and ridge measures for the explicit derivative formulas in \autoref{prop:sec-8-support-derivatives}. 

\begin{proposition}[Support derivatives]\label{prop:sec-8-support-derivatives}
For $h\in\mathcal{U}_{\mathrm{irr}}$ and $i\neq j$,
\eq{
  \frac{\partial V}{\partial q_i}(h)
  &=\mathcal{H}^{n-1}(F_i),\\
  \frac{\partial^2V}{\partial q_i\partial q_j}(h)
  &=\begin{cases}
    \dfrac{\mathcal{H}^{n-2}(F_i\cap F_j)}{\sin\theta_{ij}},
      & \dim(F_i\cap F_j)=n-2,\\[1.1em]
    0, & \text{otherwise}.
  \end{cases}
}
Here $F_i$ and $F_j$ are the corresponding facets of $P(h)$.
\end{proposition}

\begin{proof}
\autoref{thm:sec-9-C2-regularity} establishes that $V\in C^2(\mathcal{U}_{\mathrm{irr}})$. Let $h\in\mathcal{U}_{\mathrm{irr}}$. By \autoref{lem:sec-8-simple-chamber-approximation}, some simple type chamber $\sigma=\mathcal{T}_{\Sigma}$ is incident to $h$ and contains a sequence
\eq{
  h^{(\nu)}\in\sigma\cap\mathcal{U}_{\mathrm{irr}},
  \quad
  h^{(\nu)}\to h.
}
Denote the chamber volume polynomial of $\sigma$ by $\widehat{V}_{\sigma}=V_{\Sigma}$. Since $V=\widehat{V}_{\sigma}$ on $\sigma\cap\mathcal{U}_{\mathrm{irr}}$,
\eq{
  D^rV(h^{(\nu)})=D^r\widehat{V}_{\sigma}(h^{(\nu)})
  \quad(0\leq r\leq 2).
}
The $C^2$ regularity of $V$ and the continuity of the derivatives of $\widehat{V}_{\sigma}$ at $h$ imply
\eq{\label{eq:sec-8-incident-derivative-agreement}
  D^rV(h)
  =\lim_{\nu\to\infty}D^rV(h^{(\nu)})
  =\lim_{\nu\to\infty}D^r\widehat{V}_{\sigma}(h^{(\nu)})
  =D^r\widehat{V}_{\sigma}(h)
  \quad(0\leq r\leq 2).
}

For each $\nu$, denote the corresponding facets and their intersections by
\eq{
  F_i^{(\nu)}&=P(h^{(\nu)})\cap\{x\in\R^n:\ip{x}{u_i}=h_i^{(\nu)}\},\\
  F_i&=P(h)\cap\{x\in\R^n:\ip{x}{u_i}=h_i\},\\
  R_{ij}^{(\nu)}&=F_i^{(\nu)}\cap F_j^{(\nu)}.
}
The convergence $h^{(\nu)}\to h$ implies
\eq{\label{eq:sec-8-ridge-volume-convergence}
  F_i^{(\nu)}&\to F_i
    \quad\text{in the Hausdorff metric},\\
  \mathcal{H}^{n-2}(R_{ij}^{(\nu)})
  &\to
    \begin{cases}
      \mathcal{H}^{n-2}(F_i\cap F_j),
        &\dim(F_i\cap F_j)=n-2,\\
      0, &\text{otherwise},
    \end{cases}
    \quad i\neq j.
}

For every $\nu$, \autoref{lem:sec-4-volume-support-derivatives} applies to $\widehat{V}_{\sigma}$ at $h^{(\nu)}$. Taking $\nu\to\infty$ in its formulas and using \eqref{eq:sec-8-incident-derivative-agreement}--\eqref{eq:sec-8-ridge-volume-convergence} yields the two identities in the proposition.
\end{proof}

For $h\in\mathcal{U}_{\mathrm{irr}}$, denote the facet masses and ridge conductances of $P(h)$ by $m_i$ and $c_{ij}$. The intrinsic first two logarithmic support forms, mass forms, and energy form are
\eq{\label{eq:sec-8-irredundant-support-forms}
      S&=nV(h),\\
  A(a)&=DV(h)[h\odot a],\\
  B(a,b)&=D^2V(h)[h\odot a,h\odot b],\\
  \mathsf{M}(a,b)&=A(a\odot b),\\
  \widetilde{\mathsf{M}}(a,b)
  &=\mathsf{M}(a,b)-\frac{A(a)A(b)}{S},\\
  \mathsf{L}(a,b)&=(n-1)\mathsf{M}(a,b)-B(a,b).
}
By \autoref{prop:sec-8-support-derivatives}, the calculation leading to \autoref{prop:sec-4-graph-laplacian} remains valid on $\mathcal{U}_{\mathrm{irr}}$. Thus
\eq{\label{eq:sec-8-irredundant-network-formulas}
  A(a)=\sum_{i\in J}m_i a_i,\quad
  \mathsf{L}(a,b)=\sum_{\{i,j\}\in\binom{J}{2}}c_{ij}(a_i-a_j)(b_i-b_j).
}

For a $C^2_+$ convex body $K$ with $h_K>0$, the centro-affine metric $g_K$ and volume measure $dV_K$ determine the $L^2$ and Dirichlet forms
\eq{
  \mathsf{M}_K(\varphi,\psi)=\int_{\Sn}\varphi\psi\,dV_K,\quad
  \mathsf{L}_K(\varphi,\psi)=\int_{\Sn}\ip{\nabla\varphi}{\nabla\psi}_{g_K}\,dV_K.
}

For $P=P(h)$ with $h\in\mathcal{U}_{\mathrm{irr}}$, \eqref{eq:sec-8-irredundant-network-formulas} shows that $\mathsf{M}(a,b)=\sum_{i\in J}m_ia_ib_i$ is the discrete counterpart of $\mathsf{M}_K$, while $\mathsf{L}(a,b)=\sum_{\{i,j\}\in\binom{J}{2}}c_{ij}(a_i-a_j)(b_i-b_j)$ is the discrete counterpart of $\mathsf{L}_K$. Thus $\mathsf{M}$ corresponds to the $L^2(dV_K)$ inner product. The masses discretize $dV_K$, while the conductances discretize the combination $g_K^{-1}dV_K$ that enters the Dirichlet form. Consequently, the labeled facet network from \autoref{def:sec-2-labeled-facet-network} is the discrete counterpart of the metric-measure space $(\Sn,g_K,dV_K)$. 

The nonnegative weighted graph Laplacian, formed from the vertex weights $m_i$ and edge weights $c_{ij}$, is defined by
\eq{
  (\Delta_{m,c}a)_i=\frac{1}{m_i}\sum_{j\in J\setminus\{i\}}c_{ij}(a_i-a_j).
}
This convention is standard; see \cite[p. 15, Def. 0.15]{KLW21}. For all facet functions $a$ and $b$,
\eq{\label{eq:sec-8-weighted-graph-Laplacian}
  \mathsf{M}(\Delta_{m,c}a,b)
  =\sum_{i\in J}\sum_{j\in J\setminus\{i\}}c_{ij}(a_i-a_j)b_i
  =\mathsf{L}(a,b).
}
The form $\mathsf{M}$ is positive definite. Hence, for every $a\in\R^J$ and $\lambda\in\R$,
\eq{\label{eq:sec-8-weighted-graph-eigen-equation}
  \Delta_{m,c}a=\lambda a
  \quad\iff\quad
  \mathsf{L}(a,b)=\lambda\mathsf{M}(a,b)
  \quad\text{for every }b\in\R^J.
}

For the antipodally indexed normals, the same notation $V$, $A$, $B$, $S$, $\mathsf{M}$, $\widetilde{\mathsf{M}}$, $\mathsf{L}$, and $\Delta_{m,c}$ is used on $\mathcal{U}_{\mathrm{irr}}^I$. The preceding definitions and network formulas apply with indices in $I$. By \autoref{thm:sec-9-C2-regularity},
\eq{
  V\in C^2(\mathcal{U}_{\mathrm{irr}}^I).
}
For all $a,b\in\R^I$,
\eq{\label{eq:sec-8-signed-weighted-graph-Laplacian}
  \mathsf{M}(\Delta_{m,c}a,b)=\mathsf{L}(a,b).
}
For $a\in\R^I$ and $\lambda\in\R$,
\eq{\label{eq:sec-8-signed-weighted-graph-eigen-equation}
  \Delta_{m,c}a=\lambda a
  \quad\iff\quad
  \mathsf{L}(a,b)=\lambda\mathsf{M}(a,b)
  \quad\text{for every }b\in\R^I.
}
If $h$ is even, then $F_{-i}=-F_i$. Therefore, $ m_{-i}=m_i, c_{-i,-j}=c_{ij}$. For all $a,b\in\R^I$, these coefficient identities imply $\mathsf{M}(\iota a,\iota b)=\mathsf{M}(a,b)$, $\mathsf{L}(\iota a,\iota b)=\mathsf{L}(a,b)$, and $\Delta_{m,c}(\iota a)=\iota(\Delta_{m,c}a)$. Therefore, even and odd facet functions are orthogonal with respect to both forms, and $\Delta_{m,c}$ preserves the even and odd subspaces.

The simple-chamber versions of the next theorem and corollary were stated without proof in \autoref{thm:sec-5-support-hessian-kernel} and \autoref{cor:sec-5-strict-even-eigenvalue-bound}. We now prove the corresponding statements on the entire irredundant support domain.

\begin{theorem}[The support-Hessian inequality and its equality case]\label{thm:sec-8-irredundant-support-Hessian-inequality}
For $h\in\mathcal{U}_{\mathrm{irr}}$, every facet function $a$ satisfies the \emph{support-Hessian inequality}
\eq{\label{eq:sec-8-support-Hessian-inequality}
   B(a,a)\leq (n-1)\frac{A(a)^2}{S}.
}
Equality holds if and only if
\eq{\label{eq:sec-8-irredundant-AF-equality}
  a_i=c_0+\frac{\ip{v}{u_i}}{h_i}
  \quad(i\in J),
}
for some $c_0\in\R$ and $v\in\R^n$. Consequently, the matrix of $B$ in the logarithmic support coordinates has exactly one positive eigenvalue, and
\eq{
  \ker B=\{a\in\R^J:h\odot a\in\mathfrak{t}(\R^n)\}.
}
\end{theorem}

\begin{proof}
For $q_0,q_1\in\mathcal{U}_{\mathrm{irr}}$ and $0\leq t\leq 1$, we have
\eq{
  (1-t)P(q_0)+tP(q_1)
  \subseteq P((1-t)q_0+tq_1).
}
From the Brunn--Minkowski inequality and monotonicity of volume, it follows that
\eq{
  V((1-t)q_0+tq_1)^{\frac{1}{n}}
  \geq (1-t)V(q_0)^{\frac{1}{n}}+tV(q_1)^{\frac{1}{n}}.
}
That is, $V^{\frac{1}{n}}$ is concave on the open convex cone $\mathcal{U}_{\mathrm{irr}}$. For $z=h\odot a$, the definitions of $A$, $B$, and $S$ yield
\eq{
  0\geq D^2(V^{\frac{1}{n}})(h)[z,z]
  =\frac{V(h)^{\frac{1-n}{n}}}{n}
    \left(
      B(a,a)-(n-1)\frac{A(a)^2}{S}
    \right).
}

We next characterize equality. Suppose that equality holds in \eqref{eq:sec-8-support-Hessian-inequality}, and set $z=h\odot a$. Choose $\mathcal{U}$ as in \eqref{eq:sec-9-local-incident-chambers} and $\varepsilon>0$ small enough that
\eq{
  q=h+\varepsilon z\in\mathcal{U},
  \quad
  P=P(h),
  \quad
  Q=P(q).
}
Then $h$ and $q$ lie in the closure of a common simple type chamber $\sigma$. By \autoref{lem:sec-3-type-chamber-Minkowski-addition},
\eq{\label{eq:sec-8-irredundant-Minkowski-segment}
  P((1-t)h+tq)=(1-t)P+tQ
  \quad(0\leq t\leq 1).
}

Since $(1-t)h+tq=h+t\varepsilon z$, taking volumes in \eqref{eq:sec-8-irredundant-Minkowski-segment} yields
\eq{
  V(h+t\varepsilon z)=\Vol((1-t)P+tQ)
  \quad(0\leq t\leq 1).
}
The mixed-volume expansion of the right-hand side is
\eq{
  \Vol((1-t)P+tQ)
  =\sum_{k=0}^n\binom{n}{k}(1-t)^{n-k}t^k
    \mathcal{V}(P[n-k],Q[k]).
}
At $t=0$, only the terms with $k\leq 2$ contribute to its first two derivatives. With
\eq{
  V_0=\Vol(P),\quad
  V_1=\mathcal{V}(P[n-1],Q),\quad
  V_2=\mathcal{V}(P[n-2],Q[2]),
}
these derivatives are
\eq{
  \left.\frac{d}{dt}\right|_{t=0}\Vol((1-t)P+tQ)
  &=n(V_1-V_0),\\
  \left.\frac{d^2}{dt^2}\right|_{t=0}\Vol((1-t)P+tQ)
  &=n(n-1)(V_2-2V_1+V_0).
}
Since $z=h\odot a$,
\eq{
  \left.\frac{d}{dt}\right|_{t=0}V(h+t\varepsilon z)
  =\varepsilon A(a),\quad
  \left.\frac{d^2}{dt^2}\right|_{t=0}V(h+t\varepsilon z)
  =\varepsilon^2 B(a,a).
}
Combining these identities with the mixed-volume derivatives above, we obtain
\eq{
  A(a)=\frac{n(V_1-V_0)}{\varepsilon},
  \quad
  B(a,a)=\frac{n(n-1)(V_2-2V_1+V_0)}{\varepsilon^2}.
}
Since $S=nV_0$, the preceding identities imply
\eq{
  (n-1)\frac{A(a)^2}{S}-B(a,a)
  =\frac{n(n-1)}{V_0\varepsilon^2}
    (V_1^2-V_0V_2).
}
Thus $V_1^2=V_0V_2$. By \autoref{lem:sec-5-common-normal-Minkowski-equality}, $Q=\gamma P+v$, for some $\gamma>0$ and $v\in\R^n$. Comparison of the prescribed support values yields
\eq{
  h_i+\varepsilon h_i a_i
  =\gamma h_i+\ip{v}{u_i}
  \quad(i\in J).
}
Solving these equations for $a_i$, we find
\eq{
  a_i=\frac{\gamma-1}{\varepsilon}+\frac{\ip{\frac{v}{\varepsilon}}{u_i}}{h_i}.
}

Conversely, suppose that \eqref{eq:sec-8-irredundant-AF-equality} holds. Then $h\odot(a-c_0\one)=\mathfrak{t}(v)$, and translation invariance of volume implies
\eq{
  A(a-c_0\one)&=0,\\
  B(a-c_0\one,b)&=0
  \quad(b\in\R^J).
}
Together with the Euler identities, these relations imply
\eq{
  A(a)=c_0S,
  \quad
  B(a,a)=c_0^2(n-1)S.
}
Hence equality holds.

We now determine the kernel of $B$. If $h\odot a=\mathfrak{t}(v)$, translation invariance of volume implies
\eq{
  B(a,b)=D^2V(h)[\mathfrak{t}(v),h\odot b]=0
  \quad(b\in\R^J).
}
Every such facet function belongs to the kernel:
\eq{
  \{a\in\R^J:h\odot a\in\mathfrak{t}(\R^n)\}
  \subseteq\ker B.
}
For the reverse inclusion, suppose that $a\in\ker B$, meaning that $B(a,b)=0$ for every $b\in\R^J$. Taking $b=\one$ and using the Euler identity $B(a,\one)=(n-1)A(a)$ shows that $A(a)=0$, while taking $b=a$ shows that $B(a,a)=0$. Hence equality holds in \eqref{eq:sec-8-support-Hessian-inequality}, and \eqref{eq:sec-8-irredundant-AF-equality} holds for some $c_0\in\R$ and $v\in\R^n$. Here $h\odot(a-c_0\one)=\mathfrak{t}(v)$, and translation invariance implies $A(a-c_0\one)=0$. Since $A(\one)=S>0$, we have $c_0=0$. Hence $h\odot a=\mathfrak{t}(v)$.

For the centered version $a^{\circ}$ defined in \eqref{eq:sec-4-centered-representative}, the Euler identities imply
\eq{
  B(a^{\circ},a^{\circ})
  =B(a,a)-(n-1)\frac{A(a)^2}{S}.
}
Thus \eqref{eq:sec-8-support-Hessian-inequality} is equivalent to $B(a^{\circ},a^{\circ})\leq 0$. Since $A(\one)=S>0$, we have
\eq{
  \R^J=\R\one\oplus\ker A.
}
This decomposition is $B$-orthogonal because
\eq{
  B(\one,b)=(n-1)A(b)=0
  \quad(b\in\ker A).
}
Therefore, $B$ has exactly one positive eigenvalue.
\end{proof}

The support-Hessian inequality has an immediate spectral consequence.

\begin{corollary}\label{cor:sec-8-irredundant-strict-even-eigenvalue-bound}
For the antipodally indexed normals, let $h\in\mathcal{U}_{\mathrm{irr}}^I$ satisfy $h_{-i}=h_i$ for every $i\in I^+$. Then the full-dimensional origin-symmetric polytope $P(h)$ satisfies
\eq{
  \lambda_{1,e}(P(h))>n-1.
}
\end{corollary}

\begin{proof}
First note that equality in \eqref{eq:sec-8-support-Hessian-inequality} cannot hold for a nonzero centered even facet function. Now let $a$ be a nonzero centered even facet function. Then \eqref{eq:sec-8-support-Hessian-inequality} implies
\eq{
  \mathsf{L}(a,a)=(n-1)A(a^{\odot2})-B(a,a)
  >(n-1)A(a^{\odot2}).
}
The normalized centered even set
\eq{
  \{a:a\text{ is even},\ A(a)=0,\ A(a^{\odot2})=1\}
}
is compact. Therefore, $\lambda_{1,e}(P(h))$ is strictly larger than $n-1$.
\end{proof}

\subsection{Upper semicontinuity under polyhedral approximation}
\label{sec:upper-semicontinuity}

We now apply the intrinsic support forms to polyhedral approximation; the proof of the main theorem resumes in \autoref{sec:mixed-volume-support-regularity}.

For a convex body $M$, let $S_M$ denote its surface-area measure. If a polytope $P$ has facets $(F_i)_{i\in J}$ with outer unit normals $(u_i)_{i\in J}$, then
\eq{\label{eq:sec-8-polytope-surface-area-measure}
  S_P=\sum_{i\in J}\mathcal{H}^{n-1}(F_i)\delta_{u_i}.
}
If $K$ is a $C^2_+$ convex body, then
\eq{\label{eq:sec-8-smooth-surface-area-measure}
  dS_K=\det_{\bar{g}}(r_K)\,d\omega,\quad dV_K=h_K\,dS_K.
}
With the mixed-volume notation from \eqref{eq:sec-5-mixed-volume-expansion}, the first mixed-volume formula for convex bodies $M$ and $L$ is
\eq{\label{eq:sec-8-first-mixed-volume}
  n\mathcal{V}(M[n-1],L)=\int_{\Sn}h_L\,dS_M.
}
See \cite[p. 280, Thm. 5.1.7]{Sch14}.

\begin{lemma}\label{lem:sec-8-smooth-Dirichlet-identity}
Let $K$ and $L$ be $C^2_+$ convex bodies and $h_K>0$. Then
\eq{\label{eq:sec-8-smooth-Dirichlet-identity}
  (n-1)\int_{\Sn}\frac{h_L^2}{h_K}\,dS_K
  -n(n-1)\mathcal{V}(K[n-2],L[2])
  =\int_{\Sn}\left|\nabla\frac{h_L}{h_K}\right|_{g_K}^2\,dV_K.
}
\end{lemma}

\begin{proof}
The radius-of-curvature tensor of $K+tL$ is $r_K+t r_L$. By \eqref{eq:sec-8-smooth-surface-area-measure},
\eq{
  \left.\frac{d}{dt}\right|_{t=0^+}dS_{K+tL}
  =\tr\bigl(r_K^{-1}r_L\bigr)\,dS_K.
}
Differentiating \eqref{eq:sec-8-first-mixed-volume} with $M=K+tL$ at $t=0^+$, we obtain
\eq{
  n(n-1)\mathcal{V}(K[n-2],L[2])=\int_{\Sn}h_L\tr\bigl(r_K^{-1}r_L\bigr)\,dS_K=\int_{\Sn}\frac{h_L}{h_K}\tr\bigl(r_K^{-1}r_L\bigr)\,dV_K.
}
By the integration-by-parts formula following \cite[Eq. (2.6)]{Mil25},
\eq{
  \int_{\Sn}\frac{h_L}{h_K}\tr\bigl(r_K^{-1}r_L\bigr)\,dV_K
  =(n-1)\int_{\Sn}\frac{h_L^2}{h_K^2}\,dV_K
  -\int_{\Sn}\left|\nabla\frac{h_L}{h_K}\right|_{g_K}^2\,dV_K.
}
\end{proof}

\begin{lemma}\label{lem:sec-8-fixed-normal-comparison}
Let $h\in\mathcal{U}_{\mathrm{irr}}$ and $P=P(h)$. For a convex body $L$, define $z,a\in\R^J$ by
\eq{
  z_i=h_L(u_i),\quad a_i=\frac{z_i}{h_i}
  \quad(i\in J).
}
Then
\eq{\label{eq:sec-8-fixed-normal-energy-comparison}
  \mathsf{L}(a,a)\leq(n-1)\int_{\Sn}\frac{h_L^2}{h_P}\,dS_P
  -n(n-1)\mathcal{V}(P[n-2],L[2]).
}
\end{lemma}

\begin{proof}
For $t\geq 0$,
\eq{
  h_{P+tL}(u_i)=h_i+t z_i
  \quad(i\in J).
}
By \autoref{lem:sec-8-irredundant-open}, $h+tz\in\mathcal{U}_{\mathrm{irr}}$ whenever $\abs{t}$ is sufficiently small. Therefore,
\eq{\label{eq:sec-8-fixed-normal-inclusion}
  P+tL\subseteq P(h+tz).
}
The $C^2$ regularity of $V$ established in \autoref{thm:sec-9-C2-regularity} ensures that $t\mapsto\Vol(P(h+tz))$ is $C^2$ near zero. Note that the first derivative formula in \autoref{prop:sec-8-support-derivatives} also follows from the Wulff-shape variational formula \cite[p. 412, Lem. 7.5.3]{Sch14}. Together with \eqref{eq:sec-8-first-mixed-volume} and the mixed-volume expansion \eqref{eq:sec-5-mixed-volume-expansion}, this yields
\eq{
  DV(h)[z]=\sum_{i\in J}h_L(u_i)\mathcal{H}^{n-1}(F_i)
  =n\mathcal{V}(P[n-1],L)
  =\left.\frac{d}{dt}\right|_{t=0^+}\Vol(P+tL).
}
Since $z=h\odot a$, Taylor's formula yields
\eq{
  \Vol(P(h+tz))
  =\Vol(P)+tDV(h)[z]+\frac{t^2}{2}B(a,a)+o(t^2).
}
Moreover, \eqref{eq:sec-5-mixed-volume-expansion} implies
\eq{
  \Vol(P+tL)
  &=\Vol(P)+nt\mathcal{V}(P[n-1],L)\\
  &\quad+\frac{n(n-1)t^2}{2}\mathcal{V}(P[n-2],L[2])+O(t^3).
}
In view of \eqref{eq:sec-8-fixed-normal-inclusion},
\eq{
  0&\leq\Vol(P(h+tz))-\Vol(P+tL)\\
  &=\frac{t^2}{2}\bigl(B(a,a)-n(n-1)\mathcal{V}(P[n-2],L[2])\bigr)+o(t^2)
  \quad(t\to 0^+).
}
Consequently,
\eq{\label{eq:sec-8-fixed-normal-Hessian-comparison}
  B(a,a)\geq n(n-1)\mathcal{V}(P[n-2],L[2]).
}
Combining \eqref{eq:sec-8-irredundant-support-forms}, \eqref{eq:sec-8-irredundant-network-formulas}, and \eqref{eq:sec-8-polytope-surface-area-measure}, we obtain
\eq{
  \mathsf{M}(a,a)=A(a^{\odot2})
  =\int_{\Sn}\frac{h_L^2}{h_P}\,dS_P.
}
Now using $\mathsf{L}=(n-1)\mathsf{M}-B$, we conclude that
\eq{
  \mathsf{L}(a,a)\leq(n-1)\int_{\Sn}\frac{h_L^2}{h_P}\,dS_P
  -n(n-1)\mathcal{V}(P[n-2],L[2]).
}
\end{proof}

\begin{remark}\label{rem:sec-8-Milman-upper-semicontinuity}
We compare $\lambda_{1,e}(P)$ with the mixed-volume spectral quantity $\lambda^C_{1,e}(P)$ introduced by Milman \cite[Sec. 1.4, Eq. (1.7)]{Mil24}. Milman's definition uses Putterman's mixed-volume formulation of the local $L^p$-Brunn--Minkowski inequality \cite[App. A]{Put21}.

Let $P=P(h)$ be a full-dimensional origin-symmetric polytope. For a full-dimensional origin-symmetric convex body $L$, set $a_i=h_L(u_i)/h_i$. By \eqref{eq:sec-8-polytope-surface-area-measure} and \eqref{eq:sec-8-first-mixed-volume},
\eq{
  \widetilde{\mathsf{M}}(a,a)
  &=\mathsf{M}(a,a)-\frac{A(a)^2}{S}\\
  &=\int_{\Sn}\frac{h_L^2}{h_P}\,dS_P
  -\frac{n\mathcal{V}(P[n-1],L)^2}{\Vol(P)}.
}
The admissibility condition in Milman's definition is precisely that $a$ is nonconstant. For every admissible $L$, \eqref{eq:sec-8-fixed-normal-energy-comparison} implies
\eq{
  \lambda_{1,e}(P)
  \leq\frac{\mathsf{L}(a,a)}{\widetilde{\mathsf{M}}(a,a)}\leq\frac{(n-1)\displaystyle\int_{\Sn}\frac{h_L^2}{h_P}\,dS_P
  -n(n-1)\mathcal{V}(P[n-2],L[2])}
  {\int_{\Sn}\frac{h_L^2}{h_P}\,dS_P
  -\frac{n\mathcal{V}(P[n-1],L)^2}{\Vol(P)}}.
}
Taking the infimum of the last quotient over admissible $L$, we obtain
\eq{
  \lambda_{1,e}(P)\leq\lambda^C_{1,e}(P).
}
\end{remark}

The upper semicontinuity of $\lambda^C_{1,e}$ and its agreement with $\lambda_{1,e}$ on $C^2_+$ bodies \cite[Sec. 1.4]{Mil24}, together with the comparison above, establish upper semicontinuity of $\lambda_{1,e}$ under the assumptions below. For completeness, we provide a direct proof of the following theorem.

\begin{theorem}[Upper semicontinuity under polyhedral approximation]\label{thm:sec-8-upper-semicontinuity}
Suppose that the full-dimensional origin-symmetric polytopes $P_{\nu}$ converge in the Hausdorff metric to an origin-symmetric $C^2_+$ convex body $K$. Then
\eq{\label{eq:sec-8-upper-semicontinuity}
  \limsup_{\nu\to\infty}\lambda_{1,e}(P_{\nu})\leq\lambda_{1,e}(K).
}
\end{theorem}

\begin{proof}
Let $\psi\in C^2(\Sn)$ be a nonzero even function satisfying
\eq{
  \int_{\Sn}\psi\,dV_K=0.
}
Choose $c_0>0$ sufficiently large that
\eq{
  h_K(c_0+\psi)> 0,\quad
  r_{h_K(c_0+\psi)}=c_0 r_K+r_{h_K\psi}
  \quad\text{is positive definite}.
}
Therefore, $h_K(c_0+\psi)$ is the support function of a $C^2_+$ origin-symmetric convex body $L$.

For each $\nu$, we index the facets of $P_{\nu}$ by a signed set $I_{\nu}$, with outer unit normals satisfying $u_{-i}^{(\nu)}=-u_i^{(\nu)}$. Define
\eq{
  a_i^{(\nu)}=\frac{h_L(u_i^{(\nu)})}{h_{P_{\nu}}(u_i^{(\nu)})}
  \quad(i\in I_{\nu}).
}
Since $h_L$ and $h_{P_{\nu}}$ are even, $a^{(\nu)}$ is an even facet function. Let $\widetilde{\mathsf{M}}_{\nu}$ and $\mathsf{L}_{\nu}$ be the centered mass and energy forms of $P_{\nu}$. By \eqref{eq:sec-8-irredundant-support-forms}, \eqref{eq:sec-8-irredundant-network-formulas}, \eqref{eq:sec-8-first-mixed-volume}, and \eqref{eq:sec-8-fixed-normal-energy-comparison}, the centered mass and energy satisfy
\eq{\label{eq:sec-8-approximating-mass-energy}
  \widetilde{\mathsf{M}}_{\nu}(a^{(\nu)},a^{(\nu)})&=\int_{\Sn}\frac{h_L^2}{h_{P_{\nu}}}\,dS_{P_{\nu}}
  -\frac{n\mathcal{V}(P_{\nu}[n-1],L)^2}{\Vol(P_{\nu})},\\
  \mathsf{L}_{\nu}(a^{(\nu)},a^{(\nu)})&\leq(n-1)\int_{\Sn}\frac{h_L^2}{h_{P_{\nu}}}\,dS_{P_{\nu}}
  -n(n-1)\mathcal{V}(P_{\nu}[n-2],L[2]).
}
Under Hausdorff convergence, $h_{P_{\nu}}\to h_K>0$ uniformly and $S_{P_{\nu}}\rightharpoonup S_K$. Consequently,
\eq{
  \frac{h_L^2}{h_{P_{\nu}}}&\to\frac{h_L^2}{h_K}
  \quad\text{uniformly on }\Sn,\\
  \int_{\Sn}\frac{h_L^2}{h_{P_{\nu}}}\,dS_{P_{\nu}}
  &\to\int_{\Sn}\frac{h_L^2}{h_K}\,dS_K.
}
Since mixed volumes are continuous under Hausdorff convergence, passing to the limit in the first relation of \eqref{eq:sec-8-approximating-mass-energy} and using $dV_K=h_K\,dS_K$ and $h_L/h_K=c_0+\psi$, we obtain
\eq{\label{eq:sec-8-mass-limit}
  \lim_{\nu\to\infty}\widetilde{\mathsf{M}}_{\nu}(a^{(\nu)},a^{(\nu)})=\int_{\Sn}(c_0+\psi)^2\,dV_K
  -\frac{\left(\int_{\Sn}(c_0+\psi)\,dV_K\right)^2}
        {\int_{\Sn}1\,dV_K}=\int_{\Sn}\psi^2\,dV_K.
}
Similarly, passing to the $\limsup$ in the second relation of \eqref{eq:sec-8-approximating-mass-energy} and using \autoref{lem:sec-8-smooth-Dirichlet-identity},
\eq{\label{eq:sec-8-energy-limit}
  \limsup_{\nu\to\infty}\mathsf{L}_{\nu}(a^{(\nu)},a^{(\nu)})\leq\int_{\Sn}\abs{\nabla\psi}_{g_K}^2\,dV_K.
}
By \eqref{eq:sec-8-mass-limit}, the centered mass of $a^{(\nu)}$ is positive for all sufficiently large $\nu$. The Rayleigh characterization \eqref{eq:introduction-Rayleigh-characterization}, together with \eqref{eq:sec-8-mass-limit} and \eqref{eq:sec-8-energy-limit}, implies
\eq{
  \limsup_{\nu\to\infty}\lambda_{1,e}(P_{\nu})
  \leq\limsup_{\nu\to\infty}
  \frac{\mathsf{L}_{\nu}(a^{(\nu)},a^{(\nu)})}
       {\widetilde{\mathsf{M}}_{\nu}(a^{(\nu)},a^{(\nu)})}
       \leq\frac{\displaystyle\int_{\Sn}\abs{\nabla\psi}_{g_K}^2\,dV_K}
       {\displaystyle\int_{\Sn}\psi^2\,dV_K}.
}
Taking the infimum over $\psi$ in the smooth Rayleigh characterization \eqref{eq:introduction-smooth-Rayleigh-characterization} proves the claim.
\end{proof}

\subsection{Regularity of the volume}
\label{sec:mixed-volume-support-regularity}

\begin{lemma}\label{lem:sec-9-linear-mixed-volume}
Suppose that $h,r\in\mathcal{U}_{\mathrm{irr}}$ belong to the closure of a common simple type chamber. For $0\leq k\leq n-1$, the function
\eq{
  q\mapsto\mathcal{V}(P(q),P(h)[n-1-k],P(r)[k])
  \quad(q\in\mathcal{U}_{\mathrm{irr}})
}
is the restriction of a linear functional on $\R^J$.
\end{lemma}

\begin{proof}
By \autoref{lem:sec-3-type-chamber-Minkowski-addition},
\eq{
  P(h)+sP(r)=P(h+sr)\quad(s\geq 0).
}
Every facet normal of this polytope is among the prescribed normals. From the first mixed-volume formula \eqref{eq:sec-8-first-mixed-volume} and Minkowski multilinearity, it follows that
\eq{
  &\sum_{k=0}^{n-1}\binom{n-1}{k}s^k
    \mathcal{V}(P(q),P(h)[n-1-k],P(r)[k])\\
  &\quad=\mathcal{V}(P(q),(P(h)+sP(r))[n-1])\\
  &\quad=\frac1n\sum_{i\in J}q_i\mathcal{H}^{n-1}(P(h+sr)^{u_i})
  \quad(s\geq 0).
}
For every fixed $s$, the last expression is linear in $q$. The coefficients of a polynomial of degree at most $n-1$ are linear combinations of its values at $n$ distinct nonnegative numbers. Hence every coefficient on the first line is linear in $q$.
\end{proof}

\begin{theorem}\label{thm:sec-9-C2-regularity}
The volume function belongs to $C^2(\mathcal{U}_{\mathrm{irr}})$.
\end{theorem}

\begin{proof}
For $h\in\mathcal{U}_{\mathrm{irr}}$, we choose $\mathcal{U}$ as in \eqref{eq:sec-9-local-incident-chambers} and define
\eq{
  \mathcal{Q}_h(q,r)=n(n-1)\mathcal{V}(P(q),P(r),P(h)[n-2])
  \quad(q,r\in \mathcal{U}).
}
For each fixed $r\in\mathcal{U}$, the vectors $h$ and $r$ lie in the closure of a common simple type chamber. Thus \autoref{lem:sec-9-linear-mixed-volume} shows that $q\mapsto\mathcal{Q}_h(q,r)$ is the restriction of a linear functional on $\R^J$. By symmetry, $r\mapsto\mathcal{Q}_h(q,r)$ has the same property for each fixed $q\in\mathcal{U}$. Since $\mathcal{U}$ contains a basis of $\R^J$, the map $\mathcal{Q}_h:\mathcal{U}\times\mathcal{U}\to\R$ extends uniquely to a symmetric bilinear form on $\R^J\times\R^J$, still denoted by $\mathcal{Q}_h$.

We first compare the derivatives of the incident chamber volume polynomials at $h$. If $\sigma$ is incident to $h$ and $q\in\sigma\cap \mathcal{U}$, then $h+sq\in\sigma$ for $s>0$. By \autoref{lem:sec-3-type-chamber-Minkowski-addition},
\eq{
  \widehat{V}_{\sigma}(h+sq)=\Vol(P(h)+sP(q))
  =\sum_{k=0}^n\binom nk s^k\mathcal{V}(P(q)[k],P(h)[n-k]).
}
Comparing the constant, linear, and quadratic coefficients, we obtain
\eq{
  \widehat{V}_{\sigma}(h)&=V(h),\\
  D\widehat{V}_{\sigma}(h)[q]&=\sum_{i\in J}q_i\mathcal{H}^{n-1}(P(h)^{u_i})
    =:\mathcal{L}_h[q],\\
  D^2\widehat{V}_{\sigma}(h)[q,q]&=\mathcal{Q}_h(q,q).
}
The set $\sigma\cap\mathcal{U}$ is nonempty and open. The polynomial identity principle and polarization show that
\eq{\label{eq:sec-9-common-chamber-derivatives}
  \widehat{V}_{\sigma}(h)=V(h),\quad
  D\widehat{V}_{\sigma}(h)=\mathcal{L}_h,\quad
  D^2\widehat{V}_{\sigma}(h)=\mathcal{Q}_h
}
for every chamber incident to $h$.

We now verify differentiability of $V$. For sufficiently small $z$, some chamber $\sigma_z$ is incident to both $h$ and $h+z$. Applying Taylor's formula and using \eqref{eq:sec-9-common-chamber-derivatives}, we obtain
\eq{
  V(h+z)=\widehat{V}_{\sigma_z}(h+z)
  =V(h)+\mathcal{L}_h[z]+\frac12\mathcal{Q}_h(z,z)+O(\lvert z\rvert^3).
}
The remainder is uniform because only finitely many chamber polynomials occur. Thus $DV(h)=\mathcal{L}_h$. Since $h$ was arbitrary, $V$ is differentiable on $\mathcal{U}_{\mathrm{irr}}$, and its derivative agrees with the derivative of every incident branch.

Similarly, Taylor's formula for the first derivatives reads
\eq{
  DV(h+z)=D\widehat{V}_{\sigma_z}(h+z)
  =\mathcal{L}_h+\mathcal{Q}_h(z,\mathord{\cdot})+O(\lvert z\rvert^2).
}
The remainder is measured in the norm on $(\R^J)^*$, with a constant independent of $\sigma_z$. Since $DV(h)=\mathcal{L}_h$, there is a constant $C>0$ such that
\eq{
  \frac{\norm{DV(h+z)-DV(h)-\mathcal{Q}_h(z,\mathord{\cdot})}}{\lvert z\rvert}\leq C\lvert z\rvert\to 0\quad(z\to 0,\ z\neq 0).
}
Hence $z\mapsto\mathcal{Q}_h(z,\mathord{\cdot})$ is the Fr\'echet derivative of $DV$ at $h$. In particular, $D^2V(h)=\mathcal{Q}_h$. Since $h$ was arbitrary, $V$ is twice differentiable on $\mathcal{U}_{\mathrm{irr}}$. By \eqref{eq:sec-9-common-chamber-derivatives},
\eq{\label{eq:sec-9-incident-Hessian-agreement}
  D^2V(q)=\mathcal{Q}_q=D^2\widehat{V}_{\sigma}(q)\quad(q\in\mathcal{U}_{\mathrm{irr}}\cap\overline{\sigma}).
}

Finally, for $h+z\in\mathcal{U}$, we choose a simple type chamber $\sigma_z$ incident to both $h$ and $h+z$, as in \eqref{eq:sec-9-local-incident-chambers}. By \eqref{eq:sec-9-incident-Hessian-agreement},
\eq{
  \norm{D^2V(h+z)-D^2V(h)}
  &=\norm{D^2\widehat{V}_{\sigma_z}(h+z)-D^2\widehat{V}_{\sigma_z}(h)}\\
  &\leq\max_{\sigma:\,h\in\overline{\sigma}}
    \norm{D^2\widehat{V}_{\sigma}(h+z)-D^2\widehat{V}_{\sigma}(h)}\to 0\quad(z\to 0).
}
Therefore, $V\in C^2(\mathcal{U}_{\mathrm{irr}})$.
\end{proof}

\subsection{A midpoint inequality for the support Hessian}
\label{sec:support-Hessian-midpoint}

The next theorem compares the support Hessian in opposite directions. This inequality will be used for incident-chamber stationarity.

\begin{theorem}[Support-Hessian midpoint inequality]\label{thm:sec-9-support-Hessian-midpoint}
For every $h\in\mathcal{U}_{\mathrm{irr}}$ and $w,z\in\R^J$,
\eq{\label{eq:sec-9-support-Hessian-midpoint}
  \liminf_{t\to 0^+}
  \frac{D^2V(h+tz)[w,w]+D^2V(h-tz)[w,w]-2D^2V(h)[w,w]}{t}\geq 0.
}
\end{theorem}

\begin{proof}
For $t>0$ sufficiently small, write
\eq{
  \Delta_t(u,v)
  =D^2V(h+tz)[u,v]+D^2V(h-tz)[u,v]-2D^2V(h)[u,v].
}
We first use mixed volumes to bound $\Delta_t(h+\varepsilon w,h+\varepsilon w)$ from below for a fixed sufficiently small $\varepsilon>0$. We then apply bilinearity to obtain the required bound for $\Delta_t(w,w)$.

Let $\mathcal{U}$ be as in \eqref{eq:sec-9-local-incident-chambers}, and fix $\varepsilon>0$ such that $\tilde{h}=h+\varepsilon w\in \mathcal{U}$. Thus $h$ and $\tilde{h}$ belong to the closure of a common simple type chamber. We write $P=P(h)$, $L=P(\tilde{h})$, and
\eq{
  \Psi(q)=n(n-1)\mathcal{V}(P(q)[n-2],L[2])
  \quad(q\in \mathcal{U}).
}
Since $h_L(u_i)=\tilde{h}_i$, \eqref{eq:sec-8-fixed-normal-Hessian-comparison} implies
\eq{\label{eq:sec-9-Hessian-lower-comparison}
  D^2V(q)[\tilde{h},\tilde{h}]\geq\Psi(q).
}
At $q=h$, equality follows by differentiating $V(h+s\tilde{h})=\Vol(P+sL)$ twice at $s=0$, using \autoref{lem:sec-3-type-chamber-Minkowski-addition}:
\eq{\label{eq:sec-9-Hessian-comparison-contact}
  D^2V(h)[\tilde{h},\tilde{h}]=\Psi(h).
}

We next show that the first-order terms of $\Psi$ along the two rays $h\pm tz$ cancel. When $n=2$, this holds because $\Psi$ is constant. For $n\geq 3$, by \autoref{lem:sec-9-linear-mixed-volume} with $r=\tilde{h}$ and $k=2$, there is a linear functional $\ell$ on $\R^J$ such that
\eq{
  \ell(q)=n(n-1)(n-2)\mathcal{V}(P(q),P[n-3],L[2])
  \quad(q\in \mathcal{U}).
}
Fix $\tau>0$ small enough that $q_{\pm}=h\pm\tau z\in \mathcal{U}$. Each of $q_+$ and $q_-$ shares a simple chamber closure with $h$. Hence \autoref{lem:sec-3-type-chamber-Minkowski-addition} implies
\eq{
  P(h\pm tz)=\left(1-\frac{t}{\tau}\right)P+\frac{t}{\tau}P(q_{\pm})
  \quad(0\leq t\leq\tau).
}
By multilinearity,
\eq{
  \Psi(h\pm tz)
  &=\Psi(h)+\frac{t}{\tau}\bigl(\ell(q_{\pm})-\ell(h)\bigr)+O(t^2)\\
  &=\Psi(h)\pm t\ell(z)+O(t^2).
}
Thus, in every dimension $n\geq 2$, the sum $\Psi(h+tz)+\Psi(h-tz)-2\Psi(h)$ is $O(t^2)$. The lower bound \eqref{eq:sec-9-Hessian-lower-comparison} and equality \eqref{eq:sec-9-Hessian-comparison-contact} imply, for some $C\geq 0$,
\eq{\label{eq:sec-9-comparison-midpoint}
  \Delta_t(\tilde{h},\tilde{h})
  \geq\Psi(h+tz)+\Psi(h-tz)-2\Psi(h)\geq-Ct^2.
}

For every fixed $v\in\R^J$, Euler's identity and the $C^2$ regularity from \autoref{thm:sec-9-C2-regularity} imply
\eq{
  D^2V(h+sz)[h,v]
  &=(n-1)DV(h+sz)[v]-sD^2V(h+sz)[z,v]\\
  &=D^2V(h)[h,v]+(n-2)sD^2V(h)[z,v]+o(\lvert s\rvert).
}
Adding the expansions for $s=t$ and $s=-t$, we obtain
\eq{\label{eq:sec-9-radial-Hessian-midpoint}
  \Delta_t(h,v)=o(t)
  \quad(t\to 0^+).
}
Since $\tilde{h}=h+\varepsilon w$, we expand by bilinearity:
\eq{
  \varepsilon^2\Delta_t(w,w)
  =\Delta_t(\tilde{h},\tilde{h})-\Delta_t(h,h)-2\varepsilon\Delta_t(h,w)
  \geq-Ct^2+o(t).
}
Dividing by $\varepsilon^2t$ and taking the lower limit proves the claim.
\end{proof}

\section{Incident-chamber stationarity and second variation}
\label{sec:incident-chamber-stationarity-and-second-variation}

In this section, we extend the fixed-chamber variation argument from \autoref{sec:simple-chamber-Rayleigh-variation} to a nonsimple local minimizer in the origin-symmetric class. For a local minimizer $h$ and a prescribed even direction $z$, we first select a simple type chamber $\sigma$ such that $h+tz\in\overline{\sigma}$ for all sufficiently small $t\geq 0$. Local minimality and the support-Hessian midpoint inequality force the one-sided Rayleigh derivatives along $z$ and $-z$ to vanish. We then derive the incident-chamber second variation and complete the proof of \autoref{thm:main}.

We use the notation $\mathcal{U}_{\mathrm{irr}}^I$ and $\mathcal{E}^+$ from \autoref{sec:irredundant-support-domain}. Recall that a simple type chamber $\sigma$ is incident to $h$ if $h\in\overline{\sigma}$, and symmetric if $\iota\sigma=\sigma$. For $\sigma=\mathcal{T}_{\Sigma}$, write $\widehat{V}_{\sigma}=V_{\Sigma}$ for its chamber volume polynomial.

\begin{corollary}\label{cor:sec-10-symmetric-chamber-selection}
Let $h\in\mathcal{U}_{\mathrm{irr}}^I$ be even. For every even facet function $f$, there is a symmetric simple type chamber $\sigma$ incident to $h$ such that $ h+t(h\odot f)\in\overline{\sigma}$ for all sufficiently small $t\geq 0$.
\end{corollary}

\begin{proof}
Let $\mathcal{U}$ be as in \eqref{eq:sec-9-local-incident-chambers} and choose $\tau>0$ such that $q=h+\tau(h\odot f)\in\mathcal{U}$. By \eqref{eq:sec-8-symmetric-local-chamber-covering}, there is a symmetric simple type chamber $\sigma$ with $h,q\in\overline{\sigma}$. Convexity of $\overline{\sigma}$ implies $h+t(h\odot f)\in\overline{\sigma}$ for $0\leq t\leq\tau$.
\end{proof}

\begin{corollary}\label{cor:sec-10-signed-support-Hessian-regularity}
The volume function satisfies
\eq{
  V\in C^2(\mathcal{U}_{\mathrm{irr}}^I).
}
If $h\in\mathcal{U}_{\mathrm{irr}}^I$ and $\sigma$ is a simple type chamber incident to $h$, then
\eq{\label{eq:sec-10-signed-incident-derivative-agreement}
  D^r\widehat{V}_{\sigma}(h)=D^rV(h)
  \quad(0\leq r\leq 2).
}
Moreover, for every $h\in\mathcal{U}_{\mathrm{irr}}^I$ and $w,z\in\R^I$,
\eq{\label{eq:sec-10-signed-support-Hessian-midpoint}
  \liminf_{t\to 0^+}
  \frac{
    D^2V(h+tz)[w,w]+D^2V(h-tz)[w,w]-2D^2V(h)[w,w]
  }{t}
  \geq 0.
}
\end{corollary}

\begin{proof}
This is the antipodal version of \autoref{thm:sec-9-C2-regularity}, \eqref{eq:sec-9-common-chamber-derivatives}, \eqref{eq:sec-9-incident-Hessian-agreement}, and \autoref{thm:sec-9-support-Hessian-midpoint}.
\end{proof}

For $q\in\mathcal{U}_{\mathrm{irr}}^I$, define
\eq{
      S_q&=nV(q),\\
  A_q(a)&=DV(q)[q\odot a],\\
  B_q(a,b)&=D^2V(q)[q\odot a,q\odot b],\\
  \mathsf{M}_q(a,b)&=A_q(a\odot b),\\
  \widetilde{\mathsf{M}}_q(a,b)
  &=\mathsf{M}_q(a,b)-\frac{A_q(a)A_q(b)}{S_q},\\
  \mathsf{L}_q(a,b)&=(n-1)\mathsf{M}_q(a,b)-B_q(a,b).
}
At a fixed support vector $q=h$, we omit the subscript $h$. Recall from \autoref{def:sec-5-even-facet-functions} that
\eq{
  \mathcal{E}_0^+=\{a\in\mathcal{E}^+:A(a)=0\}.
}
If $\sigma$ is incident to $h$, then $\widehat{V}_{\sigma}$ and $V$ agree at $h$ through order two by \eqref{eq:sec-10-signed-incident-derivative-agreement}. At a nonsimple $h$, however, the third and fourth logarithmic support forms may depend on $\sigma$.

\subsection{Rayleigh stationarity in an incident chamber}
\label{sec:incident-chamber-Rayleigh-stationarity}

Let $h\in\mathcal{U}_{\mathrm{irr}}^I$ be even and suppose that $h$ is a local minimizer of $\lambda_{1,e}$ under even support variations; that is, $ \lambda_{1,e}(P(q))\geq \lambda_{1,e}(P(h))$ for every even $q\in\mathcal{U}_{\mathrm{irr}}^I$ sufficiently close to $h$. 

Choose a first nonconstant even eigenfunction $f$ normalized by
\eq{
  A(f)=0, \quad
  A(f^{\odot2})=1.
}
Write
\eq{
  \lambda=\lambda_{1,e}(P(h)), \quad
  \beta=n-1-\lambda.
}

Let $\mathcal{U}$ be an open ball centered at $h$, with $\overline{\mathcal{U}}\subseteq\mathcal{U}_{\mathrm{irr}}^I$, satisfying \eqref{eq:sec-9-local-incident-chambers} in $\R^I$. After shrinking $\mathcal{U}$, we may assume that $\lambda_{1,e}(P(q))\geq\lambda$ for every $q\in\mathcal{U}\cap\mathcal{E}^+$ and that $\widetilde{\mathsf{M}}_q(f,f)>0$ for every $q\in\mathcal{U}$.

We now derive the chamberwise Euler--Lagrange identity needed for the second variation. For a simple type chamber $\sigma$ incident to $h$, define the branchwise third and fourth logarithmic support forms by
\eq{\label{eq:sec-10-incident-higher-forms}
  C_{\sigma}(a,b,c)
  &=D^3\widehat{V}_{\sigma}(h)
    [h\odot a,h\odot b,h\odot c],\\
  D_{\sigma}(a,b,c,d)
  &=D^4\widehat{V}_{\sigma}(h)
    [h\odot a,h\odot b,h\odot c,h\odot d].
}
By Euler's identities and \eqref{eq:sec-10-signed-incident-derivative-agreement},
\eq{
  C_{\sigma}(\one,a,b)&=(n-2)B(a,b),\\
  D_{\sigma}(\one,a,b,c)&=(n-3)C_{\sigma}(a,b,c).
}

\begin{proposition}[Rayleigh stationarity in an incident chamber]\label{prop:sec-10-incident-chamber-Rayleigh-stationarity}
For every symmetric simple type chamber $\sigma$ incident to $h$ and every $g\in\mathcal{E}^+$,
\eq{\label{eq:sec-10-chamber-stationarity}
  C_{\sigma}(f,f,g)
  =\beta\bigl\{B(f^{\odot2},g)
    -A(f^{\odot2}\odot g)\bigr\}.
}
For every $g\in\mathcal{E}^+$, define $h_t=h\odot(\one+tg)$. Then
\eq{
  \left.\frac{d}{dt}\right|_{0^+}
  \frac{\mathsf{L}_{h_t}(f,f)}
  {\widetilde{\mathsf{M}}_{h_t}(f,f)}=0.
}
\end{proposition}

\begin{proof}
For $q\in\mathcal{U}\cap\mathcal{E}^+$, define
\eq{
  f_q=f-\frac{A_q(f)}{S_q}\one.
}
Then $A_q(f_q)=0$, and the centered Rayleigh quotient of $f_q$ at $q$ is
\eq{
  \mathcal{R}(q)
  =\frac{\mathsf{L}_q(f,f)}{\widetilde{\mathsf{M}}_q(f,f)}.
}
The Rayleigh principle and local minimality imply
\eq{\label{eq:sec-10-Rayleigh-sandwich}
  \mathcal{R}(q)
  &\geq \lambda_{1,e}(P(q))
  \geq \lambda
  =\mathcal{R}(h).
}

Let $\sigma$ be a symmetric simple type chamber incident to $h$. In view of \eqref{eq:sec-8-even-incident-neighborhood}, the set $\mathcal{U}\cap\sigma\cap\mathcal{E}^+$ is nonempty and relatively open in $\mathcal{E}^+$. Consider the set of points $q$ in this intersection satisfying
\eq{\label{eq:sec-10-reflected-support-slacks}
\delta_{\tilde{J},k}(2h-q)\neq 0
\quad(\tilde{J}\in\mathcal{B},\ k\notin\tilde{J}).
}
Since each slack restricts to a nonzero linear functional on $\mathcal{E}^+$, this set is nonempty and relatively open. For $q$ in this set, define
\eq{
  z=q-h=h\odot a,
  \quad
  a_i=\frac{q_i-h_i}{h_i}\quad(i\in I).
}
Since $h\in\overline{\sigma}$ and $\sigma$ is open and convex, $h+tz=(1-t)h+tq\in\sigma\cap\mathcal{U}$ for $0<t\leq 1$.

Also note that the reflected support vector $q_-=2h-q$ belongs to $\mathcal{U}$ and, by \eqref{eq:sec-10-reflected-support-slacks}, to a support cell $\mathfrak{C}$. Let $\tilde{\sigma}=\sigma_{\mathfrak{C}}$ denote the chamber that contains $q_-$. By \eqref{eq:sec-9-local-incident-chambers}, $h\in\overline{\mathfrak{C}}\cap\mathcal{U}$, and hence $h\in\overline{\tilde{\sigma}}$. Convexity implies $h-tz=(1-t)h+tq_-\in\tilde{\sigma}\cap\mathcal{U}$ for $0<t\leq 1$.

As in \eqref{eq:sec-6-eigen-equations}, the eigenvalue equation extends to every $k\in\mathcal{E}^+$:
\eq{
  B(f,k)=\beta A(f\odot k)
  \quad(k\in\mathcal{E}^+).
}
Using \eqref{eq:sec-10-signed-incident-derivative-agreement}, we apply the calculation in the proof of \autoref{lem:sec-6-rayleigh-stationarity} to $\widehat{V}_{\sigma}$ along $h+tz$ and to $\widehat{V}_{\tilde{\sigma}}$ along $h-tz$:
\eq{
  \left.\frac{d}{dt}\right|_{0^+}\mathcal{R}(h+tz)
  &=\beta\bigl\{B(f^{\odot2},a)-A(f^{\odot2}\odot a)\bigr\}
    -C_{\sigma}(f,f,a),\\
  \left.\frac{d}{dt}\right|_{0^+}\mathcal{R}(h-tz)
  &=-\beta\bigl\{B(f^{\odot2},a)-A(f^{\odot2}\odot a)\bigr\}
    +C_{\tilde{\sigma}}(f,f,a).
}

For $\xi=h\odot f$, Taylor expansion of the two chamber polynomials at $h$ leads to
\eq{
  D^2V(h+tz)[\xi,\xi]
  &=D^2V(h)[\xi,\xi]+tC_{\sigma}(f,f,a)+O(t^2),\\
  D^2V(h-tz)[\xi,\xi]
  &=D^2V(h)[\xi,\xi]-tC_{\tilde{\sigma}}(f,f,a)+O(t^2).
}
The midpoint inequality \eqref{eq:sec-10-signed-support-Hessian-midpoint} now implies $C_{\sigma}(f,f,a)\geq C_{\tilde{\sigma}}(f,f,a)$. Hence
\eq{
  \left.\frac{d}{dt}\right|_{0^+}\mathcal{R}(h+tz)
  +\left.\frac{d}{dt}\right|_{0^+}\mathcal{R}(h-tz)
  &=C_{\tilde{\sigma}}(f,f,a)-C_{\sigma}(f,f,a)
  \leq 0.
}
Both derivatives are nonnegative. Consequently, both vanish.

Thus, for every $q\in\mathcal{U}\cap\sigma\cap\mathcal{E}^+$ satisfying \eqref{eq:sec-10-reflected-support-slacks} and $a_i=(q_i-h_i)/h_i$,
\eq{
  \beta\bigl\{B(f^{\odot2},a)-A(f^{\odot2}\odot a)\bigr\}
  -C_{\sigma}(f,f,a)=0.
}
The map $q\mapsto((q_i-h_i)/h_i)_{i\in I}$ is an affine isomorphism of $\mathcal{E}^+$ which maps the restricted set of support vectors onto a nonempty relatively open subset of $\mathcal{E}^+$. Since the left-hand side is linear in $a$ and vanishes on this image,
\eq{
  C_{\sigma}(f,f,g)=\beta\bigl\{B(f^{\odot2},g)-A(f^{\odot2}\odot g)\bigr\}\quad (g\in\mathcal{E}^+).
}

Finally, let $g\in\mathcal{E}^+$. By \autoref{cor:sec-10-symmetric-chamber-selection}, there is a symmetric simple type chamber $\sigma_g$ incident to $h$ such that $h_t\in\overline{\sigma_g}\cap\mathcal{U}$ for all sufficiently small $t\geq 0$. Applying \eqref{eq:sec-10-signed-incident-derivative-agreement} at $h_t$, the calculation in the proof of \autoref{lem:sec-6-rayleigh-stationarity}, and \eqref{eq:sec-10-chamber-stationarity} for $\sigma_g$, the claim follows.
\end{proof}

\subsection{Second variation in an incident chamber}
\label{sec:incident-chamber-second-variation}

In this subsection, we work in a simple type chamber incident to a possibly nonsimple local minimizer. We approach the minimizing support vector through positive even vectors in the chamber and pass the degree-two Hodge--Riemann inequality to the limit. The limiting fourth-derivative bound, together with chamberwise stationarity, yields the required estimate for the second derivative of the trial Rayleigh quotient.

\begin{lemma}\label{lem:sec-10-incident-fourth-derivative-lower-bound}
Assume that $n\geq 3$. There is a unique $\eta\in\mathcal{E}^+$ such that, for every symmetric simple type chamber $\sigma$ incident to $h$,
\eq{\label{eq:sec-10-incident-eta}
  C_{\sigma}(f,f,g)=(n-2)B(\eta,g)
  \quad(g\in\mathcal{E}^+).
}
For every such $\sigma$,
\eq{\label{eq:sec-10-incident-fourth-bound}
  D_{\sigma}(f,f,f,f)
  \geq (n-2)(n-3)B(\eta,\eta).
}
\end{lemma}

\begin{proof}
The proof of \autoref{cor:sec-8-irredundant-strict-even-eigenvalue-bound} shows that $B$ is negative definite on $\mathcal{E}_0^+$. The argument following \eqref{eq:sec-7-eta-definition} then shows that $B:\mathcal{E}^+\to(\mathcal{E}^+)^*$ is an isomorphism. By \eqref{eq:sec-10-chamber-stationarity}, the functional $C_{\sigma}(f,f,\mathord{\cdot})$ on $\mathcal{E}^+$ is independent of $\sigma$. Hence a unique $\eta\in\mathcal{E}^+$ satisfies \eqref{eq:sec-10-incident-eta} for every symmetric simple type chamber $\sigma$ incident to $h$.

When $n=3$, both sides of \eqref{eq:sec-10-incident-fourth-bound} are zero. Assume now that $n\geq 4$.

Fix such a chamber $\sigma$ and choose $q\in\mathcal{U}\cap\sigma\cap\mathcal{E}^+$, which is possible by \eqref{eq:sec-8-even-incident-neighborhood}. Since $h\in\overline{\sigma}$ and $\sigma$ is convex,
\eq{
  q_s=(1-s)h+sq
  \in\sigma\cap\mathcal{E}^+\cap(0,\infty)^I
  \quad(0<s\leq 1)
}
and $q_s\to h$ as $s\to 0^+$.

Let $B_{\sigma}^s$, $C_{\sigma}^s$, and $D_{\sigma}^s$ denote the logarithmic support forms of $\widehat{V}_{\sigma}$ at $q_s$. Continuity of the derivatives of $\widehat{V}_{\sigma}$ and \eqref{eq:sec-10-signed-incident-derivative-agreement} show that, as $s\to 0^+$,
\eq{
  B_{\sigma}^s&\to B,\\
  C_{\sigma}^s(f,f,\mathord{\cdot})&\to C_{\sigma}(f,f,\mathord{\cdot}),\\
  D_{\sigma}^s(f,f,f,f)&\to D_{\sigma}(f,f,f,f).
}
For all sufficiently small $s>0$, the map $B_{\sigma}^s:\mathcal{E}^+\to(\mathcal{E}^+)^*$ is also an isomorphism. Let $\eta^s\in\mathcal{E}^+$ be the unique vector satisfying
\eq{\label{eq:sec-10-incident-eta-approximation}
  C_{\sigma}^s(f,f,g)=(n-2)B_{\sigma}^s(\eta^s,g)
  \quad(g\in\mathcal{E}^+).
}
Viewing $B$ and $B_{\sigma}^s$ as maps from $\mathcal{E}^+$ to $(\mathcal{E}^+)^*$, we have
\eq{
  \eta
  =\frac{1}{n-2}B^{-1}\bigl(C_{\sigma}(f,f,\mathord{\cdot})\bigr),
  \quad
  \eta^s
  =\frac{1}{n-2}(B_{\sigma}^s)^{-1}
  \bigl(C_{\sigma}^s(f,f,\mathord{\cdot})\bigr).
}
Therefore $\eta^s\to\eta$.

By \autoref{lem:sec-3-antipodal-volume-parity}, both sides of \eqref{eq:sec-10-incident-eta-approximation} vanish for odd $g$. Thus this equation holds for every facet function $g$. The construction in \eqref{eq:sec-7-Gamma-definition} and the argument following it apply to $\widehat{V}_{\sigma}$ at $q_s$, and the corresponding class is primitive. The argument of \autoref{lem:sec-7-fourth-derivative-lower-bound} then establishes
\eq{
  D_{\sigma}^s(f,f,f,f)
  \geq (n-2)(n-3)B_{\sigma}^s(\eta^s,\eta^s).
}
Letting $s\to 0^+$ proves the second claim.
\end{proof}

\begin{corollary}\label{cor:sec-10-incident-trial-Rayleigh-second-variation}
Assume that $n\geq 3$. For $t\geq 0$ sufficiently small, set
\eq{
  h_t=h\odot(\one+tf),\quad
  u_t=f-(1+\beta)t f^{\odot2},
}
and define
\eq{
  \mathcal{R}_{\mathrm{tr}}(t)
  =\frac{\mathsf{L}_{h_t}(u_t,u_t)}
  {\widetilde{\mathsf{M}}_{h_t}(u_t,u_t)}.
}
Then $\mathcal{R}_{\mathrm{tr}}(0)=\lambda$, and its right derivatives satisfy
\eq{\label{eq:sec-10-incident-branch-second-variation}
  \mathcal{R}_{\mathrm{tr}}'(0^+)=0,\quad
  \mathcal{R}_{\mathrm{tr}}''(0^+)
  \leq \frac{2\beta\lambda}{S(n-1)}
  \bigl(1+(n-1)\beta\bigr).
}
Moreover, if $-1/(n-1)<\beta<0$, then
\eq{
  \mathcal{R}_{\mathrm{tr}}(t)<\lambda
  \quad\text{for all sufficiently small }t>0.
}
\end{corollary}

\begin{proof}
In view of \autoref{cor:sec-10-symmetric-chamber-selection}, we may choose a symmetric simple type chamber $\sigma$ incident to $h$ such that $h_t\in\overline{\sigma}\cap\mathcal{U}$ for all sufficiently small $t\geq 0$. By \eqref{eq:sec-10-signed-incident-derivative-agreement} at $h_t$, the quotient $\mathcal{R}_{\mathrm{tr}}(t)$ agrees on this interval with the quotient determined by $\widehat{V}_{\sigma}$. Since $\widehat{V}_{\sigma}$ is a polynomial and $\widetilde{\mathsf{M}}_h(f,f)=1$, the chamber quotient provides a real-analytic extension of $\mathcal{R}_{\mathrm{tr}}$ across zero.

By \eqref{eq:sec-10-chamber-stationarity}, the calculation leading to \eqref{eq:sec-6-exact-second-derivative} applies here with $d_4=D_{\sigma}(f,f,f,f)$. Hence $\mathcal{R}_{\mathrm{tr}}'(0^+)=0$ and
\eq{\label{eq:sec-10-exact-incident-trial-second-derivative}
  \mathcal{R}_{\mathrm{tr}}''(0^+)
  =-\beta(2\beta-1)(\beta+2)A(f^{\odot4})
  +3\beta^2B(f^{\odot2},f^{\odot2})
  -D_{\sigma}(f,f,f,f).
}
Using \eqref{eq:sec-10-chamber-stationarity} and \autoref{lem:sec-10-incident-fourth-derivative-lower-bound}, the calculation in the proof of \autoref{lem:sec-7-trial-Rayleigh-second-variation-bound} yields
\eq{
  \mathcal{R}_{\mathrm{tr}}''(0^+)
  \leq\frac{2\beta\lambda}{S(n-1)}
  \bigl(1+(n-1)\beta\bigr).
}

If $-1/(n-1)<\beta<0$, then \eqref{eq:sec-10-incident-branch-second-variation} implies $\mathcal{R}_{\mathrm{tr}}''(0^+)<0$. Since $\mathcal{R}_{\mathrm{tr}}(0)=\lambda$ and $\mathcal{R}_{\mathrm{tr}}'(0^+)=0$, Taylor's formula now shows that
\eq{
  \mathcal{R}_{\mathrm{tr}}(t)-\lambda
  =\frac{1}{2}\mathcal{R}_{\mathrm{tr}}''(0^+)t^2+o(t^2)<0
}
for all sufficiently small $t>0$.
\end{proof}

\begin{proof}[Proof of \autoref{thm:main}]
By \autoref{cor:sec-8-irredundant-strict-even-eigenvalue-bound}, $\beta<0$. If the estimate failed, then $-\frac{1}{n-1}<\beta<0$. For all sufficiently small $t>0$, \autoref{cor:sec-10-incident-trial-Rayleigh-second-variation}, the Rayleigh principle, and local minimality imply $\lambda\leq\lambda_{1,e}(P(h_t))\leq\mathcal{R}_{\mathrm{tr}}(t)<\lambda$, a contradiction.
\end{proof}

\section*{Acknowledgment}
Hu was supported by the National Key Research and Development Program of China (Grant No. 2021YFA1001800). Ivaki was supported by the Austrian Science Fund (FWF) under Project P36545.

\section*{AI disclosure}
The authors used AI tools during the development of this paper.

\vspace{5mm}

\noindent
\begin{minipage}[t]{0.40\textwidth}
\textsc{School of Mathematical\\ Sciences, Beihang University,\\ Beijing 100191, China}\\
\email{\href{mailto:huyingxiang@buaa.edu.cn}{huyingxiang@buaa.edu.cn}}
\end{minipage}
\hfill
\begin{minipage}[t]{0.40\textwidth}
\textsc{Institut f\"{u}r Diskrete\\ Mathematik und Geometrie,\\ Technische Universit\"{a}t Wien, Wiedner Hauptstra{\ss}e 8--10,\\ 1040 Wien, Austria}\\
\email{\href{mailto:mohammad.ivaki@tuwien.ac.at}{mohammad.ivaki@tuwien.ac.at}}
\end{minipage}

\end{document}